\RequirePackage{fix-cm}
\documentclass{svjour3}                     
\smartqed  
\usepackage{graphicx}
\usepackage{amssymb,amsmath}
\usepackage{graphicx,float}

\usepackage[hypertexnames=false,colorlinks]{hyperref}
\hypersetup{
    colorlinks=true,
    linkcolor=blue
}
\usepackage{booktabs}
\usepackage{subcaption}
\usepackage{fancyhdr}
\usepackage{lastpage}
\usepackage{geometry}
\usepackage{algorithm}
\usepackage{algpseudocode}
\usepackage{ragged2e}
\usepackage{makecell}
\usepackage{xurl}

\floatname{algorithm}{\sc Algorithm}
\DeclareMathOperator*{\argmin}{arg\,min}

\newtheorem{thm}{Theorem}

\newtheorem{lem}{Lemma}
\newtheorem{prop}{Proposition}

\newtheorem{rmk}{Remark}

\newtheorem{assum}{Assumption}

\begin{document}

\title{Accelerated Plug-and-Play Davis–Yin Splitting for Nonconvex Image Reconstruction 
}


\author{Kuntal Roy         \and
        Pankaj Gautam 
}


\institute{Kuntal Roy \at
              Department of Applied Mathematics \& Scientific Computing \\
              Indian Institute of Technology Roorkee\\
              \email{kuntal\_r@amsc.iitr.ac.in}           
           \and
           Pankaj Gautam \at
             Department of Applied Mathematics \& Scientific Computing \\
              Indian Institute of Technology Roorkee\\
              \email{pankaj.gautam@amsc.iitr.ac.in}   
}

\date{Received: date / Accepted: date}

\maketitle

\begin{abstract}
In this work, we study a class of structured non-convex and non-smooth optimization problems arising in imaging applications, where the objective is the sum of three functions. We consider the Davis–Yin splitting method, a FISTA-type accelerated variant, along with a quasi-Newton line-search method and a plug-and-play (PnP) extension to solve the problem. We develop a unified convergence analysis based on the Davis–Yin envelope under mild assumptions, including the Kurdyka–\L{}ojasiewicz property. Within this framework, we establish subsequential and global convergence of the iterates to stationary points, together with residual convergence rates. We demonstrate the performance of the proposed methods on image restoration and low-rank matrix completion problems. For low-rank matrix completion, we perform experiments on both synthetic data and public datasets to evaluate recovery accuracy and efficiency. In imaging tasks, including image deblurring, we compare the convergence behaviour and reconstruction quality of several algorithmic variants using peak signal-to-noise ratio (PSNR). The results show that FISTA-acceleration and line-search methods improve convergence, while PnP denoisers enhance image quality, and the proposed methods remain competitive on matrix completion tasks.
\end{abstract}

\numberwithin{algorithm}{section}
\setcounter{algorithm}{0}

\section{Introduction}
Many applications in machine learning, signal processing, image restoration, and related fields can be formulated as an optimization problem \cite{jain2017non}. In many cases, these optimization problems can be reformulated as the sum of multiple components. In this paper, we consider a non-convex optimization problem, modeled as:

\begin{equation}\label{Prob 1}
    \min_{x\in\mathbb{R}^n}\varphi(x)=f(x)+g(x)+h(x),
\end{equation}
where all of $f, g$ and $h$ are possibly non-convex, and $g$ and $h$ are $L_g,~L_h$-smooth, respectively.

Splitting algorithms are very popular for solving structured optimization problems because, in these algorithms, we split the objective function into a sum of multiple functions and then tackle the problem, which is very helpful. To solve the non-convex problem \eqref{Prob 1}, the Davis-Yin splitting (DYS) algorithm \ref{DYS} have been studied in \cite{davis2017three,pedregosa2018adaptive,liu2019envelope,bian2021three,wu2024extrapolated}. 

\begin{algorithm}[h]
\caption{Davis-Yin Splitting}
\begin{algorithmic}[1]

\State \textbf{Input:} $x^0 \in \mathbb{R}^n$, step size $\gamma > 0$, relaxation parameter $\lambda>0$.

\For{$k = 0,1,2,\dots$}
    \State $y^k = \mathrm{prox}_{\gamma g}(x^k)$
    \State $z^k \in \mathrm{prox}_{\gamma f}\big(2y^k - x^k - \gamma \nabla h(y^k)\big)$
    \State $x^{k+1} = x^k + \lambda(z^k - y^k)$
\EndFor

\end{algorithmic}\label{DYS}
\end{algorithm}

Forward-backward splitting (FBS) \cite{combettes2011proximal,beck2017first} and Douglas-Rachford splitting (DRS) \cite{douglas1956numerical,lions1979splitting,themelis2020douglas} algorithms are used for optimization problems involving two functions, while the DYS algorithm extends these methods for optimization problems involving three functions. If we choose $g\equiv0$, then Algorithm \ref{DYS} turns into FBS, and if $h\equiv0$, then Algorithm \ref{DYS} turns into DRS. 

Although the DYS algorithm is very useful for the sum of three function type structured optimization problems, but for large-scale optimization problems, it will be very slow. In image restoration, matrix recovery, and related problems, this algorithm requires too many iterations to achieve satisfactory results, increasing the computational cost. Therefore, it is important to increase the convergence speed of this algorithm without changing its structure. For this, the extrapolated-DYS \cite{wu2024extrapolated} algorithm is proposed. In this algorithm, an extrapolated term is used to utilize the information from the previous iteration in the current iteration, which helps the algorithm converge faster. The convergence rate of the sequence generated by the extrapolated-DYS algorithm is $\mathcal{O}\left(\frac{1}{\sqrt{k}}\right)$. The fast iterative shrinkage-thresholding algorithm (FISTA) \cite{beck2009fast} is a special kind of extrapolation algorithm, where we use the sequence of the extrapolation parameter $\{t^k\}$ as $t^{k+1}=\frac{1+\sqrt{1+4(t^k)^2}}{2}$, where $t^0=1$. In the FISTA step, we include the information from the previous iterations in the current iterations with the sequence of extrapolation parameters $\{t^k\}$, which improves the convergence behavior and leads to faster practical performance. So, to further improve the convergence speed and convergence rate, we use a FISTA acceleration term in the DYS algorithm, and we call it the FISTA-type DYS algorithm. For this algorithm, we get the convergence rate of the generated sequence is $o\left(\frac{1}{\sqrt{k}}\right)$.

Now, the FISTA-type DYS algorithm is still an accelerated first order splitting algorithm; therefore, it may be slow for some large-scale problems. So, for further improvement, we use a quasi-Newton direction \cite{dennis1974characterization,dennis1977quasi,nocedal1980updating,byrd1989tool,ip1992local,stella2017forward,themelis2018forward,themelis2022douglas} with the modified Broyden \cite{themelis2018forward,themelis2022douglas} update rule and Continuous-Lyapunov Descent (CLyD) based line-search framework \cite{themelis2018,themelis2022douglas} in our FISTA-type DYS algorithm. In \cite{nocedal2006numerical}, the quasi-Newton method is basically used for smooth optimization. In this method, we try to find the point where the gradient of the main objective function is zero. Later, in \cite{themelis2022douglas}, the quasi-Newton method is applied in DRS and alternating direction method of multipliers (ADMM) algorithms, but to tackle the non-smoothness of the objective function, the main motive has been changed to find the point where the residual is zero. We use this new methodology of the quasi-Newton method in the FISTA-type DYS algorithm. Here, the CLyD-based line-search framework is used to avoid the requirements of differentiability of the Davis-Yin envelope (DYE). In our work, we see that the practical performance of this FISTA-type DYS with a line-search algorithm is better than the FISTA-type DYS algorithm.

Another important development in our work is the plug-and-play (PnP) framework \cite{venkatakrishnan2013plug,gavaskar2021plug,wei2022tfpnp,wu2023retinex}, which is very useful in image restoration problems. In \cite{venkatakrishnan2013plug}, the PnP framework was introduced. In recent years, the PnP framework has been successfully used in many splitting algorithms such as PnP-FBS, PnP-DRS, PnP-ADMM, PnP-primal dual, extrapolated PnP-DYS \cite{ono2017primal,ryu2019plug,hurault2022proximal,ebner2024plug,nair2024convergent,hurault2024convergent,wu2024extrapolated}, etc. In these proximal algorithms the proximal operator has been replaced with a denoiser to include prior knowledge, which improves the reconstruction quality. In \cite{ryu2019plug}, the non-expensiveness issue of denoisers is addressed through spectral normalization in each layer. Then, in \cite{hurault2022gradient}, the gradient-step (GS) denoiser is used to tackle this non-expansiveness problem, and in \cite{hurault2022proximal}, the GS denoiser is related to the proximal operator of a specific non-convex functional. Later, in \cite{wu2024extrapolated}, this GS denoiser is used in the PnP framework for the non-convex optimization problem, where the objective function is written as a sum of three functions, and great success has been achieved.  Motivated by the success of these methods, we use the PnP framework in our proposed algorithms, FISTA-type DYS and FISTA-type DYS with line search.

Although several variants of DYS have been used for structured optimization problems in recent years to improve convergence speed while preserving the simple splitting structure of Algorithm \ref{DYS}, there is still a need to improve convergence speed in large-scale optimization problems like low-rank matrix recovery problems, image restoration problems, etc. To solve this issue, we add an FISTA-type acceleration term to the DYS algorithm. For further improvement, we add a quasi-Newton direction with modified Broyden and CLyD-based line-search in the FISTA-type DYS algorithm. On the other hand, the success of the PnP framework in splitting algorithms for image restoration problems in recent years motivates us to add the PnP framework to our proposed algorithms. Now, the main contributions to this paper are given below:

\subsection*{Our contributions}
\begin{itemize}
    \item[$\bullet$] We propose a FISTA-type Davis-Yin splitting algorithm (Algorithm \ref{FISTA-DYS}), an accelerated variant of the DYS algorithm, for the non-convex optimization problem \eqref{Prob 1} and study the convergence analysis of the proposed algorithm.
    \item[$\bullet$] We discuss the subsequential convergence. Here, we also establish the convergence rate for the generated sequence from Algorithm \ref{FISTA-DYS}, which is $o\left(\frac{1}{\sqrt{k}}\right)$, and using the Polyak-\L{}ojasiewicz (PL) condition, we establish the convergence rate of the objective function, which is $o\left(\frac{1}{k}\right)$. We also establish global convergence for Algorithm \ref{FISTA-DYS}, using the Kurdyka-\L{}ojasiewicz (KL) property.
    \item[$\bullet$] For further improvement in convergence speed, we use the quasi-Newton method with the modified Broyden update rule and the CLyD-based line-search together in Algorithm \ref{FISTA-DYS}, and we propose Algorithm \ref{FISTA-DYS-LS}.
    \item[$\bullet$] We use the PnP framework in our proposed algorithms by replacing a proximal operator with the gradient-step denoiser. 
    \item[$\bullet$] By numerical experiments in structured quadratic-type matrix optimization problems, low-rank matrix recovery problems, and image restoration problems, we show the efficiency of our proposed algorithms in converging very quickly.
\end{itemize}

Before proceeding to our main work, we discuss the outline of the structure of our paper. In Section \ref{Preliminaries}, we introduce some important notations, definitions, and results that are used in our work. In Section \ref{Davis-Yin Envelope}, we define the DYE and discuss the properties of DYE. In Section \ref{FISTA-type DYS section}, we propose our first accelerated algorithm, FISTA-type DYS, and establish the convergence results (subsequential, global) for this proposed algorithm. In Section \ref{FISTA-type DYS with Line-search section}, we propose another accelerated algorithm, FISTA-type DYS with line-search, which is based on the FISTA-type DYS algorithm, quasi-Newton direction with modified Broyden update rule, and CLyD-based line-search, and establish the convergence results (subsequential, superlinear) for this new algorithm. We have extended our proposed algorithms with the PnP framework in Sections \ref{FISTA-type PnP-DYS for Image Restoration} and \ref{FISTA-type PnP-DYS with Line-search for Image Restoration}. In Section \ref{Numerical Experiment}, we give some practical experiments and compare our proposed algorithms with some proposed algorithms (i.e., extrapolated-DYS, DRFDR, extrapolated PnP-DYS, etc.).

\section{Preliminaries}\label{Preliminaries}
\subsection{Notations}
We denote the extended real line by $\overline{\mathbb{R}}:= \mathbb{R} \cup \{+\infty\}$, the set of positive real numbers by $\mathbb{R}_{++}$, and the set of non-negative real numbers by $\mathbb{R}_+:=\mathbb{R}_{++}\cup\{0\}$. We also denote the class of all continuously differentiable functions on $\mathbb{R}^n$, whose gradients are Lipschitz, by $C^{1,1}(\mathbb{R}^n)$. For any $r \in \mathbb{R}$, its positive and negative parts are defined as $[r]^+ := \max\{0, r\}$ and $[r]^- := \max\{0, -r\}$, respectively, so that $r = [r]^+ - [r]^-$ holds. The open and closed balls centered at $x$ with radius $r$ are denoted by $B(x; r)$ and $\overline{B}(x; r)$, respectively. The identity mapping is denoted by $I$, i.e., $I(x) = x$. Given a sequence $\{x^k\}\subset\mathbb{R}^n$, which means that $x^k \in \mathbb{R}^n$ for all $k \in \mathbb{N}$. A sequence $\{x^k\} \subset \mathbb{R}^n$ is said to be summable if $\sum_{k \in \mathbb{N}} \|x^k\| < +\infty$, and square-summable if $\sum_{k \in \mathbb{N}} \|x^k\|^2 < +\infty$ and if $\sum_{k \in \mathbb{N}} \|x^k\| < +\infty$ then $\lim_{k\to+\infty}\|x^k\|=0.$ Let $f : \mathbb{R}^n \to \overline{\mathbb{R}}$ then its domain and epigraph are defined as $\mathrm{dom}\,f := \{x \in \mathbb{R}^n \mid f(x) < \infty\},\;\mathrm{epi}\,f := \{(x,\alpha) \in \mathbb{R}^n \times \mathbb{R} \mid f(x) \le \alpha\}.$ For $\alpha \in \mathbb{R}$, define the level set $\mathrm{lev}_{\leq \alpha} f := \{x \in \mathbb{R}^n \mid f(x) \le \alpha\}.$ The function $f$ is level-bounded if all such level sets are bounded. The regular subdifferential $\hat{\partial}f$ is defined by
$$v \in \hat{\partial}f(\bar{x}) \quad\Longleftrightarrow\quad\liminf_{\substack{x \to \bar{x} \\ x \neq \bar{x}}}\frac{f(x) - f(\bar{x}) - \langle v, x - \bar{x} \rangle}{\|x - \bar{x}\|} \ge 0.$$
Let $\{x^k\}\subseteq\mathbb{R}^n$ and $\{y^k\}\subseteq\mathbb{R}^n$ be two sequences and $\|y^k\|\ne0$ for any sufficiently large $k$. Then we say $x^k=o(y^k)$ as $k\to+\infty$ if $\lim_{k\to+\infty}\frac{\|x^k\|}{\|y^k\|}=0$ and we call $\{x^k\}$ converge to $x^{\star}\in\mathbb{R}^n$ with superlinear rate if $\lim_{k\to+\infty}\frac{\|x^{k+1}-x^{\star}\|}{\|x^k-x^{\star}\|}=0.$ Let $X\in\mathbb{R}^{m\times n}$ be a matrix; then the Frobenius norm is denoted by $\|X\|_F$, and defined by $\|X\|_F=\left(\sum_{i=1}^m\sum_{j=1}^n X_{ij}^2\right)^{\frac{1}{2}}$, where $X_{ij}$ denotes the $(i,j)$-th element of $X$. For a set $S$, the indicator function $\delta_{S}:\mathbb{R}^n\to\overline{\mathbb{R}}$ is defined by
\begin{align*}
    \delta_{S}(x)&=\begin{cases}
        0\quad\;\;, \text{ if $x\in S$},\\
        +\infty\;, \text{ if $x\notin S$}.
    \end{cases}
\end{align*}

\subsection{Backgrounds}
\begin{definition}[\cite{beck2017first}]
    A function $f:\mathbb{R}^n\to[-\infty,+\infty]$ is called 
    \begin{enumerate}
        \item[(i)] proper if $f(x)>-\infty,$ for any $x\in\mathbb{R}^n$ and there exist at least one $x\in\mathbb{R}^n$ such that $f(x)<+\infty;$
        \item[(ii)] lower semicontinuous at $x\in \mathbb{R}^n$ if $f(x)\le\liminf_{n\to+\infty}f(x^k)$ for any sequence $\{x^k\}$ from $\mathbb{R}^n$, where $\lim_{n\to+\infty}x^k=x.$ If $f$ is lower semicontinuous at each point in $\mathbb{R}^n$ then it is called a lower semicontinuous function over $\mathbb{R}^n;$
        \item[(iii)] closed if the set $\{(x,~y):f(x)\le y,~x\in\mathbb{R}^n,~y\in\mathbb{R}\}$ is closed.
  \end{enumerate}
\end{definition}
\begin{thm}[Fermat's optimality condition \cite{beck2017first}]
    Let $f:\mathbb{R}^n\to]-\infty,+\infty]$ be proper. Then $\argmin f = \text{zer } \hat{\partial} f = \{x\in\mathbb{R}^n\;|\;0\in\hat{\partial} f(x)\}.$
\end{thm}
\begin{definition}[\cite{beck2017first}]
    Let $L\ge0$. A function $f:\mathbb{R}^n\to(-\infty,+\infty]$ is called as $L$-smooth over a set $D\subseteq\mathbb{R}^n$ if it is differentiable and $\nabla f$ is Lipschitz continuous over $D$, and $L$ is called the smoothness parameter.
\end{definition}
\begin{lem}[\cite{bauschke2011convex}]
    Let $D\subseteq\mathbb{R}^n$ and $f:\mathbb{R}^n\to(-\infty,+\infty]$ be $L$-smooth $(L\ge0)$ over $D$. Then for any $x,~y\in D,$
    $$|f(y)-f(x)-\langle\nabla f(x),~y-x\rangle|\le\frac{L}{2}\|y-x\|^2.$$
\end{lem}
\begin{definition}[\cite{wang2010chebyshev}]
    \justifying
    If a proper, lower semicontinuous function $f:\mathbb{R}^n\to(-\infty,+\infty]$ is $\sigma$-hypoconvex if for all $x,\; y\in \mathbb{R}^n$, $\lambda\in(0,\;1)$, 
    $$f\left((1-\lambda)x+\lambda y\right)\le(1-\lambda)f(x) + \lambda f(y) + \frac{\sigma}{2}\lambda(1-\lambda)\|x-y\|^2.$$ 
\end{definition}
\begin{rmk}[\cite{themelis2020douglas}]
    A proper, lower semicontinuous function $f:\mathbb{R}^n\to(-\infty,+\infty]$ is $L$- smooth, where $L\ge0$ then $\exists\; \sigma\in[-L,\;L]$ such that $f$ is $\sigma$-hypoconvex and satisfy the following inequality $$\frac{\sigma}{2}\|y-x\|^2\le f(y)-f(x)-\langle\nabla f(x), y-x\rangle\le \frac{L}{2}\|y-x\|^2.$$
\end{rmk}

\begin{definition}[\cite{beck2017first}]
    Let $f:\mathbb{R}^n\to(-\infty,+\infty]$ be a proper and closed function, and $\gamma>0$. Then the Moreau envelope and the proximal mapping of $f$ with respect to $\gamma$ are defined as 
    $f^{\gamma}(x)=\min_{u\in\mathbb{R}^n}\left\{f(u)+\frac{1}{2\gamma}\|x-u\|^2\right\}$ and $\mathrm{prox}_{\gamma f}(x)=\argmin_{u\in\mathbb{R}^n}\Big\{f(u)+\frac{1}{2\gamma}\|x-u\|^2\Big\}$, respectively.
\end{definition}

\begin{definition}[\cite{themelis2018forward}]
    A function $f:\mathbb{R}^n\to\overline{\mathbb{R}}$ is called prox-bounded if $\exists$ a positive $\gamma$ such that $g(x)+\frac{1}{2\gamma}\|x\|^2$ is bounded below for all $x\in\mathbb{R}^n$ and the threshold value of prox-boundedness of $f$ is defined by  $\gamma_f=\sup\left\{\gamma>0:g+\frac{1}{2\gamma}\|\cdot\|^2 \text{ is bounded below on } \mathbb{R}^n\right\}$.
\end{definition}

\begin{prop}[\cite{themelis2020douglas}]\label{Proximal properties of smooth functions}
Let $f \in C^{1,1}(\mathrm{dom} f)$ be an $L_f$-smooth and $\sigma_f$-hypocon- vex function for some $\sigma_f \in [-L_f, L_f]$. Then, its proximal mapping $\mathrm{prox}_{\gamma f}$ with $\gamma \geq \frac{1}{|\sigma_f|}$ is well-defined, and the following properties hold:

\begin{enumerate}
    \item[(i)] The mapping $\mathrm{prox}_{\gamma f}$ is single-valued. Moreover, for all $x \in \mathbb{R}^n$, it holds that
    $$y = \mathrm{prox}_{\gamma f}(x) \quad \Longleftrightarrow \quad x = y + \gamma \nabla f(y).$$
    \item[(ii)] The operator $\mathrm{prox}_{\gamma f}$ is $\frac{1}{1+\gamma L_f}$-strongly monotone and $(1+\gamma \sigma_f)$-cocoercive, i.e., for all $x, \tilde{x} \in \mathbb{R}^n$,
    $$\langle y - \tilde{y}, x - \tilde{x} \rangle \geq \frac{1}{1+\gamma L_f} \|x - \tilde{x}\|^2,\quad\text{and}\quad\langle y - \tilde{y}, x - \tilde{x} \rangle \geq (1+\gamma \sigma_f)\|y - \tilde{y}\|^2,$$
    where $y = \mathrm{prox}_{\gamma f}(x)$ and $\tilde{y} = \mathrm{prox}_{\gamma f}(\tilde{x})$.

    In particular,
    $$\frac{1}{1+\gamma L_f} \|x - \tilde{x}\|\leq \|y - \tilde{y}\|\leq \frac{1}{1+\gamma \sigma_f} \|x - \tilde{x}\|.$$
    Consequently, $\mathrm{prox}_{\gamma f}$ is $\frac{1}{1+\gamma \sigma_f}$-Lipschitz continuous and invertible, and its inverse $I + \gamma \nabla f$ is $(1+\gamma L_f)$-Lipschitz continuous.
    \item[(iii)] The Moreau envelope $f^\gamma \in C^{1,1}(\mathbb{R}^n)$ is $L_{f^\gamma}$-smooth and $\gamma_{f^\gamma}$-hypoconvex, where $\quad\gamma_{f^\gamma} = \frac{\sigma_f}{1+\gamma \sigma_f}$, $L_{f^\gamma} = \max\left\{ \frac{L_f}{1+\gamma L_f}, \frac{1}{\gamma} \right\}$. Moreover, $\nabla f^\gamma(x) = \frac{1}{\gamma}(x - \mathrm{prox}_{\gamma f}(x)),\;\nabla f(\mathrm{prox}_{\gamma f}(x)) = \frac{1}{\gamma}(x - \mathrm{prox}_{\gamma f}(x)).$
\end{enumerate}
\end{prop}
\begin{definition}[\cite{themelis2018forward}]
    Let $f:\mathbb{R}^n\to\mathbb{R}$ be a real-valued function. Then $f$ is called
    \begin{enumerate}
        \item[(i)] strictly continuous at $\tilde{x}$, if $\limsup_{\substack{x,y\to\tilde{x},\\x\ne y}}\frac{|f(x)-f(y)|}{\|x-y\|}<+\infty$;
        \item[(ii)] strictly differentiable at $\tilde{x}$ with $\nabla f(\tilde{x})$, if $\lim_{\substack{x,y\to\tilde{x},\\x\ne y}}\frac{f(x)-f(y)-\langle\nabla f(\tilde{x}),x-y\rangle}{\|x-y\|}=0$.
    \end{enumerate}
\end{definition}

\begin{lem}[\cite{Dong2015comments}]\label{min order lemma}
    Let $\{x^n\},\;\{y^n\}\in\mathbb{R}_{++}$ and $\sum_{n\in\mathbb{N}}x^ny^n<+\infty$. Now assume that, $\sum_{n\in\mathbb{N}}x^n$ is not finite and $\{y^n\}$ is a decreasing sequence. Then $y^n=o\left(\frac{1}{\sum_{k=1}^nx^k}\right).$
\end{lem}


\begin{definition}[\cite{hurault2022proximal}]
    A proper, lower semicontinuous function $f$ is said to be $\beta$-weakly convex, for some $\beta \in \mathbb{R}$, if $f(x)+\frac{\beta}{2}\|x\|^2$ is convex.
\end{definition}

\begin{definition}[\cite{kong2025linear}]
    We say that, $f$ satisfies the Polyak-\L{}ojasiewicz (PL) inequality with parameter $\mu>0$ if $$\mathrm{dist}^2(0,\hat{\partial} f(x)) \ge 2\mu\bigl(f(x)-f^*\bigr),\quad \forall x\in\mathrm{dom}f,$$ where $f^*=\min_{x\in \mathrm{dom}f}f(x).$
\end{definition}

\begin{definition}[\cite{li2016douglas}]
    Consider that, $f_1,~f_2,\cdots,f_p$ and $g_1,~g_2,\cdots,g_q$ are real polynomial functions then the set 
    $$S=\left\{x\in\mathbb{R}^n:f_1(x)=f_2(x)=\cdots=f_p(x),~g_1(x)<0,~g_2(x)<0,\cdots, g_q(x)<0\right\}$$ 
    is called semialgebraic set and a function $f:\mathbb{R}^n\to\mathbb{R}$, then $f$ is called a semialgebraic function if the set $\left\{(x,~f(x))\in\mathbb{R}^{n+1}:x\in\mathbb{R}^n\right\}$ is semialgebraic set.
\end{definition}
\begin{definition}[\cite{bolte2014proximal}]
    Suppose that $f:\mathbb{R}^n\to(-\infty,+\infty]$ is a proper and lower semicontinuous function. Then, we say $f$ satisfies Kurdyka-\L{}ojasiewicz (KL) property at $\tilde{x}\in\mathrm{dom}(\hat{\partial}f)$ if $\exists~\xi>0$, a neighborhood $\mathcal{V}$ of $\tilde{x}$, and a continuous and concave function $\psi:[0,~\xi)\to[0,+\infty)$ such that
    \begin{enumerate}
        \item[(i)] $\psi(0)=0$ and $\psi$ is continuously differentiable on $(0,~\xi)$, where $\psi'>0.$
        \item[(ii)] $\psi'(f(x)-f(\tilde{x}))~\mathrm{dist}(0,~\hat{\partial}f(x))\ge0,~\forall x\in\mathcal{V}\cup\{y\in\mathbb{R}^n:f(\tilde{x})<f(y)<f(\tilde{x})+\xi\}.$ This inequality is called the KL inequality.
    \end{enumerate}
    If $f$ satisfies the KL property at every point in $\mathrm{dom}(\hat{\partial}f)$, then $f$ is called a KL function.
\end{definition}
\begin{prop}[\cite{li2016douglas}]\label{Relation between semialgebraic and KL}
    Suppose that a proper and lower semicontinuous function $f$ is semialgebraic on $\mathbb{R}^n$. Then $f$ is KL function with $\psi(x)=\alpha x^{1-\eta}$, where $\alpha>0$ and $\eta\in[0,~1)$.
\end{prop}



\section{Davis-Yin Envelope}\label{Davis-Yin Envelope}
We define the Davis-Yin envelope for the sum of three non-convex functions as follows:
\begin{equation}\label{DYE_1}
    \varphi^{\gamma}(x) = \min_{u\in\mathbb{R}^n}\left\{ f(u)+g(y)+h(y) +\langle u-y, \nabla g(y)+\nabla h(y)\rangle + \frac{1}{2\gamma}\|u-y\|^2\right\},
\end{equation}
where $y=\mathrm{prox}_{\gamma g}(x).$

Now, by the Algorithm \ref{DYS}, we define
\begin{equation}\label{DYE}
    \varphi^{\gamma}(x) = f(z)+g(y)+h(y) +\langle z-y, \nabla g(y)+\nabla h(y)\rangle + \frac{1}{2\gamma}\|z-y\|^2,
\end{equation}
where $y=\mathrm{prox}_{\gamma g}(x).$

\begin{rmk}\label{FBE, DRE}
    \begin{enumerate}
        \item[(i)] If we consider that $g=0$ in (\ref{DYE_1}) and (\ref{DYE}) then DYE becomes forward backward envelope (FBE) \cite{themelis2018forward}.
        \item[(ii)] If we consider that $h=0$ in (\ref{DYE_1}) and (\ref{DYE}) then DYE becomes Douglas-Rachford envelope (DRE) \cite{themelis2020douglas}.
    \end{enumerate}    
\end{rmk}
\begin{assum}\label{Assumption 1}
    \leavevmode
    \begin{enumerate}
        \item[\textit{(i)}] $f:\mathbb{R}^n\to\overline{\mathbb{R}}$ is proper, closed, and $\beta_f$-weakly convex.
        \item[\textit{(ii)}] $g\in C^{1,1}(\mathbb{R}^n)$ is $L_g$-smooth, hence $\sigma_g$-hypoconvex for some $\sigma_g\in[-L_g,~L_g].$
        \item[\textit{(iii)}] $h\in C^{1,1}(\mathbb{R}^n)$ is $L_h$-smooth, hence $\sigma_h$-hypoconvex for some $\sigma_h\in[-L_h,~L_h].$
        \item[\textit{(iv)}] $\argmin \varphi \ne\emptyset.$
    \end{enumerate}
\end{assum}
\begin{rmk}
    Under Assumption \ref{Assumption 1}, $f$ and $g$ are prox-bounded if $\gamma<\frac{1}{L_g+L_h}$.
\end{rmk}
\subsection{Properties}
In this section, we discuss some properties of DYE, which are used in the convergence proof.

\begin{prop}[Strict continuity]\label{Strict continuity}
    Let us consider Assumption \ref{Assumption 1} to be valid. For any $\gamma<\frac{1}{L_g+L_h}$, the DYE $\varphi^{\gamma}$ is a strictly continuous real valued function.
\end{prop}
\begin{proof}
    Let $\tilde{g}=g+h$, then $\|\nabla \tilde{g}(x)-\nabla\tilde{g}(y)\|\le (L_g+L_h)\|x-y\|,$ $\forall x,~y.$ Then, using \cite[Proposition 3.2]{themelis2020douglas}, we reach our proof.
\end{proof}

Now, in the next proposition, we find the relation between the values of DYE and the main objective function.

\begin{prop}[Sandwiching property]\label{Sandwiching property}
    Let us consider Assumption \ref{Assumption 1} is satisfied. Let $\gamma<\frac{1}{L_g+L_h}$ be fixed, and consider $y,~z$ generated by DYS iteration starting from $x\in\mathbb{R}^n$. Then
    \begin{enumerate}
        \item[(i)] $\varphi^{\gamma}(x)\le\varphi(y).$
        \item[(ii)] $\varphi^{\gamma}(x)\ge\varphi(z)+\frac{1-\gamma(L_g+L_h)}{2\gamma}\|z-y\|^2.$
    \end{enumerate}
\end{prop}
\begin{proof}
    From (\ref{DYE_1}), we see
    \begin{align*}
        \varphi^{\gamma}(x) &= \min_{u\in\mathbb{R}^n}\left\{ f(u)+g(y)+h(y) +\langle u-y, \nabla g(y)+\nabla h(y)\rangle + \frac{1}{2\gamma}\|u-y\|^2\right\}\\
        &\le f(y)+g(y)+h(y)=\varphi(y).
    \end{align*}
    
 Hence, the proof of \textit{(i)} is done.

    Now, from (\ref{DYE}), we see
    \begin{align*}
        \varphi^{\gamma}(x) &= f(z)+g(y)+h(y) +\langle z-y, \nabla g(y)+\nabla h(y)\rangle + \frac{1}{2\gamma}\|z-y\|^2\\
        &=f(z) + g(y) + \langle z-y,\nabla g(y)\rangle+h(y)+\langle z-y,\nabla h(y)\rangle+\frac{1}{2\gamma}\|z-y\|^2\\
        &\ge f(z)+g(z)-\frac{L_g}{2}\|z-y\|^2+h(z)-\frac{L_h}{2}\|z-y\|^2+\frac{1}{2\gamma}\|z-y\|^2\\
        &\qquad\qquad\qquad\qquad\qquad\qquad\qquad\qquad\qquad[\text{Using decent lemma for $g$ and $h$}]\\
        &=\varphi(z)+\frac{1-\gamma(L_g+L_h)}{2\gamma}\|z-y\|^2.
    \end{align*}
    
    Hence, \textit{(ii)} is proved.
\end{proof}

Now, the next proposition connects the minimum value and the set of minimizers of DYE and the main objective function. This proposition also establishes the equivalence of the level-boundedness property of the main objective function and DYE.

\begin{prop}[Minimization and level-boundedness equivalence]\label{Minimization and level-boundedness equivalence}
     Let us assume Assumption \ref{Assumption 1} is satisfied. For all $\gamma<\frac{1}{L_g+L_h}$ the following hold:
     \begin{itemize}
         \item[(i)] $\inf\varphi=\inf\varphi^{\gamma}.$
         \item[(ii)] $\argmin \varphi=\mathrm{prox}_{\gamma g}(\argmin\varphi^{\gamma}).$
         \item[(iii)] $\varphi$ is level bounded if and only if so is $\varphi^{\gamma}.$
     \end{itemize}
\end{prop}
\begin{proof}
    In Proposition \ref{Sandwiching property}, we see that for any sequence $(x,~y,~z)$ from Algorithm \ref{DYS},
    {\allowdisplaybreaks
    \begin{align}\label{Prop 4.3 (i)}
        \notag
        &\varphi(z)+\frac{1-\gamma(L_g+L_h)}{2\gamma}\|z-y\|^2\le\varphi^{\gamma}(x)\le\varphi(y)\\\notag
        \Rightarrow\;&\varphi(z)+\frac{1-\gamma(L_g+L_h)}{2\gamma}\|z-\mathrm{prox}_{\gamma g}(x)\|^2\le\varphi^{\gamma}(x)\le\varphi(\mathrm{prox}_{\gamma g}(x))\\\notag
        \Rightarrow\;&\inf_{x\in\mathbb{R}^n}\left\{\varphi(z)+\frac{1-\gamma(L_g+L_h)}{2\gamma}\|z-\mathrm{prox}_{\gamma g}(x)\|^2\right\}\le\inf_{x\in\mathbb{R}^n}\varphi^{\gamma}(x)\\
        &\qquad\qquad\qquad\qquad\qquad\qquad\qquad\qquad\qquad\qquad\qquad\quad\le\inf_{x\in\mathbb{R}^n}\varphi(\mathrm{prox}_{\gamma g}(x)).
    \end{align}}

    Now, we have $\gamma<\frac{1}{L_g+L_h}$. Then
    \begin{align*}
        &\varphi(z)\le\inf_{x\in\mathbb{R}^n}\varphi^{\gamma}(x)\le\inf_{y\in\mathbb{R}^n}\varphi(y)\\
        \Rightarrow\;&\inf_{x\in\mathbb{R}^n}\varphi^{\gamma}(x)=\inf_{y\in\mathbb{R}^n}\varphi(y)\\
        \Rightarrow\;&\inf\varphi=\inf\varphi^{\gamma}.
    \end{align*}

    Hence, \textit{(i)} is proved.

    Let $\tilde{x}\in\argmin_{x\in\mathbb{R}^n}\varphi^{\gamma}(x)$, which implies that
    \begin{align}
        \notag
        \Rightarrow\;&\varphi^{\gamma}(\tilde{x})\le\varphi^{\gamma}(x),\;\;\;\;\;\forall x\in\mathbb{R}^n\\\notag
        \Rightarrow\;&\varphi^{\gamma}(\tilde{x})\le\varphi(\mathrm{prox}_{\gamma g}(x)),\;\;\;\;\;\forall x\in\mathbb{R}^n\;\;\;\;\;[\text{By using Proposition \ref{Sandwiching property}}(i)]\\\notag
        \Rightarrow\;&\varphi^{\gamma}(\tilde{x})\le(\varphi\circ\mathrm{prox}_{\gamma g})(x),\;\;\;\;\;\forall x\in\mathbb{R}^n\\\notag
        \Rightarrow\;&(\varphi\circ\mathrm{prox}_{\gamma g})(\tilde{x})\le(\varphi\circ\mathrm{prox}_{\gamma g})(x),\;\;\;\;\;\forall x\in\mathbb{R}^n\;\;\;\;\;[\text{By using Proposition }\ref{Minimization and level-boundedness equivalence}(i)]\\\notag
        \Rightarrow\;&\tilde{x}\in\argmin_{x\in\mathbb{R}^n}(\varphi\circ\mathrm{prox}_{\gamma g})(x)\\\notag
        \Rightarrow\;&\argmin_{x\in\mathbb{R}^n}\varphi^{\gamma}(x)\subseteq\argmin_{x\in\mathbb{R}^n}(\varphi\circ\mathrm{prox}_{\gamma g})(x)\\
        \Rightarrow\;&\argmin\varphi^{\gamma}\subseteq\argmin(\varphi\circ\mathrm{prox}_{\gamma g})\label{Prop 4.3 (ii) 1}.
    \end{align}

    Let us suppose $\tilde{x}\notin \argmin_{x\in\mathbb{R}^n}\varphi^{\gamma}(x)$, which implies that
    \begin{align}
        \notag
        \Rightarrow\;&\exists~ x\in\mathbb{R}^n~\text{s.t. }\varphi^{\gamma}(\tilde{x})>\varphi^{\gamma}(x)\\\notag
        \Rightarrow\;&(\varphi\circ\mathrm{prox}_{\gamma g})(\tilde{x})>\inf \varphi^{\gamma}\;\;\;\;\;[\text{By using (\ref{Sandwiching property})}]\\\notag
        \Rightarrow\;&(\varphi\circ\mathrm{prox}_{\gamma g})(\tilde{x})>\inf \varphi\;\;\;\;\;[\text{By using Proposition }\ref{Minimization and level-boundedness equivalence}(i)]\\\notag
        \Rightarrow\;&(\varphi\circ\mathrm{prox}_{\gamma g})(\tilde{x})>\inf (\varphi\circ\mathrm{prox}_{\gamma g})\\
        \Rightarrow\;&\tilde{x}\notin\argmin_{x\in\mathbb{R}^n}(\varphi\circ\mathrm{prox}_{\gamma g})(x).\label{Prop 4.3 (ii) 2}
    \end{align}
 Using the contrapositive law of logic in (\ref{Prop 4.3 (ii) 2}), we get
    \allowdisplaybreaks{
    \begin{align}\label{Prop 4.3 (ii) 3}
        \notag
        &\text{if }\tilde{x}\in\argmin_{x\in\mathbb{R}^n}(\varphi\circ\mathrm{prox}_{\gamma g})(x)\text{ then }\tilde{x}\in \argmin_{x\in\mathbb{R}^n}\varphi^{\gamma}(x)\\\notag
        \Rightarrow\;&\argmin_{x\in\mathbb{R}^n}(\varphi\circ\mathrm{prox}_{\gamma g})(x)\subseteq\argmin_{x\in\mathbb{R}^n}\varphi^{\gamma}(x)\\
        \Rightarrow\;&\argmin(\varphi\circ\mathrm{prox}_{\gamma g})\subseteq\argmin\varphi^{\gamma}.
    \end{align}}

    Now, from (\ref{Prop 4.3 (ii) 1}) and (\ref{Prop 4.3 (ii) 3}), we get
    \begin{align*}
        &\argmin(\varphi\circ\mathrm{prox}_{\gamma g})=\argmin\varphi^{\gamma}\\
        \Rightarrow\;&\mathrm{prox}_{\gamma g}(\argmin(\varphi\circ\mathrm{prox}_{\gamma g}))=\mathrm{prox}_{\gamma g}(\argmin\varphi^{\gamma})\\
        \Rightarrow\;&\argmin \varphi=\mathrm{prox}_{\gamma g}(\argmin\varphi^{\gamma}).
    \end{align*}

    Hence, \textit{(ii)} is proved.

    Let $\varphi^{\gamma}$ be level bounded and let $y\in\mathrm{lev}_{\le \alpha}\varphi$ for some $\alpha>\inf \varphi.$ Now, we have $y=\mathrm{prox}_{\gamma g}(x),$ i.e., $ x=y+\gamma\nabla g(y).$ Then,
    \begin{align*}
        &y\in\mathrm{lev}_{\le\alpha}\varphi\\
        \Rightarrow\;&\varphi(y)\le\alpha\\
        \Rightarrow\;&\varphi^{\gamma}(x)\le\alpha\;\;\;\;\;[\text{By using Proposition \ref{Sandwiching property}}(i)]\\
        \Rightarrow\;&x\in\mathrm{lev}_{\le\alpha}\varphi^{\gamma}\\
        \Rightarrow\;&(I+\gamma\nabla g)(y)\in\mathrm{lev}_{\le\alpha}\varphi^{\gamma}\\
        \Rightarrow\;&y\in(I+\gamma\nabla g)^{-1}\mathrm{lev}_{\le\alpha}\varphi^{\gamma}\\
        \Rightarrow\;&\mathrm{lev}_{\le\alpha}\varphi\subseteq(I+\gamma\nabla g)^{-1}\mathrm{lev}_{\le\alpha}\varphi^{\gamma}\\
        \Rightarrow\;&\mathrm{lev}_{\le\alpha}\varphi\subseteq\mathrm{prox}_{\gamma g}(\mathrm{lev}_{\le\alpha}\varphi^{\gamma}).
    \end{align*}
    
 Since $\mathrm{prox}_{\gamma g}$ is Lipschitz continuous and $\varphi^{\gamma}$ is level bounded. Hence, $\varphi$ is level bounded.

    Let us suppose that $\varphi^\gamma$ is not level bounded. Then $\exists\;\alpha > \inf \varphi$ and a sequence $\{x^k\}$ satisfying $x^k \in \mathrm{lev}_{\leq \alpha} \varphi^\gamma$ s.t. $\|x^k\| \geq k, \;\forall k \in \mathbb{N}$. Let $y^k = \mathrm{prox}_{\gamma g}(x^k) \ \Rightarrow \ x^k = y^k + \gamma \nabla g(y^k)$ and $z^k = \mathrm{prox}_{\gamma f}(2y^k - x^k - \nabla h(y^k)) = \mathrm{prox}_{\gamma f}(y^k - \nabla g(y^k) - \nabla h(y^k))$.
    
    If $\gamma < \frac{1}{L_g + L_h}$, then by Proposition \ref{Sandwiching property}, $\varphi(z^k) \leq \varphi^\gamma(x^k) \leq \alpha \; \Rightarrow \ z^k \in \mathrm{lev}_{\leq \alpha} \varphi.$

    Now, using Proposition \ref{Sandwiching property}, we say
    \begin{align}\label{Prop 4.3 (iii) 1}
        \notag
        &\alpha - \inf \varphi \geq \varphi^\gamma(x^k) - \inf \varphi \geq \varphi^\gamma(x^k) - \varphi(z^k) \geq \frac{1 - (L_g + L_h)\gamma}{2\gamma} \|z^k - y^k\|^2\\
        \Rightarrow\;&\|z^k - y^k\|^2 \leq \frac{2\gamma(\alpha - \inf \varphi)}{1 - (L_g + L_h)\gamma}.
    \end{align}

 Then, using the triangle inequality of the norm and (\ref{Prop 4.3 (iii) 1}), we get
    \allowdisplaybreaks{
    \begin{align*}
        &\|z^k\| = \|y^k - y^0 + y^0 - y^k + z^k\|\geq \|y^k - y^0\| - \|y^0\| - \|z^k - y^k\|\\
        &\qquad\qquad\qquad\qquad\qquad\qquad\;\;\;\geq \frac{1}{1 + \gamma L_g} \|x^k - x^0\| - \|y^0\| - \sqrt{\frac{2\gamma(\alpha - \inf \varphi)}{1 - (L_g + L_h)\gamma}}\\
        \Rightarrow\;&\|z^k\|\ge \frac{\|x^k\| - \|x^0\|}{1 + \gamma L_g} - \|y^0\| - \sqrt{\frac{2\gamma(\alpha - \inf \varphi)}{1 - (L_g + L_h)\gamma}}\\
        \Rightarrow \;& \|z^k\| \geq \frac{k - \|x^0\|}{1 + \gamma L_g} - \|y^0\| - \sqrt{\frac{2\gamma(\alpha - \inf \varphi)}{1 - (L_g + L_h)\gamma}}\to \infty \ \text{as } k \to \infty\\
        \Rightarrow \;& \varphi \text{ is not level bounded.}
    \end{align*}}

    Now, by the contrapositive law of logic, we say that if $\varphi$ is level bounded, then $\varphi^\gamma$ is level bounded.
    
      Thus, $\varphi$ is level bounded if and only if $\varphi^\gamma$ is level bounded. Hence, \textit{(iii)} is proved.
\end{proof}

Now, the next proposition helps to bound the DYE by the objective function.

\begin{prop}\label{quadratic upper bound}
    Consider that Assumption \ref{Assumption 1} holds. Let $\gamma<\frac{1}{L_g+L_h}$ be fixed, $x,\;\tilde{y}\in\mathbb{R}^n$, and $y=\mathrm{prox}_{\gamma g}x$. Then $\varphi^{\gamma}(x)\le\varphi(\tilde{y})+\frac{1+\gamma(L_g+L_h)}{2\gamma}\|y-\tilde{y}\|^2.$
\end{prop}
\begin{proof}
    Since $y=\mathrm{prox}_{\gamma g}x$, from \eqref{DYE_1}, we have
    \begin{equation}\label{DYE_prop}
        \varphi^{\gamma}(x) = \min_{u\in\mathbb{R}^n}\left\{ f(u)+g(y)+h(y) +\langle u-y, \nabla g(y)+\nabla h(y)\rangle + \frac{1}{2\gamma}\|u-y\|^2\right\}.
    \end{equation}

    Now, we plug $u=\tilde{y}$ into \eqref{DYE_prop} and get
    \begin{align*}
        \varphi^{\gamma}(x)&\le f(\tilde{y})+g(y)+h(y)+\langle \tilde{y}-y,\nabla g(y)+\nabla h(y)\rangle+\frac{1}{2\gamma}\|\tilde{y}-y\|^2\\
        &=f(\tilde{y})+[g(y)+\langle\nabla g(y), \tilde{y}-y\rangle]+[h(y)+\langle\nabla h(y),\tilde{y}-y\rangle]+\frac{1}{2\gamma}\|\tilde{y}-y\|^2\\
        &\le f(\tilde{y})+g(\tilde{y})+\frac{L_g}{2}\|\tilde{y}-y\|^2+h(\tilde{y})+\frac{L_h}{2}\|\tilde{y}-y\|^2+\frac{1}{2\gamma}\|\tilde{y}-y\|^2\\
        &\qquad\qquad\qquad\qquad\qquad\qquad\qquad\qquad[\text{By using descent lemma for $g$ and $h$}]\\
        &=\varphi(\tilde{y})+\frac{1+\gamma(L_g+L_h)}{2\gamma}\|\tilde{y}-y\|^2.
    \end{align*}
\end{proof}

\section{FISTA-type DYS}\label{FISTA-type DYS section}
FISTA-type DYS algorithm is used to solve a non-smooth and non-convex optimization problem. This algorithm is based on the DYS framework. In this algorithm, we try to keep the simple structure of the DYS algorithm while adding the acceleration term to improve the convergence speed. Since the FISTA-type DYS algorithm does not require convexity of the objective function, this algorithm is very useful for many practical problems. Even though the DYS algorithm has good convergence properties, it can be very slow to converge in some cases, like ill-posed or poorly scaled problems. To tackle these difficulties, we use the FISTA-type acceleration term. Now, we will discuss the FISTA-type DYS algorithm for the same problem (\ref{Prob 1}) and analyze its convergence properties.

\begin{algorithm}[h]
\caption{FISTA-type DYS}
\begin{algorithmic}[1]

\State \textbf{Input:} $x^0 = w^{-1} \in \mathbb{R}^n$, step size $\gamma > 0$, relaxation parameter $\lambda>0$.
\State $t^0 = 1$

\For{$k = 0,1,2,\dots$}
    \State $y^k = \mathrm{prox}_{\gamma g}(x^k)$
    \State $w^k \in \mathrm{prox}_{\gamma f}\big(2y^k - x^k - \gamma \nabla h(y^k)\big)$
    \State $t^{k+1} = \frac{1 + \sqrt{1 + 4(t^k)^2}}{2}$
    \State $\beta^k = \frac{t^k - 1}{t^{k+1}}$
    \State $z^k = w^k + \beta^k (w^k - w^{k-1})$
    \State $x^{k+1} = x^k + \lambda(z^k - y^k)$
\EndFor

\end{algorithmic}\label{FISTA-DYS}
\end{algorithm}


\subsection{Convergence Analysis for FISTA-type DYS}
To analyze the convergence of our algorithms, we use an approach similar to that of Themelis and Patrinos \cite{themelis2020douglas}. Before analyzing the sufficient decrease property of the FISTA-type DYS algorithm, we discuss the sufficient decrease property for DYS in the next lemma, which is very helpful to prove the sufficient decrease property for FISTA-type DYS.
\begin{lem}[Sufficient decrease for DYS]\label{DYS-SD} Consider that Assumption \ref{Assumption 1} is satisfied and let $(y^k,~z^k,~x^{k+1})\in \text{DYS}(x^k)$. Then for some $C<0$
    $$\varphi^{\gamma}(x^{k+1})\le\varphi^{\gamma}(x^k)+\frac{C}{(1+\gamma L_g)^2}\|x^{k+1}-x^k\|^2,$$
    if the following hold:
    \begin{itemize}
        \item[(i)] $0<\gamma<\frac{1}{L_1},~\frac{4L_2}{L_1+L_2}<\lambda<2\left(1+\frac{\sigma_1}{L_1}\right),~L_1>\frac{L_2-\sigma_1+\sqrt{(L_2-\sigma_1)^2-4L_2\sigma_1}}{2}.$
        \item[(ii)] $0<\gamma<\frac{1}{L_1},~2\gamma L_1\left(1+\frac{\sigma_1}{L_1}\right)<\lambda\le\frac{\sigma_1+2L_2}{\frac{\sigma_1L_1}{2(L_1+\sigma_1)}+\frac{\sigma_2L_2}{2(L_2+\sigma_2)}},~\sigma_1\le-2L_2,~L_1>\max\left\{|\sigma_1|,\frac{\sigma_1\sigma_2}{2L_2+\sigma_2}\right\}.$
        \item[(iii)] $0<\gamma<\frac{1}{L_1+\sigma_1},~2\left(1+\frac{\sigma_1}{L_1}\right)\le\lambda\le\frac{2}{1-\frac{\sigma_1L_1}{(L_1+\sigma_1)^2}-\frac{\sigma_2L_2}{(L_1+\sigma_1)(L_2+\sigma_2)}+\frac{2L_1L_2}{(L_1+\sigma_1)^2}},~0>\sigma_1\ge\frac{\sigma_2L_2}{L_2+\sigma_2},~0<L_1\le\frac{1}{L_2(2L_2+\sigma_2)}(L_2\sigma_1\sigma_2-L_2\sigma_1^2-\sigma_1^2\sigma_2).$
    \end{itemize}

    Here, $\sigma_1=\min\{\sigma_g,~0\},~L_1\ge L_g$ s.t. $L_1+\sigma_1>0$ and $\sigma_2=\min\{\sigma_h,~0\},~L_2\ge L_h$ s.t. $L_2+\sigma_2>0.$
\end{lem}
\begin{proof}
    \justifying
    From (\ref{DYE_1}), for $y^k=\mathrm{prox}_{\gamma g}(x^k)$, we write
    \allowdisplaybreaks{
    \begin{align}\label{sufficient decrease 1}
        \notag
        \varphi^{\gamma}(x^{k+1}) &=  \min_{u\in\mathbb{R}^n}\bigg\{ f(u)+g(y^{k+1})+h(y^{k+1}) +\langle u-y^k, \nabla g(y^{k+1})+\nabla h(y^{k+1})\rangle \\\notag
        &\qquad\qquad+ \frac{1}{2\gamma}\|u-y^{k+1}\|^2\bigg\}\\\notag
        & \leq f(z^k)+g(y^{k+1})+h(y^{k+1})+\langle z^k-y^{k+1},\nabla g(y^{k+1}) + \nabla h(y^{k+1})\rangle\\\notag
        &\qquad+ \frac{1}{2\gamma}\|z^k-y^{k+1}\|^2\\\notag
        & = f(z^k)+g(y^{k+1})+\langle\nabla g(y^{k+1}), z^k-y^{k+1}\rangle+h(y^{k+1})+\langle\nabla h(y^{k+1}), \\\notag
        &\qquad z^k-y^{k+1}\rangle+ \frac{1}{2\gamma}\|z^k-y^{k+1}\|^2\\\notag
        \Rightarrow\; \varphi^{\gamma}(x^{k+1}) & = f(z^k)+g(y^{k+1})+\langle\nabla g(y^{k+1}), z^k-y^{k+1}\rangle+h(y^{k+1})+\langle\nabla h(y^{k+1}), \\
        &\qquad z^k-y^{k+1}\rangle+ \frac{1}{2\gamma}\|z^k-y^{k+1}\|^2.
    \end{align}}

    Now, by using the \cite[Theorem 2.2]{themelis2020douglas}, (\ref{sufficient decrease 1}) becomes
    \allowdisplaybreaks{
    \begin{align}\label{sufficient decrease 2}
        \notag
        \varphi^{\gamma}(x^{k+1}) &\leq f(z^k)+g(y^k)-\rho_g(y^k,y^{k+1})+h(y^k)-\rho_h(y^k,y^{k+1})+\langle z^k-y^k,\\\notag
        &\qquad \nabla g(y^{k+1}) + \nabla h(y^{k+1})\rangle+ \frac{1}{2\gamma}\|z^k-y^{k+1}\|^2\\\notag
        \Rightarrow\; \varphi^{\gamma}(x^{k+1})& \le f(z^k)+g(y^k)+h(y^k)+\langle z^k-y^k,\nabla g(y^k) + \nabla h(y^k)\rangle+ \frac{1}{2\gamma}\|z^k-y^k\|^2\\\notag
        &\qquad-\rho_g(y^k,y^{k+1})-\rho_h(y^k,y^{k+1})-\langle z^k-y^k,\nabla g(y^k) + \nabla h(y^k)\rangle\\
        &\qquad+ \frac{1}{2\gamma}\|z^k-y^{k+1}\|^2- \frac{1}{2\gamma}\|z^k-y^k\|^2.
    \end{align}}

    Substituting (\ref{DYE}) into (\ref{sufficient decrease 2}), we get
    \begin{align*}
        \varphi^{\gamma}(x^{k+1})&\leq \varphi^{\gamma}(x^k)-\rho_g(y^k,y^{k+1})+\langle z^k-y^k,\nabla g(y^{k+1})-\nabla g(y^k)\rangle-\rho_h(y^k,y^{k+1})\\
        &\qquad+\langle z^k-y^k,\nabla h(y^{k+1})-\nabla h(y^k)\rangle + \frac{1}{2\gamma}\|z^k-y^k\|^2 + \frac{1}{2\gamma}\|y^k-y^{k+1}\|^2\\
        &\qquad+\frac{1}{\gamma}\langle z^k-y^k, y^k-y^{k+1}\rangle - \frac{1}{2\gamma}\|z^k-y^k\|^2,
    \end{align*}
    which implies that
    \begin{align}\label{sufficient decrease 3}
        \notag
        \varphi^{\gamma}(x^{k+1})&= \varphi^{\gamma}(x^k)-\rho_g(y^k,y^{k+1})+\langle z^k-y^k,\nabla g(y^{k+1})-\nabla g(y^k)\rangle\\\notag
        &\qquad-\rho_h(y^k,y^{k+1})+\langle z^k-y^k,\nabla h(y^{k+1})-\nabla h(y^k)\rangle+\frac{1}{\gamma}\langle z^k-y^k,\\
        &\qquad y^k-y^{k+1}\rangle+ \frac{1}{2\gamma}\|y^k-y^{k+1}\|^2.
    \end{align}

    Now, from Algorithm \ref{DYS}, we show that
    \begin{equation}\label{residual}
        z^k-y^k=\frac{1}{\lambda}(y^{k+1}-y^k)+\frac{\gamma}{\lambda}(\nabla g(y^{k+1})-\nabla g(y^k)).
    \end{equation}

    Again, substituting (\ref{residual}) in (\ref{sufficient decrease 3}), we get
    \begin{align}\label{sufficient decrease 4}
        \notag
        \varphi^{\gamma}(x^{k+1})&\leq \varphi^{\gamma}(x^k)-\rho_g(y^k,y^{k+1})+\bigg\langle \frac{1}{\lambda}(y^{k+1}-y^k)+\frac{\gamma}{\lambda}(\nabla g(y^{k+1})-\nabla g(y^k)),\\\notag
        &\qquad\nabla g(y^{k+1})-\nabla g(y^k)\bigg\rangle-\rho_h(y^k,y^{k+1})+\langle z^k-y^k,\\\notag
        &\qquad\nabla h(y^{k+1})-\nabla h(y^k)\rangle+ \frac{1}{2\gamma}\|y^k-y^{k+1}\|^2+\frac{1}{\gamma}\bigg\langle \frac{1}{\lambda}(y^{k+1}-y^k)\\\notag
        &\qquad +\frac{\gamma}{\lambda}(\nabla g(y^{k+1})-\nabla g(y^k)),y^k-y^{k+1}\bigg\rangle\\\notag
        \Rightarrow\;\varphi^{\gamma}(x^{k+1})&= \varphi^\gamma(x^k) - \rho_g(y^k, y^{k+1}) + \frac{\gamma}{\lambda} \|\nabla g(y^{k+1}) - \nabla g(y^k)\|^2 - \rho_h(y^k, y^{k+1}) \\
        &\qquad+ \frac{\lambda - 2}{2\lambda\gamma} \|y^k - y^{k+1}\|^2+ \langle z^k - y^k, \nabla h(y^{k+1}) - \nabla h(y^k) \rangle.
    \end{align}
    
    Hence, from (\ref{sufficient decrease 4}), we get
    \begin{align}\label{sufficient decrease 5}
        \notag
        \varphi^{\gamma}(x^{k+1})- \varphi^\gamma(x^k) &\leq  - \rho_g(y^k, y^{k+1}) + \frac{\gamma}{\lambda} \|\nabla g(y^{k+1}) - \nabla g(y^k)\|^2 - \rho_h(y^k, y^{k+1}) \\
        &\qquad+ \frac{\lambda - 2}{2\lambda\gamma} \|y^k - y^{k+1}\|^2+ \langle z^k - y^k, \nabla h(y^{k+1}) - \nabla h(y^k) \rangle.
    \end{align}

    Let us suppose $\sigma_1 = \min\{0,\sigma_g\},~L_1\geq L_g$ such that $L_1+\sigma_1>0.$

    Then, by \cite[Theorem 2.2]{themelis2020douglas}, we get
    \begin{equation}\label{hypoconvex bound g}
        \rho_g(y^k,y^{k+1})=\frac{\sigma_1L_1}{2(L_1+\sigma_1)}\|y^k-y^{k+1}\|^2+\frac{1}{2(L_1+\sigma_1)}\|\nabla g(y^k)-\nabla g(y^{k+1})\|^2.
    \end{equation}

    Again, we substitute (\ref{hypoconvex bound g}) into (\ref{sufficient decrease 5}) and use the Cauchy-Schwarz inequality and get
    {\allowdisplaybreaks
    \begin{align}\label{sufficient decrease 6}
        \notag
        &\varphi^{\gamma}(x^{k+1})- \varphi^\gamma(x^k)\\\notag &\leq \frac{\gamma}{\lambda} \|\nabla g(y^{k+1}) - \nabla g(y^k)\|^2-\frac{\sigma_1L_1}{2(L_1+\sigma_1)}\|y^k-y^{k+1}\|^2 - \rho_h(y^k, y^{k+1})\\\notag
        &\qquad-\frac{1}{2(L_1+\sigma_1)}\|\nabla g(y^k)-\nabla g(y^{k+1})\|^2+ \langle z^k - y^k, \nabla h(y^{k+1}) - \nabla h(y^k) \rangle\\\notag
        &\qquad+ \frac{\lambda - 2}{2\lambda\gamma} \|y^k - y^{k+1}\|^2\\\notag
        &\le\left[\frac{\lambda-2}{2\lambda\gamma}-\frac{\sigma_1L_1}{2(L_1+\sigma_1)}\right]\|y^k-y^{k+1}\|^2 +\left[\frac{\gamma}{\lambda}-\frac{1}{2(L_1+\sigma_1)}\right]\|\nabla g(y^k)-\nabla g(y^{k+1})\|^2\\
        &\qquad- \rho_h(y^k, y^{k+1})+\|z^k-y^k\|\|\nabla h(y^{k+1})-\nabla h(y^k)\|.
    \end{align}}
    
    Now, using the $L_g$-smoothness of $g$, $L_h$-smoothness of $h$, $L_1\ge L_g$ and (\ref{residual}), we get
    {\allowdisplaybreaks
    \begin{align}\label{sufficient decrease 7}
        \notag
        &\varphi^{\gamma}(x^{k+1})- \varphi^\gamma(x^k)\\\notag
        &\le\left[\frac{\lambda-2}{2\lambda\gamma}-\frac{\sigma_1L_1}{2(L_1+\sigma_1)}\right]\|y^k-y^{k+1}\|^2+\left[\frac{\gamma}{\lambda}-\frac{1}{2(L_1+\sigma_1)}\right]\|\nabla g(y^k)-\nabla g(y^{k+1})\|^2\\\notag
        &\qquad- \rho_h(y^k, y^{k+1})+L_h\|y^{k+1}-y^k\|\bigg(\frac{1}{\lambda}\|y^{k+1}-y^k\| +\frac{\gamma}{\lambda}\|\nabla g(y^k)-\nabla g(y^{k+1})\|\bigg)\\\notag
        &\le\left[\frac{\lambda-2}{2\lambda\gamma}-\frac{\sigma_1L_1}{2(L_1+\sigma_1)}\right]\|y^k-y^{k+1}\|^2+\left[\frac{\gamma}{\lambda}-\frac{1}{2(L_1+\sigma_1)}\right]\|\nabla g(y^k)-\nabla g(y^{k+1})\|^2\\\notag
        &\qquad- \rho_h(y^k, y^{k+1})+\frac{L_h}{\lambda}\|y^k-y^{k+1}\|^2+\frac{L_hL_1\gamma}{\lambda}\|y^k-y^{k+1}\|^2\\\notag
        &=\left[\frac{\lambda-2}{2\lambda\gamma}-\frac{\sigma_1L_1}{2(L_1+\sigma_1)}\right]\|y^k-y^{k+1}\|^2 +\left[\frac{\gamma}{\lambda}-\frac{1}{2(L_1+\sigma_1)}\right]\|\nabla g(y^k)-\nabla g(y^{k+1})\|^2\\
        &\qquad- \rho_h(y^k, y^{k+1})+\left[\frac{L_h}{\lambda}+\frac{L_hL_1\gamma}{\lambda}\right]\|y^k-y^{k+1}\|^2.
    \end{align}}

    Again, suppose $\sigma_2=\min\{\sigma_h,0\},~L_2\ge L_h$ to be such that $L_2+\sigma_2>0$.

    Then, again, by \cite[Theorem 2.2]{themelis2020douglas}, we obtain
    \begin{equation}\label{hypoconvex bound h}
        \rho_h(y^k,y^{k+1})=\frac{\sigma_2L_2}{2(L_2+\sigma_2)}\|y^k-y^{k+1}\|^2+\frac{1}{2(L_2+\sigma_2)}\|\nabla g(y^k)-\nabla g(y^{k+1})\|^2.
    \end{equation}
    
    Again, we substitute (\ref{hypoconvex bound h}) into (\ref{sufficient decrease 7}) and get
    {\allowdisplaybreaks
    \begin{align}\label{sufficient decrease 8}
        \notag
        &\varphi^{\gamma}(x^{k+1})- \varphi^\gamma(x^k)\\\notag
        &\le\left[\frac{\lambda-2}{2\lambda\gamma}-\frac{\sigma_1L_1}{2(L_1+\sigma_1)}\right]\|y^k-y^{k+1}\|^2+\left[\frac{\gamma}{\lambda}-\frac{1}{2(L_1+\sigma_1)}\right]\|\nabla g(y^k)-\nabla g(y^{k+1})\|^2\\\notag
        &\qquad-\frac{1}{2(L_2+\sigma_2)}\|\nabla h(y^k)-\nabla h(y^{k+1})\|^2-\frac{\sigma_2L_2}{2(L_2+\sigma_2)}\|y^k-y^{k+1}\|^2\\\notag
        &\qquad+\left[\frac{L_2}{\lambda}+\frac{L_1L_2\gamma}{\lambda}\right]\|y^k-y^{k+1}\|^2\\\notag
        &\Rightarrow\;\varphi^{\gamma}(x^{k+1})- \varphi^\gamma(x^k)\\\notag
        &\qquad\le\bigg[\frac{\lambda-2}{2\lambda\gamma}-\frac{\sigma_1L_1}{2(L_1+\sigma_1)}-\frac{\sigma_2L_2}{2(L_2+\sigma_2)}+\frac{L_2}{\lambda}+\frac{L_1L_2\gamma}{\lambda}\bigg]\|y^k-y^{k+1}\|^2\\
        &\qquad\qquad+\left[\frac{\gamma}{\lambda}-\frac{1}{2(L_1+\sigma_1)}\right]\|\nabla g(y^k)-\nabla g(y^{k+1})\|^2.
    \end{align}}
    \begin{itemize}
        \item \textbf{Case 1:} $0<\lambda<2\left(1+\frac{\sigma_1}{L_1}\right).$
        \begin{itemize}
            \item \textbf{Case 1a:} If $\frac{\gamma}{\lambda}-\frac{1}{2(L_1+\sigma_1)}\ge0\Rightarrow \lambda\le 2\gamma(L_1+\sigma_1)=2\gamma L_1\left(1+\frac{\sigma_1}{L_1}\right)$ and $\lambda<2\left(1+\frac{\sigma_1}{L_1}\right)$, then $2\gamma L_1\left(1+\frac{\sigma_1}{L_1}\right)<2\left(1+\frac{\sigma_1}{L_1}\right)\Rightarrow\gamma L_1<1\Rightarrow \gamma<\frac{1}{L_1}$.

            Now, from (\ref{sufficient decrease 8}), we get
            \begin{align}\label{Case 1a 1}
                \notag
                &\frac{\lambda - 2}{2\lambda\gamma} - \frac{\sigma_1 L_1}{2(L_1 + \sigma_1)}- \frac{\sigma_2 L_2}{2(L_2 + \sigma_2)}+ \frac{L_2}{\lambda}+ \frac{L_1 L_2 \gamma}{\lambda}\\\notag
                &\qquad+ L_1^2 \left[ \frac{\gamma}{\lambda} - \frac{1}{2(L_1 + \sigma_1)} \right]\\\notag
                &< \frac{\frac{\sigma_1}{L_1}}{\lambda\gamma}- \frac{\sigma_1 L_1}{2(L_1 + \sigma_1)}- \frac{\sigma_2 L_2}{2(L_2 + \sigma_2)}+ \frac{2L_2}{\lambda}+ \frac{L_1}{\lambda}- \frac{L_1^2}{2(L_1 + \sigma_1)} \\\notag
                &= \frac{\sigma_1}{\lambda \gamma L_1}- \frac{\sigma_1 L_1}{2(L_1 + \sigma_1)}- \frac{\sigma_2 L_2}{2(L_2 + \sigma_2)}+ \frac{2L_2}{\lambda}+ \frac{L_1}{\lambda}- \frac{L_1^2}{2(L_1 + \sigma_1)} \\
                &< \frac{\sigma_1}{\lambda}- \frac{L_1}{2}- \frac{\sigma_2 L_2}{2(L_2 + \sigma_2)}+ \frac{2L_2}{\lambda}+ \frac{L_1}{\lambda}.
            \end{align}

            Now, to get (\ref{Case 1a 1}) is non-positive, we use $\frac{\sigma_1}{\lambda}- \frac{L_1}{2}- \frac{\sigma_2 L_2}{2(L_2 + \sigma_2)}+ \frac{2L_2}{\lambda}+ \frac{L_1}{\lambda}\le 0,$ which implies that
            \begin{align}\label{Case 1a 2}
                \notag
                & \frac{1}{\lambda}(\sigma_1+L_1+2L_2)\le \frac{1}{2}\left(L_1+\frac{\sigma_2L_2}{L_2+\sigma_2}\right)\\\notag
                \Rightarrow\;& \frac{2L_2}{\lambda} < \frac{1}{2}\left(L_1+\frac{\sigma_2L_2}{L_2+\sigma_2}\right)\;\;\;\;\;[\text{since }L_1+\sigma_1>0]\\\notag
                \Rightarrow\;&\frac{2L_2}{\lambda}< \frac{1}{2}(L_1+\sigma_2)\;\;\;\;\;[\text{since }L_2+\sigma_2\le L_2\text{ and } \sigma_2\le 0]\\\notag
                \Rightarrow\;&\frac{2L_2}{\lambda}< \frac{1}{2}(L_1+L_2)\;\;\;\;\;[\text{since }L_2> \sigma_2]\\
                \Rightarrow\;&\lambda>\frac{4L_2}{L_1+L_2}.
            \end{align}

             Now, we have $\frac{4L_2}{L_1+L_2}<\lambda<2\left(1+\frac{\sigma_1}{L_1}\right),$ which implies that,
             {\allowdisplaybreaks
            \begin{align}\label{Case 1a 3}
                \notag
                \Rightarrow\;&\frac{4L_2}{L_1+L_2}<2\left(1+\frac{\sigma_1}{L_1}\right)\\\notag
                \Rightarrow\;&\frac{2L_2}{L_1+L_2}<\frac{L_1+\sigma_1}{L_1}\\\notag
                \Rightarrow\;&2L_1L_2<L_1^2+L_1L_2+L_1\sigma_1+L_2\sigma_1\\\notag
                \Rightarrow\;&L_1^2-L_1(L_2-\sigma_1)+L_2\sigma_1>0\\\notag
                \Rightarrow\;&\left(L_1-\frac{L_2-\sigma_1+\sqrt{(L_2-\sigma_1)^2-4L_2\sigma_1}}{2}\right)\\
                &\qquad\left(L_1-\frac{L_2-\sigma_1-\sqrt{(L_2-\sigma_1)^2-4L_2\sigma_1}}{2}\right)>0.
            \end{align}}

            Now, $\frac{L_2-\sigma_1-\sqrt{(L_2-\sigma_1)^2-4L_2\sigma_1}}{2}<0$ and $L_1\ge0$.

            So, from (\ref{Case 1a 3}), we say that,
            $L_1-\frac{L_2-\sigma_1+\sqrt{(L_2-\sigma_1)^2-4L_2\sigma_1}}{2}>0$.
            
            Hence, we get $L_1>\frac{L_2-\sigma_1+\sqrt{(L_2-\sigma_1)^2-4L_2\sigma_1}}{2}.$

            \item \textbf{Case 1b:} If $\frac{\gamma}{\lambda}-\frac{1}{2(L_1+\sigma_1)}<0\Rightarrow \lambda> 2\gamma(L_1+\sigma_1)=2\gamma L_1\left(1+\frac{\sigma_1}{L_1}\right)$ and $\lambda<2\left(1+\frac{\sigma_1}{L_1}\right)$, then $2\gamma L_1\left(1+\frac{\sigma_1}{L_1}\right)<2\left(1+\frac{\sigma_1}{L_1}\right)\Rightarrow\gamma L_1<1\Rightarrow \gamma<\frac{1}{L_1}$.

            Now, from (\ref{sufficient decrease 8}), we get
            \begin{align}\label{Case 1b 1}
                \notag
                &\frac{\lambda - 2}{2\lambda\gamma} - \frac{\sigma_1 L_1}{2(L_1 + \sigma_1)}- \frac{\sigma_2 L_2}{2(L_2 + \sigma_2)}+ \frac{L_2}{\lambda}+ \frac{L_1 L_2 \gamma}{\lambda}\\\notag
                &< \frac{\frac{\sigma_1}{L_1}}{\lambda\gamma}- \frac{\sigma_1 L_1}{2(L_1 + \sigma_1)}- \frac{\sigma_2 L_2}{2(L_2 + \sigma_2)}+ \frac{2L_2}{\lambda} \\\notag
                &= \frac{\sigma_1}{\lambda \gamma L_1}- \frac{\sigma_1 L_1}{2(L_1 + \sigma_1)}- \frac{\sigma_2 L_2}{2(L_2 + \sigma_2)}+ \frac{2L_2}{\lambda} \\
                &< \frac{\sigma_1}{\lambda}- \frac{\sigma_1 L_1}{2(L_1 + \sigma_1)}- \frac{\sigma_2 L_2}{2(L_2 + \sigma_2)}+ \frac{2L_2}{\lambda}.
            \end{align}

            Now, to get (\ref{Case 1b 1}) is non-positive, we use $\frac{\sigma_1}{\lambda}- \frac{\sigma_1 L_1}{2(L_1 + \sigma_1)}- \frac{\sigma_2 L_2}{2(L_2 + \sigma_2)}+ \frac{2L_2}{\lambda}\le0$, which implies that,
            \begin{align}\label{Case 1b 2}
            &\frac{1}{\lambda}(\sigma_1+2L_2)\le\frac{\sigma_1 L_1}{2(L_1 + \sigma_1)}+ \frac{\sigma_2 L_2}{2(L_2 + \sigma_2)}.
            \end{align}

            Since, we have $\sigma_1,~\sigma_2\le0,~L_1,~L_2>0$, then the RHS of (\ref{Case 1b 2}) is negative. Because if the RHS is equal to $0$, then both of $\sigma_1$ and $\sigma_2$ should be $0$. Then \eqref{Case 1b 2} implies that $L_2=0$, which is not possible. So, from (\ref{Case 1b 2}), we say that $\sigma_1+2L_2<0\Rightarrow\sigma_1<-2L_2.$

            Then, from (\ref{Case 1b 2}), we get $\frac{1}{\lambda}\ge\frac{1}{\sigma_1+2L_2}\left(\frac{\sigma_1 L_1}{2(L_1 + \sigma_1)}+ \frac{\sigma_2 L_2}{2(L_2 + \sigma_2)}\right)$. This implies that
            \begin{align}\label{Case 1b 3}
                &\lambda\le\frac{\sigma_1+2L_2}{\frac{\sigma_1 L_1}{2(L_1 + \sigma_1)}+ \frac{\sigma_2 L_2}{2(L_2 + \sigma_2)}}.
            \end{align}

            Now,
            {\allowdisplaybreaks
            \begin{align}\label{Case 1b 4}
                \notag
                &2\left(1+\frac{\sigma_1}{L_1}\right)-\frac{\sigma_1+2L_2}{\frac{\sigma_1 L_1}{2(L_1 + \sigma_1)}+ \frac{\sigma_2 L_2}{2(L_2 + \sigma_2)}}>0\\\notag
                \Rightarrow\;&2\left(1+\frac{\sigma_1}{L_1}\right)\left(\frac{\sigma_1 L_1}{2(L_1 + \sigma_1)}+ \frac{\sigma_2 L_2}{2(L_2 + \sigma_2)}\right)-\sigma_1-2L_2<0\\\notag
                &\qquad\qquad\qquad\qquad\left[\text{since }\frac{\sigma_1 L_1}{2(L_1 + \sigma_1)}+ \frac{\sigma_2 L_2}{2(L_2 + \sigma_2)}<0\right]\\\notag
                \Rightarrow\;&\frac{L_1+\sigma_1}{L_1}\left(\frac{\sigma_1 L_1}{L_1 + \sigma_1}+ \frac{\sigma_2 L_2}{L_2 + \sigma_2}\right)-\sigma_1-2L_2<0\\\notag
                \Rightarrow\;&\sigma_1+\frac{\sigma_2L_2(L_1+\sigma_1)}{L_1(L_2+\sigma_2)}-\sigma_1-2L_2<0\\\notag
                \Rightarrow\;&\sigma_2L_2(L_1+\sigma_1)-2L_1L_2(L_2+\sigma_2)<0\\\notag
                \Rightarrow\;&L_1L_2\sigma_2+L_2\sigma_1\sigma_2-2L_1L_2^2-2L_1L_2\sigma_2<0\\\notag
                \Rightarrow\;&L_2\sigma_1\sigma_2-2L_1L_2^2-L_1L_2\sigma_2<0\\\notag
                \Rightarrow\;&\sigma_1\sigma_2-2L_1L_2-L_1\sigma_2<0\\\notag
                \Rightarrow\;&L_1(2L_2+\sigma_2)>\sigma_1\sigma_2\\
                \Rightarrow\;&L_1>\frac{\sigma_1\sigma_2}{2L_2+\sigma_2}.
            \end{align}}
        \end{itemize}
    \end{itemize}
    \begin{itemize}
        \item \textbf{Case 2:} $2\left(1+\frac{\sigma_1}{L_1}\right)\le\lambda<2.$
        \begin{itemize}
            \item \textbf{Case 2a:} If $\frac{\gamma}{\lambda}-\frac{1}{2(L_1+\sigma_1)}\le0\Rightarrow \gamma\le \frac{\lambda}{2(L_1+\sigma_1)}$ and $\lambda<2$, then $\gamma<\frac{1}{L_1+\sigma_1}.$
            
            Now, from (\ref{sufficient decrease 8}), we say that,
            \allowdisplaybreaks{
            \begin{align}\label{Case 2a 1}
                \notag
                &\frac{\lambda - 2}{2\lambda\gamma} - \frac{\sigma_1 L_1}{2(L_1 + \sigma_1)}- \frac{\sigma_2 L_2}{2(L_2 + \sigma_2)}+ \frac{L_2}{\lambda}+ \frac{L_1 L_2 \gamma}{\lambda}\\\notag
                &\le \frac{(\lambda-2)(L_1+\sigma_1)}{2\lambda}- \frac{\sigma_1 L_1}{2(L_1 + \sigma_1)}- \frac{\sigma_2 L_2}{2(L_2 + \sigma_2)} + \frac{L_2}{2\left(1+\frac{\sigma_1}{L_1}\right)} \\\notag
                &\qquad+ \frac{L_1L_2}{2(L_1+\sigma_1)}\\
                &=\frac{(\lambda-2)(L_1+\sigma_1)}{2\lambda}- \frac{\sigma_1 L_1}{2(L_1 + \sigma_1)}- \frac{\sigma_2 L_2}{2(L_2 + \sigma_2)} + \frac{L_1L_2}{L_1+\sigma_1}.
            \end{align}}

            Now, to get (\ref{Case 2a 1}) is non-positive, we use $\frac{(\lambda-2)(L_1+\sigma_1)}{2\lambda}- \frac{\sigma_1 L_1}{2(L_1 + \sigma_1)}- \frac{\sigma_2 L_2}{2(L_2 + \sigma_2)} + \frac{L_1L_2}{L_1+\sigma_1}\le0$, which implies that,
            \begin{align}\label{Case 2a 2}
                \notag
                &\frac{\lambda-2}{\lambda}\le\frac{\sigma_1L_1}{(L_1+\sigma_1)^2}+\frac{\sigma_2L_2}{(L_1+\sigma_1)(L_2+\sigma_2)}-\frac{2L_1L_2}{(L_1+\sigma_1)^2}=M\;(\text{say})\\\notag
                &\qquad\qquad\qquad\qquad\qquad\qquad\qquad\qquad\qquad\qquad\qquad\qquad[\text{Here }M<0]\\\notag
                \Rightarrow\;&1-\frac{2}{\lambda}\le M\\
                \Rightarrow\;&\lambda\le\frac{2}{1-M}.
            \end{align}

            We know, $\frac{2}{1-M}<2.$

            Now, $\frac{2}{1-M}-2\left(1+\frac{\sigma_1}{L_1}\right)\ge0$ gives that,
            \allowdisplaybreaks{
            \begin{align}\label{Case 2a 3}
                \notag
                &\frac{2}{1-\frac{\sigma_1L_1}{(L_1+\sigma_1)^2}-\frac{\sigma_2L_2}{(L_1+\sigma_1)(L_2+\sigma_2)}+\frac{2L_1L_2}{(L_1+\sigma_1)^2}}-2\left(1+\frac{\sigma_1}{L_1}\right)\ge0\\\notag
                \Rightarrow\;&L_1(L_1+\sigma_1)(L_2+\sigma_2)-(L_1+\sigma_1)^2(L_2+\sigma_2)+\sigma_1L_1(L_2+\sigma_2)\\\notag
                &\qquad+\sigma_2L_2(L_1+\sigma_1)-2L_1L_2(L_2+\sigma_2)\ge0\\\notag
                \Rightarrow\;&L_1^2L_2+L_1^2\sigma_2+2L_1L_2\sigma_1+2L_1\sigma_1\sigma_2-L_1^2L_2-L_1^2\sigma_2-2L_1L_2\sigma_1\\\notag
                &\qquad-2L_1\sigma_1\sigma_2-L_2\sigma_1^2-\sigma_1^2\sigma_2-L_1L_2\sigma_2+L_2\sigma_1\sigma_2-2L_1L_2^2\ge0\\
                \Rightarrow\;&2L_1L_2^2+L_1L_2\sigma_2\le L_2\sigma_1\sigma_2-\sigma_1^2(L_2+\sigma_2).
            \end{align}}

            Using $L_1,~L_2>0$ and (\ref{Case 2a 3}), we say
            \begin{equation}\label{Case 2a 4}
                L_2\sigma_1\sigma_2-\sigma_1^2(L_2+\sigma_2)\ge0.
            \end{equation}

            If $\sigma_1=0$, then from (\ref{Case 2a 3}) either $L_1=0$ or $L_2=0$, which is not possible since $L_1+\sigma_1>0,~L_2+\sigma_2>0.$ So, $\sigma_1<0.$

            Then, from (\ref{Case 2a 4}), we get
            \begin{align}\label{Case 2a 5}
                \notag
                &L_2\sigma_2-\sigma_1(L_2+\sigma_2)\le0\\\notag
                \Rightarrow\;&\sigma_1(L_2+\sigma_2)\ge L_2\sigma_2\\
                \Rightarrow\;&\sigma_1\ge\frac{L_2\sigma_2}{L_2+\sigma_2}.
            \end{align}
Now, from (\ref{Case 2a 3}), we get
            \begin{align}\label{Case 2a 6}
                \notag
                &2L_1L_2^2+L_1L_2\sigma_2\le L_2\sigma_1\sigma_2-\sigma_1^2(L_2+\sigma_2)\\
                \Rightarrow\;&L_1\le \frac{1}{L_2(2L_2+\sigma_2)}(L_2\sigma_1\sigma_2-L_2\sigma_1^2-\sigma_1^2\sigma_2).
            \end{align}
        \end{itemize}
    \end{itemize}

    Hence, we show 
    \begin{align}\label{SD}
        \notag
        \varphi^{\gamma}(x^{k+1})&\le\varphi^{\gamma}(x^k)+C\|y^{k+1}-y^k\|^2\;\;\;\;\;\text{where }C<0\\
        &=\varphi^{\gamma}(x^k)+C\|\mathrm{prox}_{\gamma g}(x^{k+1})-\mathrm{prox}_{\gamma g}(x^k)\|^2\;\;\;\;\;[\text{since }y^k=\mathrm{prox}_{\gamma g}(x^k)].
    \end{align}
    
    Now, using $\frac{1}{1+\gamma L_g}$-strong monotonicity of $\mathrm{prox}_{\gamma g}$ in (\ref{SD}), we get
    \begin{align}\label{Sufficient Decrease}
        \varphi^{\gamma}(x^{k+1})&\le\varphi^{\gamma}(x^k)+\frac{C}{(1+\gamma L_g)^2}\|x^{k+1}-x^k\|^2\;\;\;\;\;\text{where }C<0.
    \end{align}
    
    Hence, we have proved.
\end{proof}

Now, the next lemma establishes an upper bounded property of the FISTA-type acceleration term, which depends on the sequence $\{t^k\}$, which is used in the proof of Theorem \ref{FISTA_SD_TH}.

\begin{lem}\label{t^k}
    In Algorithm \ref{FISTA-DYS}, we have $t^{k+1}=\frac{1+\sqrt{1+4(t^k)^2}}{2},\;\forall k\ge1$ and $t^0=1$. Then the sequence $t^k$ is strictly increasing and $\frac{t^k-1}{t^{k+1}}<1.$
\end{lem}
\begin{proof}
    We have $t^{k+1}=\frac{1+\sqrt{1+4(t^k)^2}}{2},\;\forall k\ge0$ and $t^0=1$. If $k=0$, then $t^1=\frac{1+\sqrt{5}}{2}>1=t^0.$ Now, $\forall k>0$
    \begin{align*}
        t^{k+1}-t^k&=\frac{1+\sqrt{1+4(t^k)^2}}{2}-t^k\\
        &=\frac{1+\sqrt{1+4(t^k)^2}-2t^k}{2}\\
        &>\frac{1+2t^k-2t^k}{2}>\frac{1}{2}>0.
    \end{align*}
    
    So, $t^k$ is strictly increasing and $t^{k+1}>t^k>t^k-1\;\Rightarrow\;\frac{t^k-1}{t^{k+1}}<1.$ 
\end{proof}

Now, using Lemma \ref{DYS-SD} and Lemma \ref{t^k}, we derive the corresponding sufficient decrease property for Algorithm \ref{FISTA-DYS} under some conditions, which ensures that Algorithm \ref{FISTA-DYS} decreases in every step asymptotically.

\begin{thm}[Sufficient Decrease for FISTA-type DYS]\label{FISTA_SD_TH}
    Consider that Assumption \ref{Assumption 1} is satisfied, $\gamma<\frac{1}{\beta_f}$ and consider $(y^k,~z^k,~x^{k+1})\in \text{FISTA-DYS}(x^k)$. Then for $C<0$
    $$\varphi^{\gamma}(x^{k+1})\le\varphi^{\gamma}(x^k)+\frac{C}{(1+\gamma L_g)^2}\|x^{k+1}-x^k\|^2,$$
 if the following hold:
    \begin{itemize}
        \item[(i)] $0<\gamma<\frac{1}{L_1},~\frac{4L_2}{L_1+L_2}<\lambda<2\left(1+\frac{\sigma_1}{L_1}\right),~L_1>\frac{L_2-\sigma_1+\sqrt{(L_2-\sigma_1)^2-4L_2\sigma_1}}{2}.$
        \item[(ii)] $0<\gamma<\frac{1}{L_1},~2\gamma L_1\left(1+\frac{\sigma_1}{L_1}\right)<\lambda\le\frac{\sigma_1+2L_2}{\frac{\sigma_1L_1}{2(L_1+\sigma_1)}+\frac{\sigma_2L_2}{2(L_2+\sigma_2)}},~\sigma_1\le-2L_2,~L_1>\max\left\{|\sigma_1|,\frac{\sigma_1\sigma_2}{2L_2+\sigma_2}\right\}.$
        \item[(iii)] $0<\gamma<\frac{1}{L_1+\sigma_1},~2\left(1+\frac{\sigma_1}{L_1}\right)\le\lambda\le\frac{2}{1-\frac{\sigma_1L_1}{(L_1+\sigma_1)^2}-\frac{\sigma_2L_2}{(L_1+\sigma_1)(L_2+\sigma_2)}+\frac{2L_1L_2}{(L_1+\sigma_1)^2}},~0>\sigma_1\ge\frac{\sigma_2L_2}{L_2+\sigma_2},~0<L_1\le\frac{1}{L_2(2L_2+\sigma_2)}(L_2\sigma_1\sigma_2-L_2\sigma_1^2-\sigma_1^2\sigma_2).$
    \end{itemize}

    Here, $\sigma_1=\min\{\sigma_g,~0\},~L_1\ge L_g$ s.t. $L_1+\sigma_1>0$ and $\sigma_2=\min\{\sigma_h,~0\},~L_2\ge L_h$ s.t. $L_2+\sigma_2>0.$
\end{thm}
\begin{proof}
 From Algorithm \ref{FISTA-DYS} and (\ref{sufficient decrease 3}), we get
    \begin{align}\label{F-SD 1}
        \notag
        \varphi^{\gamma}(x^{k+1})&\le\varphi^{\gamma}(x^k)-\rho_g(y^k,y^{k+1})+\langle w^k-y^k,\nabla g(y^{k+1})-\nabla g(y^k)\rangle-\rho_h(y^k,y^{k+1})\\\notag
        &\qquad+\langle w^k-y^k,\nabla h(y^{k+1})-\nabla h(y^k)\rangle+\frac{1}{2\gamma}\|y^k-y^{k+1}\|^2\\
        &\qquad+\frac{1}{\gamma}\langle w^k-y^k,y^k-y^{k+1}\rangle,
    \end{align}
 where
    \begin{align}\label{F-Envelope}
        \varphi^{\gamma}(x^k)&=f(w^k)+g(y^k)+h(y^k)+\langle w^k-y^k,\nabla g(y^k)+\nabla h(y^k)\rangle+\frac{1}{2\gamma}\|w^k-y^k\|^2.
    \end{align}

    Again, from the Algorithm \ref{FISTA-DYS}, we get
    \allowdisplaybreaks{
    \begin{align}\label{residual_fista}
        \notag
        &x^{k+1}=x^k+\lambda(z^k-y^k)\\\notag
        \Rightarrow\;&x^{k+1}=x^k+\lambda\left(w^k+\frac{t^k-1}{t^{k+1}}(w^k-w^{k-1})-y^k\right)\\\notag
        \Rightarrow\;&w^k-y^k=\frac{1}{\lambda}(x^{k+1}-x^k)-\frac{t^k-1}{t^{k+1}}(w^k-w^{k-1})\\
        \Rightarrow\;&w^k-y^k=\frac{1}{\lambda}(y^{k+1}-y^k)+\frac{\gamma}{\lambda}(\nabla g(y^{k+1})-\nabla g(y^k))-\frac{t^k-1}{t^{k+1}}(w^k-w^{k-1})\\\notag
        &\qquad\qquad\qquad\qquad\qquad[\text{since }y^k=\mathrm{prox}_{\gamma g}(x^k)\;\text{i.e., }x^k=y^k+\gamma\nabla g(y^k)].
    \end{align}}

    Now, substituting (\ref{residual_fista}) in (\ref{F-SD 1}), we get
    {\allowdisplaybreaks
    \begin{align}
        \notag
        \varphi^{\gamma}(x^{k+1})&\le\varphi^{\gamma}(x^k)-\rho_g(y^k,y^{k+1})-\rho_h(y^k,y^{k+1})+\frac{1}{2\gamma}\|y^k-y^{k+1}\|^2\\\notag
        &\qquad+\bigg\langle\frac{1}{\lambda}(y^{k+1}-y^k)+\frac{\gamma}{\lambda}(\nabla g(y^{k+1})-\nabla g(y^k))-\frac{t^k-1}{t^{k+1}}(w^k-w^{k-1}),\\\notag
        &\qquad\nabla g(y^{k+1})-\nabla g(y^k)\bigg\rangle+\bigg\langle\frac{1}{\lambda}(y^{k+1}-y^k)+\frac{\gamma}{\lambda}(\nabla g(y^{k+1})-\nabla g(y^k))\\\notag
        &\qquad-\frac{t^k-1}{t^{k+1}}(w^k-w^{k-1}),\nabla h(y^{k+1})-\nabla h(y^k)\bigg\rangle+\bigg\langle\frac{1}{\lambda}(y^{k+1}-y^k)\\\notag
        &\qquad+\frac{\gamma}{\lambda}(\nabla g(y^{k+1})-\nabla g(y^k))-\frac{t^k-1}{t^{k+1}}(w^k-w^{k-1}),y^k-y^{k+1}\bigg\rangle\\\notag
        \Rightarrow\varphi^{\gamma}(x^{k+1})&\le \varphi^{\gamma}(x^k)-\rho_g(y^k,y^{k+1})+\bigg\langle \frac{1}{\lambda}(y^{k+1}-y^k)+\frac{\gamma}{\lambda}(\nabla g(y^{k+1})-\nabla g(y^k)),\\\notag
        &\qquad\nabla g(y^{k+1})-\nabla g(y^k)\bigg\rangle-\rho_h(y^k,y^{k+1})+\bigg\langle \frac{1}{\lambda}(y^{k+1}-y^k)\\\notag
        &\qquad+\frac{\gamma}{\lambda}(\nabla g(y^{k+1})-\nabla g(y^k)),\nabla h(y^{k+1})-\nabla h(y^k)\bigg\rangle+ \frac{1}{2\gamma}\|y^k-y^{k+1}\|^2\\\notag
        &\qquad+\frac{1}{\gamma}\left\langle \frac{1}{\lambda}(y^{k+1}-y^k)+\frac{\gamma}{\lambda}(\nabla g(y^{k+1})-\nabla g(y^k)), y^k-y^{k+1}\right\rangle\\\notag
        &\qquad-\frac{t^k-1}{t^{k+1}}\langle w^k-w^{k-1},\nabla g(y^{k+1})-\nabla g(y^k)+\nabla h(y^{k+1})-\nabla h(y^k)\\
        &\qquad+y^k-y^{k+1}\rangle.\label{F-SD 2}
    \end{align}}

    In Lemma \ref{DYS-SD}, we show that (\ref{sufficient decrease 4}) implies (\ref{SD}) under some conditions, which are discussed. Using this concept, we say, under the same conditions as discussed in Lemma \ref{DYS-SD}, that
    \begin{align}
        \notag
        \varphi^{\gamma}(x^{k+1})-\varphi^{\gamma}(x^k)&\le C\|y^{k+1}-y^k\|^2-\frac{t^k-1}{t^{k+1}}\langle w^k-w^{k-1},\nabla g(y^{k+1})-\nabla g(y^k)\\\notag
        &\qquad+\nabla h(y^{k+1})-\nabla h(y^k)+y^k-y^{k+1}\rangle,\qquad\text{where }C<0
    \end{align}
    \begin{align}\label{F-SD 3}
        \notag
        \Rightarrow\; \varphi^{\gamma}(x^{k+1})-\varphi^{\gamma}(x^k)&\le C\|y^{k+1}-y^k\|^2+\frac{\eta(t^k-1)}{2t^{k+1}}\|w^k-w^{k-1}\|^2\\\notag
        &\qquad+\frac{t^k-1}{2\eta t^{k+1}}\|\nabla g(y^{k+1})-\nabla g(y^k)+\nabla h(y^{k+1})-\nabla h(y^k)\\
        &\qquad-y^{k+1}+y^k\|^2,\qquad\text{where }\eta>0.
    \end{align}

    Now, using the triangle inequality of the norm and then the $L_g,~L_h$-smoothness of $g$ and $h$ respectively in $\|\nabla g(y^{k+1})-\nabla g(y^k)+\nabla h(y^{k+1})-\nabla h(y^k)-y^{k+1}+y^k\|$, we obtain
    \begin{align}\label{F-SD 4}
        \notag
        &\quad\|\nabla g(y^{k+1})-\nabla g(y^k)+\nabla h(y^{k+1})-\nabla h(y^k)-y^{k+1}+y^k\|\\\notag
        &\le \|\nabla g(y^{k+1})-\nabla g(y^k)\|+\|\nabla h(y^{k+1})-\nabla h(y^k)\|+\|y^{k+1}-y^k\|\\
        &\le (L_g+L_h+1)\|y^{k+1}-y^k\|.
    \end{align}

    Now, we try to find an upper bound of $\|w^k-w^{k-1}\|$. Using \cite[Lemma 3.1]{liu2019envelope} for $\mathrm{prox}_{\gamma f}$ and then the triangle inequality of norm and $L_g,~L_h$-smoothness of $g$ and $h$ respectively, for $\gamma<\frac{1}{\beta_f}$, we get
    {\allowdisplaybreaks
    \begin{align}\label{F-SD 5}
        \notag
        \|w^k-w^{k-1}\|&\le\frac{1}{1-\gamma \beta_f}\|y^k-\gamma\nabla g(y^k)-\gamma\nabla h(y^k)-y^{k-1}+\gamma\nabla g(y^{k-1})\\\notag
        &\qquad+\gamma\nabla h(y^{k-1})\|\\\notag
        &\le\frac{1}{1-\gamma \beta_f}\left(\|y^k-y^{k-1}\|+\gamma L_g\|y^k-y^{k-1}\|+\gamma L_h\|y^k-y^{k-1}\|\right)\\
        \Rightarrow\;\|w^k-w^{k-1}\|&\le\frac{1+\gamma L_g+\gamma L_h}{1-\gamma\beta_f}\|y^k-y^{k-1}\|.
    \end{align}}

    After that, substituting (\ref{F-SD 4}) and (\ref{F-SD 5}) in (\ref{F-SD 3}), we deduce
    \begin{align}\label{F-SD 6}
        \notag
        &\varphi^{\gamma}(x^{k+1})-\varphi^{\gamma}(x^k)\\\notag
        &\le C\|y^{k+1}-y^k\|^2+\frac{\eta(t^k-1)}{2t^{k+1}}\frac{(1+\gamma L_g+\gamma L_h)^2}{(1-\gamma\beta_f)^2}\|y^k-y^{k-1}\|^2\\\notag
        &\qquad+\frac{t^k-1}{2\eta t^{k+1}}(L_g+L_h+1)^2\|y^{k+1}-y^k\|^2\\\notag
        &= C\|y^{k+1}-y^k\|^2+\frac{\eta(t^k-1)}{2t^{k+1}}\frac{(1-\gamma +\gamma L)^2}{(1-\gamma\beta_f)^2}\|y^k-y^{k-1}\|^2+\frac{L^2(t^k-1)}{2\eta t^{k+1}}\|y^{k+1}-y^k\|^2\\\notag
        &\qquad\qquad\qquad\qquad\qquad\qquad\qquad\qquad\qquad\qquad\qquad\qquad\qquad[\text{Let }L=1+L_g+L_h]\\
        &=\left(C+\frac{L^2(t^k-1)}{2\eta t^{k+1}}\right)\|y^{k+1}-y^k\|^2+\frac{\eta(t^k-1)}{2t^{k+1}}\frac{(1-\gamma +\gamma L)^2}{(1-\gamma\beta_f)^2}\|y^k-y^{k-1}\|^2.
    \end{align}

    Now, (\ref{F-SD 6}) implies that
    {\allowdisplaybreaks
    \begin{align}
        \notag
        &\left(|C|-\frac{L^2(t^k-1)}{2\eta t^{k+1}}\right)\|y^{k+1}-y^k\|^2-\frac{\eta(t^k-1)}{2t^{k+1}}\frac{(1-\gamma +\gamma L)^2}{(1-\gamma\beta_f)^2}\|y^k-y^{k-1}\|^2\\\notag
        &\qquad\qquad\qquad\qquad\qquad\qquad\qquad\qquad\qquad\qquad\qquad\qquad\quad\le\varphi^{\gamma}(x^k)-\varphi^{\gamma}(x^{k+1})\\\notag
        \Rightarrow\;&\left(|C|-\frac{L^2}{2\eta}\right)\|y^{k+1}-y^k\|^2-\frac{\eta}{2}\frac{(1-\gamma +\gamma L)^2}{(1-\gamma\beta_f)^2}\|y^k-y^{k-1}\|^2\le\varphi^{\gamma}(x^k)-\varphi^{\gamma}(x^{k+1})\\\notag
        &\qquad\qquad\qquad\qquad\qquad\qquad\qquad\qquad\qquad\qquad\qquad\qquad[\text{By using Lemma \ref{t^k}}]\\\notag
        \Rightarrow\;&\left(|C|-\frac{L^2}{2\eta }\right)\sum_{k\in\mathbb{N}}\|y^{k+1}-y^k\|^2-\frac{\eta}{2}\frac{(1-\gamma +\gamma L)^2}{(1-\gamma\beta_f)^2}\sum_{k\in\mathbb{N}}\|y^k-y^{k-1}\|^2\\\notag
        &\qquad\qquad\qquad\qquad\qquad\qquad\qquad\qquad\qquad\qquad\qquad\le\sum_{k\in\mathbb{N}}[\varphi^{\gamma}(x^k)-\varphi^{\gamma}(x^{k+1})]\\\notag
        &\qquad\qquad\qquad\qquad\qquad\qquad\qquad\qquad\qquad\qquad\qquad\le\varphi^{\gamma}(x^1)-\inf\varphi^{\gamma}\\\notag
        &\qquad\qquad\qquad\qquad\qquad\qquad\qquad\qquad\qquad\qquad\qquad=\varphi^{\gamma}(x^1)-\inf\varphi\\\notag
        &\qquad\qquad\qquad\qquad\qquad\qquad\qquad\qquad\qquad\qquad\qquad\qquad[\text{Since, }\inf\varphi^{\gamma}=\inf\varphi]\\\notag
        \Rightarrow\;&\left[|C|-\frac{L^2}{2\eta}-\frac{\eta}{2}\frac{(1-\gamma +\gamma L)^2}{(1-\gamma\beta_f)^2}\right] \sum_{k\in\mathbb{N}} \|y^{k+1}-y^k\|^2 \le\varphi^{\gamma}(x^1)-\inf\varphi\\
        &\qquad\qquad\qquad\qquad\qquad\qquad\qquad\qquad\qquad\qquad\qquad+\frac{\eta}{2}(1-\gamma +\gamma L)^2\|y^1-y^0\|^2.
    \end{align}}

    We know that $-\infty<\inf \varphi<+\infty$ by Assumption \ref{Assumption 1} and $\varphi^{\gamma}(x^1),~\frac{\eta}{2}(1-\gamma +\gamma L)^2\|y^1-y^0\|^2$ are real constants, then we say that
    \begin{equation}\label{F-SD 7}
        \left[|C|-\frac{L^2}{2\eta }-\frac{\eta}{2}\frac{(1-\gamma +\gamma L)^2}{(1-\gamma\beta_f)^2}\right]\sum_{k\in\mathbb{N}}\|y^{k+1}-y^k\|^2<+\infty.
    \end{equation}

    Now, if we choose an $\eta>0$ such that $|C|-\frac{L^2}{2\eta}-\frac{\eta}{2}\frac{(1-\gamma +\gamma L)^2}{(1-\gamma\beta_f)^2}>0$, then from (\ref{F-SD 7}), we say 
    \begin{align}\label{F-SD 8}
        \sum_{k\in\mathbb{N}}\|y^{k+1}-y^k\|^2<+\infty.
    \end{align}

    Using (\ref{F-SD 8}), we consider $\frac{L^2(t^k-1)}{2\eta t^{k+1}}\|y^{k+1}-y^k\|^2$ and $\frac{\eta(t^k-1)}{2t^{k+1}}\frac{(1-\gamma +\gamma L)^2}{(1-\gamma\beta_f)^2}\|y^k-y^{k-1}\|^2$ as summable error terms and from (\ref{F-SD 5}) we say essentially the following inequality is true: 
    \begin{align}\label{F_SD 9}
        \varphi^{\gamma}(x^{k+1})&\le\varphi^{\gamma}(x^k)+C\|y^{k+1}-y^k\|^2\;\;\;\;\;\text{where }C<0.
    \end{align}
    
    Now, using a similar approach, which is used to show that (\ref{SD}) implies (\ref{Sufficient Decrease}), here we obtain that (\ref{F_SD 9}) implies 

    \begin{align}\label{FISTA_Sufficient Decrease}
        \varphi^{\gamma}(x^{k+1})&\le\varphi^{\gamma}(x^k)+\frac{C}{(1+\gamma L_g)^2}\|x^{k+1}-x^k\|^2\;\;\;\;\;\text{where }C<0.
    \end{align}
    
    Hence, the proof is done.
\end{proof}

Now, in the next theorem, we use Theorem \ref{FISTA_SD_TH} and study the convergence results for Algorithm \ref{FISTA-DYS}. In Theorem \ref{Subsequential convergence}, we study the residual convergence, cluster points, and boundedness of the sequence $\{(x^k,y^k,w^k)\}$, generated by Algorithm \ref{FISTA-DYS}.

\begin{thm}[Subsequential convergence for FISTA-type DYS]\label{Subsequential convergence}
    Let us consider that Assumption \ref{Assumption 1} is satisfied, and consider a sequence $\{(x^k, y^k, w^k)\}$ generated by Algorithm \ref{FISTA-DYS} with stepsize $\gamma$ and relaxation $\lambda$, as in sufficient decrease section, starting from $x^0 \in \mathbb{R}^n$. The following holds:
    \begin{enumerate}
        \item[(i)] The residual $\{w^k - y^k\}$ vanishes with rate $\min_{i \leq k} \|w^i - y^i\| = o\left(\frac{1}{\sqrt{k}}\right).$
        \item[(ii)] Sequences $\{y^k\}$ and $\{w^k\}$ have the same cluster points, all of which are stationary for $\varphi$ and on which $\varphi$ has a finite value, this being the limit of $\{\varphi^\gamma(x^k)\}$. In fact, for each $k$ one has $\mathrm{dist}(0,~\hat{\partial} \varphi(w^k)) \leq \frac{2 + \gamma (L_g + L_h)}{\gamma} \|y^k - w^k\|.$ Moreover, if $\varphi$ satisfies the PL-inequality with respect to $\mu>0$ then $\varphi(w^k)-\varphi^*\in o\left(\frac{1}{k}\right),$ where $\varphi^*=\inf\varphi.$
        \item[(iii)] If $\varphi$ has bounded level sets, then the sequence $\{(x^k, y^k, w^k)\}$ is bounded.
    \end{enumerate}
\end{thm}

\begin{proof}
    Using the $\frac{1}{1+\gamma L_g}$-strong monotonicity of $\mathrm{prox}_{\gamma g}$, we get $\|x^{k+1}-x^k\|\le(1+\gamma L_g)\|y^{k+1}-y^k\|$, which implies that

    \allowdisplaybreaks{
    \begin{align}\label{subpart}
        \notag
        & \lambda \|z^{k} - y^{k}\|\le (1+\gamma L_g)\|y^{k+1} - y^{k}\|\;\;\;\;\;[\text{By using Algorithm \ref{FISTA-DYS}}]\\\notag
        \Rightarrow\;& \|z^{k} - y^{k}\|\le \frac{1+\gamma L_g}{\lambda} \|y^{k+1} - y^{k}\| \\
        \Rightarrow\;& \left\| w^{k} + \frac{t^{k}-1}{t^{k+1}}(w^{k} - w^{k-1}) - y^{k} \right\|\le \frac{1+\gamma L_g}{\lambda} \|y^{k+1} - y^{k}\|\qquad[\text{By Algorithm \ref{FISTA-DYS}}].
    \end{align}}

    Now, using the triangle inequality of the norm in \eqref{subpart}, we get
    $$\|w^{k} - y^{k}\| -\frac{t^{k}-1}{t^{k+1}} \|w^{k} - w^{k-1}\|\le\frac{1+\gamma L_g}{\lambda} \|y^{k+1} - y^{k}\|,$$
    which implies that
    
    \begin{align}\label{Th 7.3 (i)}
        \notag
        & \|w^{k} - y^{k}\| \le \frac{1+\gamma L_g}{\lambda} \|y^{k+1} - y^{k}\|+ \frac{t^{k}-1}{t^{k+1}} \|w^{k} - w^{k-1}\|\\\notag
        \Rightarrow\;& \|w^{k} - y^{k}\|^2 \le \frac{(1+\gamma L_g)^2}{\lambda^2} \|y^{k+1} - y^{k}\|^2 + \left(\frac{t^{k}-1}{t^{k+1}}\right)^2 \|w^{k} - w^{k-1}\|^2 \\\notag
        &\qquad\qquad\qquad + 2 \frac{1+\gamma L_g}{\lambda} \frac{t^{k}-1}{t^{k+1}} \|y^{k+1} - y^{k}\| \|w^{k} - w^{k-1}\| \\\notag
        \Rightarrow\;& \|w^{k} - y^{k}\|^2 \le \frac{(1+\gamma L_g)^2}{\lambda^2} \|y^{k+1} - y^{k}\|^2 + \|w^{k} - w^{k-1}\|^2 \\\notag
        &\qquad\qquad\qquad+ 2 \frac{1+\gamma L_g}{\lambda} \|y^{k+1} - y^{k}\| \|w^{k} - w^{k-1}\|\qquad [\text{By using Lemma \ref{t^k}}]\\\notag
        \Rightarrow\;& \|w^{k} - y^{k}\|^2 \le \frac{2(1+\gamma L_g)^2}{\lambda^2} \|y^{k+1} - y^{k}\|^2 + 2\|w^{k} - w^{k-1}\|^2\\
        \Rightarrow\;& \sum_{k\in\mathbb{N}} \|w^{k} - y^{k}\|^2 \le \frac{2(1+\gamma L_g)^2}{\lambda^2} \sum_{k\in\mathbb{N}} \|y^{k+1} - y^{k}\|^2 + 2\sum_{k\in\mathbb{N}}\|w^{k} - w^{k-1}\|^2.
    \end{align}

    Now, from (\ref{F-SD 5}), (\ref{F-SD 8}) and (\ref{Th 7.3 (i)}) we write $\sum_{k\in\mathbb{N}} \|w^{k} - y^{k}\|^2 < +\infty$ and hence, $\lim_{k\to+\infty}\|w^{k}-y^{k}\|^2=0,\;\text{i.e., }\lim_{k\to+\infty}\|w^{k}-y^{k}\|=0,\;\text{i.e., }\lim_{k\to+\infty}(w^{k}-y^{k})=0$ and, we get $\min_{i\le k}\|w^{k} - y^{k}\| = o\left(\frac{1}{\sqrt{k}}\right).$

    From Theorem \ref{Subsequential convergence} \textit{(i)}, $\lim_{k\to+\infty}(w^{k}-y^{k})=0\;\Rightarrow\;\lim_{k\to+\infty}w^k=\lim_{k\to+\infty}y^k.$

    Let $\lim_{K\ni k\to+\infty}w^k=\lim_{K\ni k\to+\infty}y^k=\tilde{y}$ and $\lim_{K\ni k\to+\infty}x^k=\tilde{x}$ for some $K\subseteq\mathbb{N}.$

    Now,
    \begin{align*}
        &y^k=\mathrm{prox}_{\gamma g}(x^k)\\
        \Rightarrow\;&x^k=y^k+\gamma\nabla g(y^k)\;\to\;\tilde{y}+\gamma\nabla g(\tilde{y})\text{ as }K\ni k\;\to\;+\infty\\
        &\qquad\qquad\qquad\qquad\qquad\qquad[\text{since $\nabla g$ is continuous as $g$ is $L_g$-smooth}].
    \end{align*}

    So, $\tilde{x}=\tilde{y}+\gamma\nabla g(\tilde{y})\;\Rightarrow\;\tilde{y}=\mathrm{prox}_{\gamma g}(\tilde{x}).$
    From Assumption \ref{Assumption 1}, $f$ is proper and closed, so by using \cite[Theorem 2.6]{beck2017first} $f$ is proper and lower semicontinuous and hence, by using \cite[Example 5.23]{rockafellar2009variational} $\mathrm{prox}_{\gamma g}$ is outer semicontinuous. Then, $\tilde{y}=\lim_{K\ni k\to+\infty}w^k\in\limsup_{K\ni k\to+\infty}\mathrm{prox}_{\gamma f}(2y^k-x^k-\gamma\nabla h(y^k))\subseteq\mathrm{prox}_{\gamma f}(2\tilde{y}-\tilde{x}-\gamma\nabla h(\tilde{y}))$, which implies that
    \begin{align*}
        &\tilde{y}\in \mathrm{prox}_{\gamma f}(\tilde{y}-\gamma\nabla g(\tilde{y})-\gamma\nabla h(\tilde{y}))\;\;\;\;\;[\text{Using }\tilde{x}=\tilde{y}+\gamma\nabla h(\tilde{y})]\\
        \Rightarrow\;&\frac{1}{\gamma}[\tilde{y}-\gamma\nabla g(\tilde{y})-\gamma \nabla h(\tilde{y})-\tilde{y}]\in\hat{\partial}f(\tilde{y})\\
        \Rightarrow\;&-\nabla g(\tilde{y})-\nabla h(\tilde{y})\in\hat{\partial}f(\tilde{y})\\
        \Rightarrow\;&0\in\hat{\partial}f(\tilde{y})+\nabla g(\tilde{y})+\nabla h(\tilde{y})\\
        \Rightarrow\;&0\in\hat{\partial}\varphi(\tilde{y}).
    \end{align*}
    
    So, $\tilde{y}$ is a stationary point of $\varphi.$

    By Assumption \ref{Assumption 1}, we say $\varphi$ is lower semicontinuous and we also have that, $\lim_{K\ni k\to+\infty}w^k=\tilde{y}$. Then, we get
    \begin{align*}
        \varphi(\tilde{y})&\le\liminf_{K\ni k\to+\infty}\varphi(w^k)\;\;\;\;\;[\text{By lower semicontinuity of }\varphi]\\
        &\le\limsup_{K\ni k\to+\infty} \varphi(w^k)\\
        &\le\limsup_{K\ni k\to+\infty} \left[\varphi^{\gamma}(x^k)-\frac{1-\gamma(L_g+L_h)}{2\gamma}\|w^k-y^k\|^2\right]\;\;\;\;\;[\text{By Proposition \ref{Sandwiching property}}]\\
        &=\limsup_{K\ni k\to+\infty}\varphi^{\gamma}(x^k)\;\;\;\;\;[\text{By Theorem \ref{Subsequential convergence} \textit{(i)}}]\\
        &=\varphi^{\gamma}(\tilde{x})\;\;\;\;\;[\text{Using Proposition \ref{Strict continuity}}]\\
        &\le\varphi(\tilde{y})\;\;\;\;\;[\text{By Proposition \ref{Sandwiching property}}].
    \end{align*}

    Thus, we write $\lim_{K\ni k\to+\infty}\varphi(w^k)=\varphi(\tilde{y})=\varphi^{\gamma}(\tilde{x}).$

    Now, since $\varphi^{\gamma}(x^k)\to\inf\varphi=\varphi^*\;\text{as }K\ni k\to+\infty$, then $\varphi(\tilde{y})=\varphi^{\gamma}(\tilde{x})=\varphi^*<+\infty$ [By Assumption \ref{Assumption 1}].

    Now, from Algorithm \ref{FISTA-DYS} we have
    \begin{align*}
        &w^k\in \mathrm{prox}_{\gamma f}(2y^k-x^k-\gamma\nabla h(y^k))\\
        \Rightarrow\;&\frac{1}{\gamma}[2y^k-x^k-\gamma\nabla h(y^k)-w^k]\in\hat{\partial}f(w^k)\\
        \Rightarrow\;&\frac{1}{\gamma}(y^k-w^k)-\nabla g(y^k)-\nabla h(y^k)\in\hat{\partial}f(w^k)\;\;\;\;\;[\text{Since }x^k=y^k+\gamma\nabla g(y^k)]\\
        \Rightarrow\;&\frac{1}{\gamma}(y^k-w^k)+\nabla g(w^k)+\nabla h(w^k)-\nabla g(y^k)-\nabla h(y^k)\in \hat{\partial}\varphi(w^k).
    \end{align*}

    So, 
    \begin{align}\label{Th 7.4 (iii)}
        \mathrm{dist}(0,\hat{\partial}\varphi(w^k))&\le\left\|\frac{1}{\gamma}(y^k-w^k)+(\nabla g(w^k)-\nabla g(y^k))+(\nabla h(w^k)-\nabla h(y^k))\right\|.
    \end{align}

    By using triangle inequality of norm and $L_g,~L_h$-smoothness of $g$ and $h$ respectively in (\ref{Th 7.4 (iii)}), we get
    \begin{align*}
        \mathrm{dist}(0,\hat{\partial}\varphi(w^k))&\le\frac{1}{\gamma}\|y^k-w^k\|+L_g\|y^k-w^k\|+L_h\|y^k-w^k\|\\
        &\le\frac{1+\gamma(L_g+L_h)}{\gamma}\|y^k-w^k\|= o\left(\frac{1}{\sqrt{k}}\right).\;\;\;\;\;[\text{By using Theorem \ref{Subsequential convergence}}]
    \end{align*}

    If $\varphi$ satisfies PL-inequality with respect to $\mu>0$ then we write $$\varphi(w^k)-\varphi^*\le \frac{1}{2\mu}\mathrm{dist}^2(0,\hat{\partial}\varphi(w^k))= o\left(\frac{1}{k}\right).$$

    Let $\varphi$ have bounded level sets. Then, by Proposition \ref{Minimization and level-boundedness equivalence} \textit{(iii)} $\varphi^{\gamma}$ has bounded level sets.

    Now, by Theorem \ref{FISTA_SD_TH}, we have
    \allowdisplaybreaks{\begin{align*}
        &\varphi^{\gamma}(x^k)\le\varphi^{\gamma}(x^0),\;\;\forall\;k\in\mathbb{N}\cup\{0\}\\
        \Rightarrow\;&x^k\in\mathrm{lev}_{\le\varphi^{\gamma}(x^0)}\varphi^{\gamma},\;\;\forall\;k\in\mathbb{N}\cup\{0\}\\
        \Rightarrow\;& \{x^k\}\text{ is bounded.}
    \end{align*}}

    Then, by Lipschitz continuity of $\mathrm{prox}_{\gamma g}$, we say $\{y^k\}$ is bounded and hence by Theorem \ref{Subsequential convergence} \textit{(i)}, $\{w^k\}$ is bounded.
\end{proof}

The KL property is a very useful tool for analyzing the global convergence of a descent method. In Proposition \ref{Relation between semialgebraic and KL}, we see how a proper, lower semicontinuous, and semialgebraic function relates to the KL property. In \cite[Theorem 2]{li2016douglas}, we see how global convergence of DRS by a merit function is proved using a semialgebraic condition. After that, using this theorem in \cite[Theorem 4.4]{themelis2020douglas}, global convergence of DRS by DRE is shown using a semialgebraic condition for functions. Now, similar to \cite[Theorem 3.10]{wu2024extrapolated}, the global convergence of DYS by a merit function is proved using the KL property of the total objective function. Here, we consider $f,~g,$ and $h$ to be semialgebraic, and then we say $\varphi=f+g+h$ is also a semialgebraic function; and hence, by Proposition \ref{Relation between semialgebraic and KL}, $\varphi$ is a KL function. Now, in Theorem \ref{FISTA_SD_TH} we show the decent property of $\varphi^{\gamma}$ under some conditions and in Theorem \ref{Subsequential convergence} we show the bound $\mathrm{dist}(0,~\hat{\partial} \varphi(w^k)) \leq \frac{2 + \gamma (L_g + L_h)}{\gamma} \|y^k - w^k\|$ for all $k$. In these theorems, we use tighter conditions related to \cite[Theorem 3.10]{wu2024extrapolated}, and we see that Theorem \ref{Global convergence for FISTA-type DYS} is also an extended result of \cite[Theorem 3.10]{wu2024extrapolated}.

\begin{thm}[Global convergence for FISTA-type DYS]\label{Global convergence for FISTA-type DYS}
    Let us consider that, Assumption \ref{Assumption 1} is satisfied, $\varphi$ is level bounded, and $f,~g,$ and $h$ are semialgebraic functions. Also if we consider a sequence $\{(x^k,~y^k,~w^k)$ generated by Algorithm \ref{FISTA-DYS} with all convergence conditions from Theorem \ref{FISTA_SD_TH} then the sequences $\{y^k\}$ and $\{w^k\}$ are convergent and they converge to stationary point of $\varphi.$
\end{thm}

\section{FISTA-type DYS with Line-search}\label{FISTA-type DYS with Line-search section}
In Section \ref{FISTA-type DYS section}, we see how we use a FISTA-type acceleration term in DYS to increase the convergence speed. Now, here we try to use a line-search algorithm in our FISTA-type DYS algorithm to make our previous algorithm faster. To improve the local convergence behavior Newton-type scheme is additionally added with the FISTA-type acceleration term. The Newton-type step is used to update a better direction using higher-order information to make the convergence speed faster near a solution point. So, the combination of FISTA and Newton-type steps make the DYS algorithm much faster. Line-search is used to control the update step with the Newton-type directions and reach a stable optimal result. In classical line-search, we need the differentiability property to hold for DYE. But, from our Assumption \ref{Assumption 1} and \eqref{DYE}, we say that, DYE may be non-differentiable. So, to tackle this difficulty, we use a line-search algorithm based on the CLyD method instead of the classical line-search as discussed in \cite{themelis2018,themelis2022douglas}, because the requirements of our line-search algorithm are only the continuity and sufficient decrease property of DYE, which are shown in Proposition \ref{Strict continuity} and Theorem \ref{FISTA_SD_TH}, respectively. An another major advantage of the CLyD-based line-search algorithm is that we can use any update direction $d$ (need not be a descent direction). If the direction $d$ is very good then we take the unit stepsize and use the full direction $d$; otherwise, by backtracking, we change the stepsize and try to get a better direction using the sufficient decrease of the DYE property. In \cite{themelis2022douglas}, a continuous-based line-search method is discussed for DRS and ADMM. We use this approach in our FISTA-type DYS algorithm.

In the Algorithm \ref{FISTA-DYS-LS}, the FISTA-type DYS step is used to calculate the DYE ($\varphi^{\gamma}$) and check the sufficient decrease property. So, the computational cost of the Algorithm \ref{FISTA-DYS-LS} depends on the number of backtracks. If the sufficient decrease property is not satisfied too many times then the number of backtracking will be very large, and then the computational cost will be very high though the total number of iteration of Algorithm \ref{FISTA-DYS-LS} is small. So, to tackle this issue, we fix maximum number of backtracking ($i_{\text{max}}$). If, for some direction, the sufficient decrease property is not satisfied and the maximum backtracking limit is reached then we discard the direction and choose the standard FISTA-type DYS as an update. In some cases, where $\mathrm{prox}_{\gamma f}$ is easy to evaluate (like projection onto simple sets, thresholding, etc.), then it takes minimal backtracking, which reduces the computational cost significantly and converges very quickly.

Although any update direction can be used in our method, but computational cost and convergence speed heavily depend on the update direction. In this work, we use the quasi-Newton method with the modified Broyden update rule, which is used in \cite{themelis2022douglas} for DRS and ADMM to obtain a better update direction. For Newton-type convergence, we have to calculate the exact Jacobian or higher-order information, which is very difficult. So, to avoid this problem, here we use the quasi-Newton method, where we use an approximated Jacobian or approximated higher-order information and use the modified Broyden update rule to evaluate the approximated Jacobian or approximated higher-order information.

\begin{itemize}
    \item[$\bullet$] \textbf{Quasi-Newton FISTA-type DYS :} The update rule is
    \begin{equation}\label{Quasi-Newton FISTA-DYS direction}
        d^k=-H^kr^k,\;p^k=d^k\;\text{and } q^k=r_0^{k+1}-r^k,
    \end{equation}
    where $r_0^{k+1}=y_0^{k+1}-w_0^{k+1},\;(y_0^{k+1},w_0^{k+1})\in\text{FISTA-type DYS}_{\gamma}(x^k+d^k)$ and $r^k=y^k-w^k,\;(y^k,w^k)\in\text{FISTA-type DYS}_{\gamma}(x^k).$
    
    \item[$\bullet$] \textbf{Modified Broyden \cite{themelis2022douglas} :} Fix $\theta\in(0,1)$ and $H^0$ as invertible matrix. Then the update rule is 
    \begin{equation}\label{Jacobian Modified Broyden}
        H^{k+1}=H^k+\frac{p^k-H^kq^k}{\left\langle p^k,\left(\frac{1}{\theta^k}-1\right)p^k+H^kq^k\right\rangle}(p^k)^{\top}H^k,
    \end{equation}
    where, 
    \begin{align*}
        \theta^k=\begin{cases}
        1, & \text{if } |\delta^k|\ge\theta\\
        \frac{1-\mathrm{sgn}(\delta^k)\theta}{1-\delta^k}, & \text{if } |\delta^k|<\theta
        \end{cases}
        \quad \text{and} \quad
        \delta^k=\frac{\langle H^kq^k,p^k\rangle}{\|p^k\|^2}.
    \end{align*}

    If we take $\theta^k\equiv 1$, then this modified Broyden update rule turns into the Broyden update rule \cite{broyden1965class}.
\end{itemize}

Now, we will study the FISTA-type DYS with line-search (Algorithm \ref{FISTA-DYS-LS}) and discuss the detailed convergence analysis of our algorithm.

\begin{algorithm}[h]
\caption{FISTA-type DYS with Line-search}
\begin{algorithmic}[1]

\State \textbf{Input:} $x^0=w^{-1} \in \mathbb{R}^p$, $\epsilon > 0$, relaxation parameter $\lambda >0$, max backtracks $i_{\max} \le +\infty$, step size $\gamma>0$, $c>\frac{C}{\lambda^2(1+\gamma L_g)^2},$ where $C$ is from Lemma \ref{DYS-SD}.
\State $t^0 = 1$, $k=0$, $(y^0,w^0)\in \text{FISTA-type DYS}_{\gamma}(x^0)$, and calculate $\varphi^{\gamma}(x^0).$

\While{true}

\State $r^k = y^k - w^k$
\If{$\|r^k\| \le \epsilon$}
    \State \Return $(x^k, y^k, w^k)$
\EndIf
\State $t^{k+1} = \frac{1 + \sqrt{1 + 4(t^k)^2}}{2}$
\State $\beta^k = \frac{t^k - 1}{t^{k+1}}$
\State $z^k=w^k + \beta^k (w^k - w^{k-1})$
\State $\bar{x}^{k+1} = x^k + \lambda (z^k - y^k)$ 

\State Select direction $d^k \in \mathbb{R}^p$, set $\tau^k = 1$, $i^k = 0$

\While{true}

\State $x^{k+1} = (1 - \tau^k)\bar{x}^{k+1} + \tau^k (x^k + d^k)$

\State Compute $(y^{k+1}, w^{k+1}) \in \text{FISTA-type DYS}_\gamma(x^{k+1})$
\State Evaluate $\varphi^\gamma(x^{k+1})$

\If{$\varphi^\gamma(x^{k+1}) < \varphi^\gamma(x^k) + c\|z^k-y^k\|^2$}
    \State $k \gets k+1$
    \State \textbf{break}
\ElsIf{$i^k = i_{\max}$}
    \State $x^{k+1} = \bar{x}^{k+1}$
    \State Compute $(y^{k+1}, w^{k+1}) \in \text{FISTA-type DYS}_\gamma(x^{k+1})$
    \State $k \gets k+1$
    \State \textbf{break}
\Else
    \State $\tau^k \gets \tau^k / 2$, \quad $i^k \gets i^k + 1$
\EndIf

\EndWhile

\EndWhile

\end{algorithmic}\label{FISTA-DYS-LS}
\end{algorithm}
\subsection{Convergence Analysis for FISTA-type DYS with Line-search}\label{Convergence Analysis for FISTA-type DYS with Line-search subsection}

Before analyzing the convergence properties of Algorithm \ref{FISTA-DYS-LS}, we first check the well-defined property of Algorithm \ref{FISTA-DYS-LS}. Theorem \ref{Well definedness} shows that Algorithm \ref{FISTA-DYS-LS} takes a finite number of backtracking steps, stops after a finite number of steps, and the stationary point exists.

\begin{thm}[Well-definedness]\label{Well definedness}
    Under Assumption \ref{Assumption 1}, consider the iterates generated by Algorithm \ref{FISTA-DYS-LS}. Then, the following holds:
    \begin{enumerate}
        \item[(i)] In each iteration, the total number of backtracking steps is finite, where $i_{max}$ can be finite or infinite.
        \item[(ii)] The algorithm stops after at most $\lceil \frac{2}{c\epsilon^2}\left(\inf\varphi-\varphi^\gamma(x^0)+cM\right) \rceil$ iterations, where $\sum_{k=0}^{N-1}\|w^k-w^{k-1}\|^2=M$.
        \item[(iii)] If the last iteration is $N$, then $\text{dist}(0,\hat{\partial}\varphi(w^N))\le\frac{1+\gamma(L_g+L_h)}{\gamma}\epsilon.$
    \end{enumerate}
\end{thm}
\begin{proof}
    If $\|r^k\|=0$, then from the Step 5 of Algorithm \ref{FISTA-DYS-LS}, our algorithm will stop. Now, suppose that $\|r^k\|>0.$ From Theorem \ref{FISTA_SD_TH} and the Step 11 of Algorithm \ref{FISTA-DYS-LS}, we have 
    \begin{align}
        \notag
        \varphi^{\gamma}(\bar{x}^{k+1}) & \le \varphi^{\gamma}(x^k) +\frac{C}{(1+\gamma L_g)^2}\|\bar{x}^{k+1}-x^k\|^2\\
        & \le \varphi^{\gamma}(x^k)+\frac{C}{\lambda^2(1+\gamma L_g)^2}\|z^k-y^k\|^2,\label{Well defindness (i) 1}
    \end{align}
    for some $C<0.$ Now, in Algorithm \ref{FISTA-DYS-LS}, since $c>\frac{C}{\lambda^2(1+\gamma L_g)^2}$, so we get, $\varphi^{\gamma}(x^k)+c\|z^k-y^k\|^2>\varphi^{\gamma}(x^k)+\frac{C}{\lambda^2(1+\gamma L_g)^2}\|z^k-y^k\|^2.$ Replacing this strict inequality in \eqref{Well defindness (i) 1} we get,
    \begin{equation}\label{Well defindness (i) 2}
        \varphi^{\gamma}(\bar{x}^{k+1})<\varphi^{\gamma}(x^k)+c\|z^k-y^k\|^2.
    \end{equation}
    
    Using strict continuity of $\varphi^\gamma$ at $\bar{x}^{k+1}$ by Proposition \ref{Strict continuity} in \eqref{Well defindness (i) 2}, we say that, for all $x$, $\varepsilon$-close($\varepsilon>0$) to $\bar{x}^{k+1}$, $\varphi^{\gamma}(x)\le\varphi^{\gamma}(x^k)+c\|z^k-y^k\|^2.$ Now, from Algorithm \ref{FISTA-DYS-LS} we have $x^{k+1}=(1-\tau^k)\bar{x}^{k+1}+\tau^k(x^k+d^k).$ So, we have, $x^{k+1}\to\bar{x}^{k+1}$ as $\tau^k\to0.$ By repeatedly halving as discussed in Algorithm \ref{FISTA-DYS-LS}, we reach the following statements:
    \begin{itemize}
        \item [$\bullet$] $x^{k+1}$ is eventually $\varepsilon$-close to $\bar{x}^{k+1}.$
        \item [$\bullet$] If $i_{max}$ is finite, then the maximum number of backtracking is reached, and then we choose $x^{k+1}=\bar{x}^{k+1}.$
    \end{itemize}
    
    From these two statements, we reach what we wanted to prove.

    Now, we say that, after some finite backtracking, the following inequality is satisfied:
    \begin{equation}\label{Well defindness (ii) 1}
        \varphi^{\gamma}(x^{k+1})<\varphi^{\gamma}(x^k)+c\|z^k-y^k\|^2.
    \end{equation}

    Let us consider that, at $N^{\text{th}}$ iteration, we reached step 5 of Algorithm \ref{FISTA-DYS-LS}. So, we say, $\|r^k\|> \epsilon$, for all $k<N.$ Now, we have,
    \begin{align}
        \notag
        \|r^k\|&=\|y^k-w^k\|\\\notag
        &=\|y^k-z^k+z^k-w^k\|\\\notag
        &=\|y^k-z^k+\beta^k(w^k-w^{k-1})\|\\
        \Rightarrow\;\|r^k\|&\le\|y^k-z^k\|+\|w^k-w^{k-1}\|.\\
        &\qquad\qquad\quad[\text{By using the triangle inequality of norm and Lemma \ref{t^k}}]
        \label{Well definedness (ii) 2}
    \end{align}

    Now, squaring both sides of \eqref{Well definedness (ii) 2}, we get for all $k<N,$
    \allowdisplaybreaks{
    \begin{align}
        \notag
        & \|r^k\|^2\le\|y^k-z^k\|^2+\|w^k-w^{k-1}\|^2+2\|y^k-z^k\|\|w^k-w^{k-1}\|\\\notag
        \Rightarrow\; & c\|r^k\|^2\ge c\|y^k-z^k\|^2+c\|w^k-w^{k-1}\|^2+2c\|y^k-z^k\|\|w^k-w^{k-1}\|\\\notag
        \Rightarrow\; & c\epsilon^2\ge c\|y^k-z^k\|^2+c\|w^k-w^{k-1}\|^2+2c\|y^k-z^k\|\|w^k-w^{k-1}\|\\\notag
        &\qquad\qquad\qquad\qquad\qquad\qquad\qquad\qquad\qquad[\text{Since }\|r^k\|> \epsilon\text{ and }c<0]\\\notag
        \Rightarrow\; & c\|y^k-z^k\|^2\le c\epsilon^2-c\|w^k-w^{k-1}\|^2-2c\|y^k-z^k\|\|w^k-w^{k-1}\|\\\notag
        \Rightarrow\; & c\|y^k-z^k\|^2\le c\epsilon^2-c\|w^k-w^{k-1}\|^2-c\|y^k-z^k\|^2-c\|w^k-w^{k-1}\|^2\\\notag
        \Rightarrow\; & 2c\|y^k-z^k\|^2\le c\epsilon^2-2c\|w^k-w^{k-1}\|^2\\
        \Rightarrow\; & c\|y^k-z^k\|^2\le \frac{c\epsilon^2}{2}-c\|w^k-w^{k-1}\|^2.\label{Well definedness (ii) 3}
    \end{align}}

    Then, using summation from $k=0$ to $k=N-1$ in both sides of \eqref{Well definedness (ii) 3}, we get
    \begin{equation}
        c\sum_{k=0}^{N-1}\|y^k-z^k\|^2\le\frac{cN\epsilon^2}{2}-c\sum_{k=0}^{N-1}\|w^k-w^{k-1}\|^2.\label{Well definedness (ii) 4}
    \end{equation}

    Since, $\sum_{k=0}^{+\infty}\|w^k-w^{k-1}\|^2<+\infty$ then for any $N\in\mathbb{R}_+$, $\sum_{k=0}^{N-1}\|w^k-w^{k-1}\|^2<+\infty$ and let $\sum_{k=0}^{N-1}\|w^k-w^{k-1}\|^2=M$, where $0\le M<+\infty.$ So, from \eqref{Well definedness (ii) 4}, we get 
    \begin{equation}
        c\sum_{k=0}^{N-1}\|y^k-z^k\|^2\le\frac{cN\epsilon^2}{2}-cM.\label{Well definedness (ii) 5}
    \end{equation}

    Now, using summation from $k=0$ to $k=N$ in both sides of \eqref{Well defindness (ii) 1} and then using \eqref{Well definedness (ii) 5}, we have
    \begin{align*}
        &\varphi^{\gamma}(x^{N+1})<\varphi^{\gamma}(x^0)+c\sum_{k=0}^N\|z^k-y^k\|^2\\
        \Rightarrow\; & \inf \varphi<\varphi^{\gamma}(x^0)+c\sum_{k=0}^{N-1}\|z^k-y^k\|^2\\
        \Rightarrow\; & \inf\varphi<\varphi^{\gamma}(x^0)+\frac{cN\epsilon^2}{2}-cM\;\;\;\;\;[\text{By using \eqref{Well definedness (ii) 5}}]\\
        \Rightarrow\; & N<\frac{2}{c\epsilon^2}\left(\inf\varphi-\varphi^\gamma(x^0)+cM\right).
    \end{align*}
    
    Hence, (ii) is proved, and the proof of (iii) is followed by Theorem \ref{Subsequential convergence}.
    
\end{proof}

After discussing the well-definedness of Algorithm \ref{FISTA-DYS-LS}, we now analyze the asymptotic behavior of the sequence $\{(x^k,y^k,w^k)\}$, generated by Algorithm \ref{FISTA-DYS-LS}. In Theorem \ref{Subsequential convergence for FISTA-DYS-LS}, we discuss the subsequential convergence property for the generated sequence. In this theorem, we focus on the characterization of the cluster points of the sequences $\{y^k\}$ and $\{w^k\}$ and the boundedness of the generated sequence.

\begin{thm}[Subsequential Convergence]\label{Subsequential convergence for FISTA-DYS-LS}
    Under Assumption \ref{Assumption 1}, consider the iterates generated by Algorithm \ref{FISTA-DYS-LS} with $\epsilon = 0$. It follows that $\sum_{k\in\mathbb{N}}\|r^k\|^2$ is finite. It also follows that
    \begin{enumerate}
        \item[(i)] Sequences $\{y^k\}$ and $\{w^k\}$ have the same cluster points, all of which are stationary for $\varphi$ and on which $\varphi$ has a finite value, this being the limit of $\{\varphi^\gamma(x^k)\}$.
        \item[(ii)] If $\varphi$ has bounded level sets, then the sequence $\{(x^k, y^k, w^k)\}$ is bounded.
    \end{enumerate}
\end{thm}
\begin{proof}
    This proof is followed by Theorem \ref{Subsequential convergence}.
\end{proof}

After the discussion of subsequential convergence, we now focus on the rate of convergence of Algorithm \ref{FISTA-DYS-LS}. To establish the superlinear convergence rate for Algorithm \ref{FISTA-DYS-LS}, we need to study the behavior of the objective function near a strong local minimizer and the search direction generated by the modified Broyden update rule. In the following lemmas, we focus on these results.

\begin{lem}\label{Strong convexity type lemma}
    Under Assumption \ref{Assumption 1}, consider the iterates generated by Algorithm \ref{FISTA-DYS-LS} with $\epsilon = 0$, and let $\varphi^{\star}$ be the limit point of $\left\{\varphi^{\gamma}(x^k)\right\}$. If $\{y^k\}$ converges to the strong local minimizer $y^{\star}$ of $\varphi$, then for some $\mu>0$, we have $\varphi^{\gamma}(x^k)-\varphi^{\star}\ge\frac{\mu}{2}\|x^k-x^{\star}\|^2$ for sufficiently large $k$, with $x^{\star} = y^{\star} + \gamma \nabla g(y^{\star}).$
\end{lem}
\begin{proof}
    We have $y^k = \mathrm{prox}_{\gamma g}(x^k)$. So, by $\frac{1}{1+\gamma L_g}$- strong monotonicity of $\mathrm{prox}_{\gamma g}$, we get
    \begin{equation}\label{SCTL 1}
        \|y^k-y^{\star}\|^2\ge\frac{1}{(1+\gamma L_g)^2}\|x^k-x^{\star}\|^2.
    \end{equation}
    
    Again, by the triangle inequality of the norm, we get
    \begin{align}
        \notag
        &\|y^k-y^{\star}\|\le\|y^k-w^k\|+\|w^k-y^{\star}\|\\\notag
        \Rightarrow\;&\|y^k-y^{\star}\|^2\le\|y^k-w^k\|^2+\|w^k-y^{\star}\|^2+2\|y^k-w^k\|\|w^k-y^{\star}\|\\\notag
        &\qquad\qquad\;\;\;\le\|y^k-w^k\|^2+\|w^k-y^{\star}\|^2+\frac{1}{\eta}\|y^k-w^k\|^2+\eta\|w^k-y^{\star}\|^2,\\\notag
        &\qquad\qquad\qquad\qquad\qquad\qquad\qquad\qquad\qquad\qquad\qquad\qquad\text{ where }\eta>0\\
        \Rightarrow\;&\|y^k-y^{\star}\|^2\le\frac{1+\eta}{\eta}\|y^k-w^k\|^2+(1+\eta)\|w^k-y^{\star}\|^2.\label{SCTL 2}
    \end{align}
    
    Now, from \eqref{SCTL 1} and \eqref{SCTL 2}, we say that
    \begin{equation}\label{SCTL 3}
        \|w^k-y^{\star}\|^2\ge\frac{1}{(1+\eta)(1+\gamma L_g)^2}\|x^k-x^{\star}\|^2-\frac{1}{\eta}\|y^k-w^k\|^2.
    \end{equation}
    
    Now, from Theorem \ref{Subsequential convergence for FISTA-DYS-LS}, we have $(y^k-w^k)\to0$ as $k\to+\infty$ and $\varphi(y^{\star})=\varphi^{\star}.$ Then $w^k\to y^{\star}$ and by strong local minimality $\exists\;\delta>0$ and $N\in\mathbb{N}$ s.t. $\varphi(w^k)-\varphi^{\star}\ge\frac{\delta}{2}\|w^k-y^{\star}\|^2,\;\forall k\ge N$. Now, from Proposition \ref{Sandwiching property}, we have $\varphi^{\gamma}(x^k)\ge\varphi(w^k)+\frac{1-\gamma(L_g+L_h)}{2\gamma}\|w^k-y^k\|^2$ for $\gamma<\frac{1}{L_g+L_h}$, which implies that
    \begin{equation}\label{SCTL 4}
        \varphi^{\gamma}(x^k)\ge\varphi^{\star}+\frac{\delta}{2}\|w^k-y^{\star}\|^2+\frac{1-\gamma(L_g+L_h)}{2\gamma}\|w^k-y^k\|^2.
    \end{equation}
    
    Now, substituting \eqref{SCTL 3} in \eqref{SCTL 4} we get, for all $\epsilon>0$ and $k\ge N$
    \begin{align*}
        \varphi^{\gamma}(x^k) & \ge \varphi^{\star}+\frac{\delta}{2}\left[\frac{1}{(1+\eta)(1+\gamma L_g)^2}\|x^k-x^{\star}\|^2-\frac{1}{\eta}\|y^k-w^k\|^2\right]\\
        &\qquad\qquad\qquad\qquad\qquad\qquad\qquad\qquad+\frac{1-\gamma(L_g+L_h)}{2\gamma}\|w^k-y^k\|^2,
    \end{align*}
    which implies that
    \begin{align}
        \varphi^{\gamma}(x^k) & \ge \varphi^{\star}+\frac{\delta}{2(1+\eta)(1+\gamma L_g)^2}\|x^k-x^{\star}\|^2+\left[\frac{1-\gamma(L_g+L_h)}{2\gamma}-\frac{\delta}{2\eta}\right]\|w^k-y^k\|^2.\label{SCTL 5}
    \end{align}
    
    If we choose $\eta=\frac{\gamma\delta}{1-\gamma(L_g+L_h)}>0,$ then from \eqref{SCTL 5}, we have $\varphi^{\gamma}(x^k)-\varphi^{\star}\ge\frac{\mu}{2}\|x^k-x^{\star}\|^2$, for all $k\ge N$, where $\mu=\frac{\delta}{(1+\eta)(1+\gamma L_g)^2}.$ Hence the proof is done.
\end{proof}
\begin{lem}\label{Acceptance of the unit stepsize}
    Under Assumption \ref{Assumption 1}, consider that the iterates generated by Algorithm \ref{FISTA-DYS-LS} and $\{y^k\}$ converge to a local strong minimum $y^{\star}$ of $\varphi$ and $\{d^k\}$ are superlinear directions (relative to a sequence $\{x^k\}$ converging to $x^{\star},$ where $x^{\star}=y^{\star}+\gamma\nabla g(y^\star)$), i.e., $\lim_{k\to+\infty}\frac{\|x^k+d^k-x^{\star}\|}{\|x^k-x^{\star}\|}=0$. Then after some finite iteration, the unit step size will be accepted, i.e., $\tau^k=1.$ Hence, we have $x^{k+1}=x^k+d^k,$ which guaranties superlinear convergence.
\end{lem}
\begin{proof}
    From Lemma \ref{Strong convexity type lemma}
    \begin{equation}\label{AUS lemma (i)}
        \varphi^{\gamma}(x^k)-\varphi^{\gamma}(x^{\star})\ge\frac{\mu}{2}\|x^k-x^{\star}\|^2,\;\text{for some }\mu>0,
    \end{equation} 
    where $x^{\star}=y^{\star}+\gamma\nabla g(y^{\star})$ and $x^{\star}$ is limit point of $\{x^k\}.$

    Let,
    \begin{align}
        \notag
        \varepsilon^k & = \frac{\varphi^{\gamma}(x^k+d^k)-\varphi^{\gamma}(x^{\star})}{\varphi^{\gamma}(x^k)-\varphi^{\gamma}(x^{\star})}\\\notag
        & = \frac{\varphi^{\gamma}(x^k+d^k)-\varphi(y^{\star})}{\varphi^{\gamma}(x^k)-\varphi^{\gamma}(x^{\star})}\;\;\;\;\;[\text{By using local minimality at }y^{\star}\text{ of }\varphi\text{ and }\\\notag
        &\qquad\qquad\qquad\qquad\qquad\qquad\qquad\qquad\qquad\qquad\qquad\;\;\;\;\text{Proposition \ref{Minimization and level-boundedness equivalence}}]\\\notag
        & \le \frac{1+\gamma(L_g+L_h)}{\gamma\mu}\frac{\|\mathrm{prox}_{\gamma g}(x^k+d^k)-\mathrm{prox}_{\gamma g}(x^{\star})\|^2}{\|x^k-x^{\star}\|^2}\\\notag
        &\qquad\qquad\qquad\qquad\qquad\qquad\;\;\;\;[\text{Combining Proposition \ref{quadratic upper bound} and \eqref{AUS lemma (i)}}]\\
        \Rightarrow\;\varepsilon^k& \le \frac{1+\gamma(L_g+L_h)}{\gamma\mu(1+\gamma \sigma_g)^2}\frac{\|x^k+d^k-x^{\star}\|^2}{\|x^k-x^{\star}\|^2}\;\;[\text{By using Proposition \ref{Proximal properties of smooth functions} }].\label{AUS lemma (ii)}
    \end{align}

    Now, since $\{d^k\}$ are superlinear directions, $\varepsilon^k\to0$ as $k\to+\infty$. Now, from Theorem \ref{Subsequential convergence for FISTA-DYS-LS}, we have, $(w^k-y^k)\to0$ as $k\to+\infty$, which implies that $z^k-y^k=w^k+\beta^k(w^k-w^{k-1})-y^k\to0$ as $k\to+\infty$, because of $\sum_{k\in\mathbb{N}}(w^{k+1}-w^k)<+\infty.$ Hence, we get $\bar{x}^{k+1}\to x^{\star}$ as $k\to+\infty.$ Now, since $\varphi^{\gamma}$ is strictly continuous, $\varphi^{\gamma}(\bar{x}^{k+1})\to\varphi^{\gamma}(x^{\star})=\varphi^{\star}$ as $k\to+\infty$. Therefore, eventually $\varphi^{\gamma}(\bar{x}^{k+1})\ge\varphi^{\star}=\varphi^{\gamma}(x^{\star})$ and $\varepsilon^k\le1.$

    Now, 
    \begin{align}
        \notag
        \varphi^{\gamma}(x^k+d^k)-\varphi^{\gamma}(x^k) & = \varphi^{\gamma}(x^k+d^k)-\varphi^{\gamma}(x^{\star})-(\varphi^{\gamma}(x^k)-\varphi^{\gamma}(x^{\star}))\\\notag
        & = -(1-\varepsilon^k)(\varphi^{\gamma}(x^k)-\varphi^{\gamma}(x^{\star}))\\\notag
        & \le -(1-\varepsilon^k)(\varphi^{\gamma}(x^k)-\varphi^{\gamma}(\bar{x}^{k+1}))\\
        & < c\|z^k-y^k\|^2\;\;\;\;\;[\text{By using \eqref{Well defindness (i) 2} and }\varepsilon^k\le1].\label{AUS lemma (iii)}
    \end{align}

    Hence, we eventually choose $\tau^k=1$, i.e., $x^{k+1}=x^k+d^k$, for which the sufficient decrease property is satisfied.
\end{proof}
\begin{lem}\label{Superlinear convergence for direction}
    Under Assumption \ref{Assumption 1}, consider the iterates generated by Algorithm \ref{FISTA-DYS-LS} and $\{x^k\}$ converge to $x^{\star}$ such that $R_{\gamma}=\mathrm{prox}_{\gamma g}-\mathrm{prox}_{\gamma f}(2\mathrm{prox}_{\gamma g}-I-\gamma\nabla h)$ is strictly differentiable and the Jacobian $JR_{\gamma}(x^{\star})$ is non-singular. Then, under the Dennis-Mor\'{e} condition
    \begin{equation}\label{Dennis-More condition}
        \lim_{k\to+\infty}\frac{\|R_{\gamma}(x^k)+JR_{\gamma}(x^{\star})d^k\|}{\|d^k\|}=0,
    \end{equation}
    the directions $\{d^k\}$ are superlinear convergent.
\end{lem}
\begin{proof}
    Since, $R_{\gamma}$ is strictly differentiable at $x^{\star}$, then
    \begin{equation}\label{Strict Differentiability}
        \lim_{x\to x^{\star}}\frac{R_{\gamma}(x)-R_{\gamma}(x^{\star})-JR_{\gamma}(x^{\star})(x-x^{\star})}{\|x-x^{\star}\|}=0.
    \end{equation}

    Let, $r_1^k$, $r_2^k\in R_{\gamma}(x^k).$ Now, using \eqref{Strict Differentiability} we get, $\|r_1^k-r_2^k\|\le\|r_1^k-R_{\gamma}(x^{\star})-JR_{\gamma}(x^{\star})(x-x^{\star})\|+\|r_2^k-R_{\gamma}(x^{\star})-JR_{\gamma}(x^{\star})(x-x^{\star})\|\to0$ as $x^k\to x^{\star}$, i.e., $k\to+\infty.$ So, for sufficiently large $k$, $R_{\gamma}(x^k)$ is single valued.

    Let $A=JR_{\gamma}(x^{\star}).$ Since, $A$ exists and non-singular, then $A^{-1}$ also exists and non-singular. So, $\|A^{-1}\|<+\infty$ and then for any $d\in\mathbb{R}^{n}$, $\|d\|=\|A^{-1}Ad\|\le\|A^{-1}\|\|Ad\|,$ which implies $\|Ad\|\ge\frac{1}{\|A^{-1}\|}\|d\|.$ Let, $\alpha_0=\frac{1}{\|A^{-1}\|}>0$. Then, $\|Ad\|\ge\alpha_0\|d\|,$ $\forall d\in \mathbb{R}^n.$

    Now, due to strict differentiability of $R_{\gamma}$ at $x^{\star}$ and $R_{\gamma}(x^{\star})=0$, we get, for any $x\in\mathbb{R}^n$,
    \begin{align}
        \notag
        & R_{\gamma}(x)=R_{\gamma}(x^{\star})+A(x-x^{\star})+o(\|x-x^{\star}\|)\\\notag
        \Rightarrow\; & R_{\gamma}(x)=A(x-x^{\star})+o(\|x-x^{\star}\|)\\\notag
        \Rightarrow\; & \|R_{\gamma}(x)\|\ge\|A(x-x^{\star})\|-\|o(\|x-x^{\star}\|)\|\\\notag
        \Rightarrow\; &\|R_{\gamma}(x)\|\ge\alpha_0\|x-x^{\star}\|-\frac{\alpha_0}{2}\|x-x^{\star}\|\;\;\;\;\;[\text{By using the definition of Little-}o]\\
        \Rightarrow\; & \|R_{\gamma}(x)\|\ge\frac{\alpha_0}{2}\|x-x^{\star}\|=\alpha\|x-x^{\star}\|\;\;\;\;\;[\text{Let }\alpha=\frac{\alpha_0}{2}>0].\label{Positive lower bound of residual}
    \end{align}

    Now, from \eqref{Strict Differentiability}, we say that
    \begin{equation}\label{Strict differentiability 2}
        \lim_{k\to+\infty}\frac{\|R_{\gamma}(x^k+d^k)-R_{\gamma}(x^k)-JR_{\gamma}(x^{\star})d^k\|}{\|d^k\|}=0.
    \end{equation}

    Now, using the triangle inequality of the norm, we get
    \begin{align}
        \notag
        & \frac{\|R_\gamma(x^k+d^k)-R_{\gamma}(x^k)-JR_{\gamma}(x^{\star})d^k\|}{\|d^k\|}\ge\frac{\|R_{\gamma}(x^k+d^k)\|}{\|d^k\|}\\\notag
        &\qquad\qquad\qquad\qquad\qquad\qquad\qquad\qquad\qquad-\frac{\|R_{\gamma}(x^k)+JR_{\gamma}(x^{\star})d^k\|}{\|d^k\|}\\\notag
        \Rightarrow\; & \frac{\|R_\gamma(x^k+d^k)\|}{\|d^k\|}\le\frac{\|R_\gamma(x^k+d^k)-R_{\gamma}(x^k)-JR_{\gamma}(x^{\star})d^k\|}{\|d^k\|}\\
        &\qquad\qquad\qquad\qquad+\frac{\|R_{\gamma}(x^k)+JR_{\gamma}(x^{\star})d^k\|}{\|d^k\|}.\label{SCD lemma (i)}
    \end{align}

    Now, taking both side limits as $k\to+\infty$ and then using \eqref{Strict differentiability 2}, \eqref{Dennis-More condition}, and the sandwich theorem of limits in \eqref{SCD lemma (i)}, we finally get that $\lim_{k\to+\infty}\frac{\|R_{\gamma}(x^k+d^k)\|}{\|d^k\|}=0.$ Now, since $\|R_\gamma(x^k+d^k)\|\ge\alpha\|x^k+d^k-x^{\star}\|$, for some $\alpha>0$, then we conclude that $\lim_{k\to+\infty}\frac{\|x^k+d^k-x^{\star}\|}{\|d^k\|}=0.$

    Now, we have
    \allowdisplaybreaks{
    \begin{align*}
        \notag
        & \frac{\|x^k+d^k-x^{\star}\|}{\|x^k-x^{\star}\|}=\frac{\|x^k+d^k-x^{\star}\|}{\|d^k-(x^k+d^k-x^{\star})\|}\\\notag
        \Rightarrow\; & \frac{\|x^k+d^k-x^{\star}\|}{\|x^k-x^{\star}\|}\le\frac{\|x^k+d^k-x^{\star}\|}{\|d^k\|-\|x^k+d^k-x^{\star}\|}\\
        \Rightarrow\; & \frac{\|x^k+d^k-x^{\star}\|}{\|x^k-x^{\star}\|}\le\frac{\|x^k+d^k-x^{\star}\|}{\|d^k\|\left(1+\frac{\|x^k+d^k-x^{\star}\|}{\|d^k\|}\right)}\to0\;\text{as }k\to+\infty.
    \end{align*}}

    Hence, the directions ${d^k}$ are superlinear is proved.
\end{proof}

Now, we are ready to discuss the superlinear convergence rate for Algorithm \ref{FISTA-DYS-LS}.

\begin{thm}[Superlinear Convergence]\label{Superlinear convergence}
    Under Assumption \ref{Assumption 1}, consider the iterates generated by Algorithm \ref{FISTA-DYS-LS} with the direction $\{d^k\}$ as discussed in \eqref{Quasi-Newton FISTA-DYS direction} selected through the modified Broyden approach, and $\{H^k\}$ as discussed in \eqref{Jacobian Modified Broyden} is bounded, and $\{x^k\}$ converges to a strong local minimizer $x^{\star}$ of $\varphi$. Also, assume that at $x^{\star}$, $R_{\gamma}$ is Lipschitz differentiable and the Jacobian $JR_{\gamma}(x^{\star})$ is non-singular. It follows that the Dennis Mor\'{e} condition holds and therefore, after some finite iterations, $\tau^k=1$ is always acceptable and $\{x^k\}$ converges to $x^{\star}$ superlinearly.
\end{thm}
\begin{proof}
    Using Lemma \ref{Strong convexity type lemma},
\ref{Acceptance of the unit stepsize},
and \ref{Superlinear convergence for direction}, and a similar approach, which is used in the proof in \cite[Theorem 4.8]{themelis2022douglas}, we conclude our result.
\end{proof}

\section{FISTA-type PnP-DYS for Image Restoration}\label{FISTA-type PnP-DYS for Image Restoration}
In this section, we propose a FISTA-type PnP-DYS algorithm and analyze its convergence. The PnP method is a flexible framework for solving inverse problems arising from large-scale measurements, where we use denoisers to incorporate prior information. This method is inspired by classical proximal splitting methods for non-smooth composite optimization, including FBS, DRS, ADMM, and extrapolated-DYS. Due to the success of deep learning, PnP methods are commonly used to include learned priors through pretrained neural networks \cite{gavaskar2021plug}. These methods perform very well in many applications. So, we replace the proximal operator of $f$ in Algorithm \ref{FISTA-DYS} with a learned denoiser $\mathcal{D}_\sigma$ and obtain a FISTA-type PnP-DYS method. In this paper, we adopt the GS-denoiser from \cite{cohen2021has,hurault2022gradient,wu2024extrapolated} as follows:
\begin{equation}\label{Denoiser}
    \mathcal{D}_\sigma=I-\nabla g_{\sigma},
\end{equation}
where $g_{\sigma}(x)=\frac{1}{2}\|x-N_{\sigma}(x)\|^2$ and $N_{\sigma}(x)$ is implemented as a differentiable neural network, which allows explicit computation of $g_{\sigma}(x)$ and ensures that $g_{\sigma}$ has a Lipschitz continuous gradient with constant $L<1$. The denoiser $D_{\sigma}$ in (\ref{Denoiser}) is originally trained to remove Gaussian noise with level $\sigma$. As shown in [27], even when constrained to be an exact conservative field, it can achieve state-of-the-art denoising performance. In the next proposition, we see that the denoiser $\mathcal{D}_{\sigma}$ in (\ref{Denoiser}) is a form of proximal mapping of a weakly convex function.

\begin{prop}[\cite{hurault2022proximal}]\label{Relation between denoiser and proximal}
    $\mathcal{D}_{\sigma}(x)=\mathrm{prox}_{\phi_{\sigma}}(x),$ where $\phi_{\sigma}$ is expressed as
    \begin{equation}\label{weak_convex_function}
        \phi_{\sigma}(x)=
        \left\{
        \begin{array}{ll}
            g_{\sigma}(\mathcal{D}_{\sigma}^{-1}(x))-\frac{1}{2}\|\mathcal{D}_{\sigma}^{-1}(x)-x\|^2, & \text{if }x\in\mathrm{Im}(\mathcal{D}_{\sigma}),\\
            +\infty, & \text{otherwise}.
        \end{array}
        \right.
    \end{equation}

    Moreover, $\phi_{\sigma}$ is $\frac{L}{L+1}$-weakly convex and $\frac{L}{1-L}$-smooth on $\mathrm{Im}(\mathcal{D}_{\sigma}),$ and $\phi_{\sigma}(x)\ge g_{\sigma}(x),$ $\forall x\in\mathbb{R}^n.$
\end{prop}

Drawing on Proposition \ref{Relation between denoiser and proximal}, we aim to develop an FISTA-type PnP-DYS algorithm using the denoiser $\mathcal{D}_\sigma$ in (\ref{Denoiser}), which corresponds to the proximal operator of a non-convex function $\phi_\sigma$ defined in (\ref{weak_convex_function}). To this end, we consider the following optimization problem:
\begin{equation}\label{Prob 2}
\min_{x} \; \varphi_{\gamma,\sigma}(x) 
= f(x) + \frac{1}{\gamma}\,\phi_\sigma(x) + h(x),
\end{equation}
where $f$ represents the data-fidelity term and may be non-convex, $h$ is a $L_h$-smooth function, and $\gamma$ represents regularization parameter. Moreover, $\phi_\sigma$ is defined as in Proposition \ref{Relation between denoiser and proximal} through a function $g_\sigma$ satisfying $\mathcal{D}_\sigma = I - \nabla g_\sigma$.

\begin{assum}\label{Assumption 2}
    \leavevmode
    \begin{itemize}
        \item[(i)] $f\in C^{1,~1}(\mathbb{R}^n)$ is $L_f$-smooth.
        \item[(ii)] $h\in C^{1,~1}(\mathbb{R}^n)$ is $L_h$-smooth.
    \end{itemize}
\end{assum}

Based on Assumption \ref{Assumption 2} of the problem (\ref{Prob 2}), we develop Algorithm \ref{FISTA-PnP-DYS}.

\begin{algorithm}[h]
\caption{FISTA-type PnP-DYS}
\begin{algorithmic}[1]

\State \textbf{Input:} $x^0 = w^{-1} \in \mathbb{R}^n$, step size $\gamma > 0$, relaxation parameter $\lambda>0$.
\State $t^0 = 1$

\For{$k = 0,1,2,\dots$}
    \State $y^k = \mathrm{prox}_{\gamma f}(x^k)$
    \State $w^k = \mathcal{D}_{\sigma}\big(2y^k - x^k - \gamma \nabla h(y^k)\big)$
    \State $t^{k+1} = \frac{1 + \sqrt{1 + 4(t^k)^2}}{2}$
    \State $\beta^k = \frac{t^k - 1}{t^{k+1}}$
    \State $z^k = w^k + \beta^k (w^k - w^{k-1})$
    \State $x^{k+1} = x^k + \lambda(z^k - y^k)$
\EndFor

\end{algorithmic}\label{FISTA-PnP-DYS}
\end{algorithm}

Now, we analyze the convergence results for Algorithm \ref{FISTA-PnP-DYS}. Since $f$ and $h$ are $L_f,~L_h$-smooth respectively, it becomes a special form of (\ref{Prob 1}) with $f=\frac{1}{\gamma}\phi_{\sigma}$ and $g=f$. So, the following theorems (Theorems \ref{Sufficient decrease for FISTA-type PnP DYS}-\ref{Global convergence for FISTA-type PnP DYS}) follow from Theorems \ref{FISTA_SD_TH}-\ref{Global convergence for FISTA-type DYS}.

\begin{thm}[Sufficient decrease for FISTA-type PnP-DYS]\label{Sufficient decrease for FISTA-type PnP DYS}
 Consider that Assumption \ref{Assumption 2} is satisfied, $\gamma<\frac{1}{\beta_f}$ and consider $(y^k,~z^k,~x^{k+1})\in \text{FISTA-PnP-DYS}(x^k)$. Then for $C<0$, 
    $$\varphi^{\gamma}(x^{k+1})\le\varphi^{\gamma}(x^k)+\frac{C}{(1+\gamma L_g)^2}\|x^{k+1}-x^k\|^2,$$
 if the following are held:
    \begin{itemize}
        \item[(i)] $0<\gamma<\frac{1}{L_1},~\frac{4L_2}{L_1+L_2}<\lambda<2\left(1+\frac{\sigma_1}{L_1}\right),~L_1>\frac{L_2-\sigma_1+\sqrt{(L_2-\sigma_1)^2-4L_2\sigma_1}}{2}.$
        \item[(ii)] $0<\gamma<\frac{1}{L_1},~2\gamma L_1\left(1+\frac{\sigma_1}{L_1}\right)<\lambda\le\frac{\sigma_1+2L_2}{\frac{\sigma_1L_1}{2(L_1+\sigma_1)}+\frac{\sigma_2L_2}{2(L_2+\sigma_2)}},~\sigma_1\le-2L_2,~L_1>\max\left\{|\sigma_1|,\frac{\sigma_1\sigma_2}{2L_2+\sigma_2}\right\}.$
        \item[(iii)] $0<\gamma<\frac{1}{L_1+\sigma_1},~2\left(1+\frac{\sigma_1}{L_1}\right)\le\lambda\le\frac{2}{1-\frac{\sigma_1L_1}{(L_1+\sigma_1)^2}-\frac{\sigma_2L_2}{(L_1+\sigma_1)(L_2+\sigma_2)}+\frac{2L_1L_2}{(L_1+\sigma_1)^2}},~0>\sigma_1\ge\frac{\sigma_2L_2}{L_2+\sigma_2},~0<L_1\le\frac{1}{L_2(2L_2+\sigma_2)}(L_2\sigma_1\sigma_2-L_2\sigma_1^2-\sigma_1^2\sigma_2).$
    \end{itemize}

    Here, $\sigma_1=\min\{\sigma_g,~0\},~L_1\ge L_g$ s.t. $L_1+\sigma_1>0$ and $\sigma_2=\min\{\sigma_h,~0\},~L_2\ge L_h$ s.t. $L_2+\sigma_2>0.$
\end{thm}

\begin{thm}[Subsequential convergence for FISTA-type PnP-DYS]\label{Subsequential convergence for FISTA-type PnP DYS}
    Let us consider that Assumption \ref{Assumption 2} is satisfied, and consider a sequence $\{(x^k, y^k, w^k)\}$ generated by Algorithm \ref{FISTA-PnP-DYS} with stepsize $\gamma$ and relaxation $\lambda$, as in sufficient decrease section, starting from $x^0 \in \mathbb{R}^n$. The following holds:
    \begin{enumerate}
        \item[(i)] The residual $\{w^k - y^k\}$ vanishes with rate $\min_{i \leq k} \|w^i - y^i\| = o\left(\frac{1}{\sqrt{k}}\right).$
        \item[(ii)] Sequences $\{y^k\}$ and $\{w^k\}$ have the same cluster points, all of which are stationary for $\varphi$ and on which $\varphi$ has a finite value, this being the limit of $\{\varphi^\gamma(x^k)\}$. In fact, for each $k$ one has $\mathrm{dist}(0,~\hat{\partial} \varphi(w^k)) \leq \frac{2 + \gamma (L_g + L_h)}{\gamma} \|y^k - w^k\|.$ Moreover, if $\varphi$ satisfies the PL-inequality with respect to $\mu>0$ then $\varphi(w^k)-\varphi^*\in o\left(\frac{1}{k}\right),$ where $\varphi^*=\inf\varphi.$
        \item[(iii)] If $\varphi$ has bounded level sets, then the sequence $\{(x^k, y^k, w^k)\}$ is bounded.
    \end{enumerate}
\end{thm}

\begin{thm}[Global convergence for FISTA-type PnP-DYS]\label{Global convergence for FISTA-type PnP DYS}
    Let us consider that Assumption \ref{Assumption 2} is satisfied, $\varphi$ is level bounded, and $f,~g,$ and $h$ are semialgebraic functions. Also if we consider a sequence $\{(x^k,~y^k,~w^k)\}$ generated by Algorithm \ref{FISTA-PnP-DYS} with all convergence conditions from Theorem \ref{Sufficient decrease for FISTA-type PnP DYS} then the sequences $\{y^k\}$ and $\{w^k\}$ are convergent and they converge to stationary point of $\varphi.$
\end{thm}


\section{FISTA-type PnP-DYS with Line-search for Image Restoration}\label{FISTA-type PnP-DYS with Line-search for Image Restoration}
In this section, we use a quasi-Newton type line-search with the modified Broyden update rule in the FISTA-type PnP-DYS algorithm for the optimization problem \eqref{Prob 2}, under Assumption \ref{Assumption 2}, to reduce the computational cost, converge very quickly, and get a very good result with respect to the FISTA-type PnP-DYS algorithm. Here, we use the same denoiser that is used in FISTA-type PnP-DYS algorithm. This problem is also a subcase of \eqref{Prob 1}. So all the convergence proofs follow from \ref{Convergence Analysis for FISTA-type DYS with Line-search subsection} of Section \ref{FISTA-type DYS with Line-search section}. Now, we discuss our FISTA-type PnP-DYS with Line-search algorithm in Algorithm \ref{FISTA-DYS-LS}.
\begin{algorithm}[h]
\caption{FISTA-type PnP-DYS with Line-search}
\begin{algorithmic}[1]

\State \textbf{Input:} $x^0=w^{-1} \in \mathbb{R}^p$, $\epsilon > 0$, relaxation parameter $\lambda >0$, max backtracks $i_{\max} \le +\infty$, step size $\gamma>0$, $c>\frac{C}{\lambda^2(1+\gamma L_g)^2},$ where $C$ is from Lemma \ref{DYS-SD}.
\State $t^0 = 1$, $k=0$, $(y^0,w^0)\in \text{FISTA-type PnP-DYS}_{\gamma}(x^0)$, and calculate $\varphi^{\gamma}(x^0).$

\While{true}

\State $r^k = y^k - w^k$
\If{$\|r^k\| \le \epsilon$}
    \State \Return $(x^k, y^k, w^k)$
\EndIf
\State $t^{k+1} = \frac{1 + \sqrt{1 + 4(t^k)^2}}{2}$
\State $\beta^k = \frac{t^k - 1}{t^{k+1}}$
\State $z^k=w^k + \beta^k (w^k - w^{k-1})$
\State $\bar{x}^{k+1} = x^k + \lambda (z^k - y^k)$ 

\State Select direction $d^k \in \mathbb{R}^p$, set $\tau^k = 1$, $i^k = 0$

\While{true}

\State $x^{k+1} = (1 - \tau^k)\bar{x}^{k+1} + \tau^k (x^k + d^k)$

\State Compute $(y^{k+1}, w^{k+1}) \in \text{FISTA-type PnP-DYS}_\gamma(x^{k+1})$
\State Evaluate $\varphi^\gamma(x^{k+1})$

\If{$\varphi^\gamma(x^{k+1}) < \varphi^\gamma(x^k) + c\|z^k-y^k\|^2$}
    \State $k \gets k+1$
    \State \textbf{break}
\ElsIf{$i^k = i_{\max}$}
    \State $x^{k+1} = \bar{x}^{k+1}$
    \State Compute $(y^{k+1}, w^{k+1}) \in \text{FISTA-type PnP-DYS}_\gamma(x^{k+1})$
    \State $k \gets k+1$
    \State \textbf{break}
\Else
    \State $\tau^k \gets \tau^k / 2$, \quad $i^k \gets i^k + 1$
\EndIf

\EndWhile

\EndWhile

\end{algorithmic}\label{FISTA-PnP-DYS-LS}
\end{algorithm}

\section{Numerical Experiment}\label{Numerical Experiment}
In this section, all experiments have been done in Python (version 3.14.3) using PyTorch (version 2.10.0+cu128) on an NVIDIA RTX 4500 Ada Generation GPU with an Intel(R) Xeon(R) w7-2595X (2.81 GHz) processor and 256 GB RAM. 

\subsection{Sparse Regularized Quadratic Optimization Problem}
Motivated by \cite{esser2013method,yin2015minimization}, we consider a non-convex sparse regularized quadratic optimization problem
\begin{equation}\label{Sparse structured quadratic optimization problem}
    \min_{X\in\mathbb{R}^{n\times n}} \iota_1\|X\|_1-\frac{\iota_2}{2}\|X\|_F^2+\frac{1}{2}\|X-M\|_F^2+\mathrm{tr}(AXB^{\top}X^{\top}),
\end{equation}
where $A,~B\in\mathbb{R}^{n\times n}$ are given matrices, $M\in\mathbb{R}^{n\times n}$ is observed matrix, $\iota_1,~\iota_2>0$ are regularization parameters, and $\mathrm{tr}(A)=\text{trace of }(A)$. In this problem, the regularization term $\iota_1\|X\|_1-\frac{\iota_2}{2}\|X\|_F^2$ promotes sparse recovery while avoiding unnecessary penalization of large entries, the term $\frac{1}{2}\|X-M\|_F^2$ encourages $X$ to remain close to the observed matrix $M$, and the term $\mathrm{tr}(AXB^{\top}X^{\top})$ represents a smooth quadratic interaction with the decision variable $X$. Now, we choose $f(X)=\iota_1\|X\|_1-\frac{\iota_2}{2}\|X\|_F^2$, $g(X)=\frac{1}{2}\|X-M\|_F^2$, and $h(X)=\mathrm{tr}(AXB^{\top}X^{\top})$ to form \eqref{Sparse structured quadratic optimization problem} as the main model \eqref{Prob 1}. Here $f$ is proper, lower semicontinuous and $\iota_2$-weakly convex, and $g,~h$ are $L_g=1,~L_h=2\|A\|\|B\|$-smooth, respectively. So, Assumption \ref{Assumption 1} is satisfied. Now, here we try to examine practical results for Algorithm \ref{DYS}, \ref{FISTA-DYS}, \ref{FISTA-DYS-LS} with the extrapolated-DYS algorithm \cite[Algorithm 3.1]{wu2024extrapolated} and the DRFDR algorithm \cite[Algorithm 1]{dao2025doubly}. For this practical work, we take parameter values that are used in all algorithms and satisfy the convergence conditions as $n=200,\;\alpha=0.25,\;\lambda=1.0,\;\eta=1.3,\;\theta=0.8,\;\iota_1=0.5,\;\iota_2=0.5,\;\text{err}=10^{-6}.$ Now, in Figure \ref{fig:Final_Results/Sparsity/gamma_0.002_residual_convergence.png} and Figure \ref{fig:Final_Results/Sparsity/gamma_0.005_residual_convergence.png}, we see how the residual values behave for different algorithms for the choices of $\gamma=0.002$ and $\gamma=0.005$, respectively. Then, in Table \ref{table: Sparse Quadratic problem}, we see how many iterations and computational time (in sec) are required by different algorithms for the same values of parameters. From Figure \ref{fig:Final_Results/Sparsity/gamma_0.002_residual_convergence.png}, Figure \ref{fig:Final_Results/Sparsity/gamma_0.005_residual_convergence.png}, and Table \ref{table: Sparse Quadratic problem}, we observe that the residual values for Algorithms\ref{FISTA-DYS} and \ref{FISTA-DYS-LS} converge faster to $<\text{err}$ than DYS, extrapolated-DYS, and DRFDR algorithms.

\begin{figure}[h]
    \centering
    \begin{subfigure}{0.42\textwidth}
        \centering
        \includegraphics[width=\linewidth]{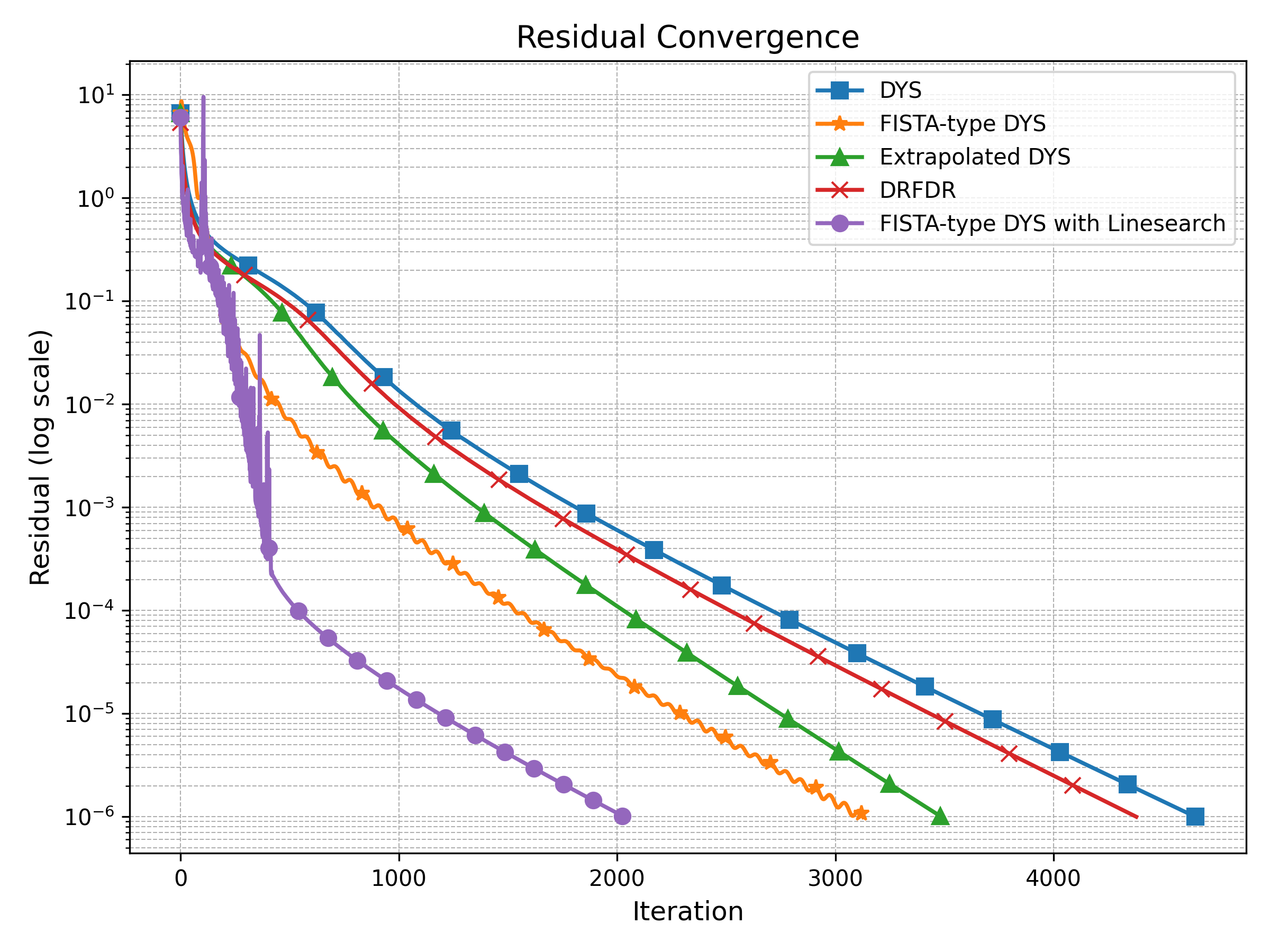}
        \caption{Residual Convergence with $\gamma=0.002$}
        \label{fig:Final_Results/Sparsity/gamma_0.002_residual_convergence.png}
    \end{subfigure}
    \hfill
    \begin{subfigure}{0.42\textwidth}
        \centering
        \includegraphics[width=\linewidth]{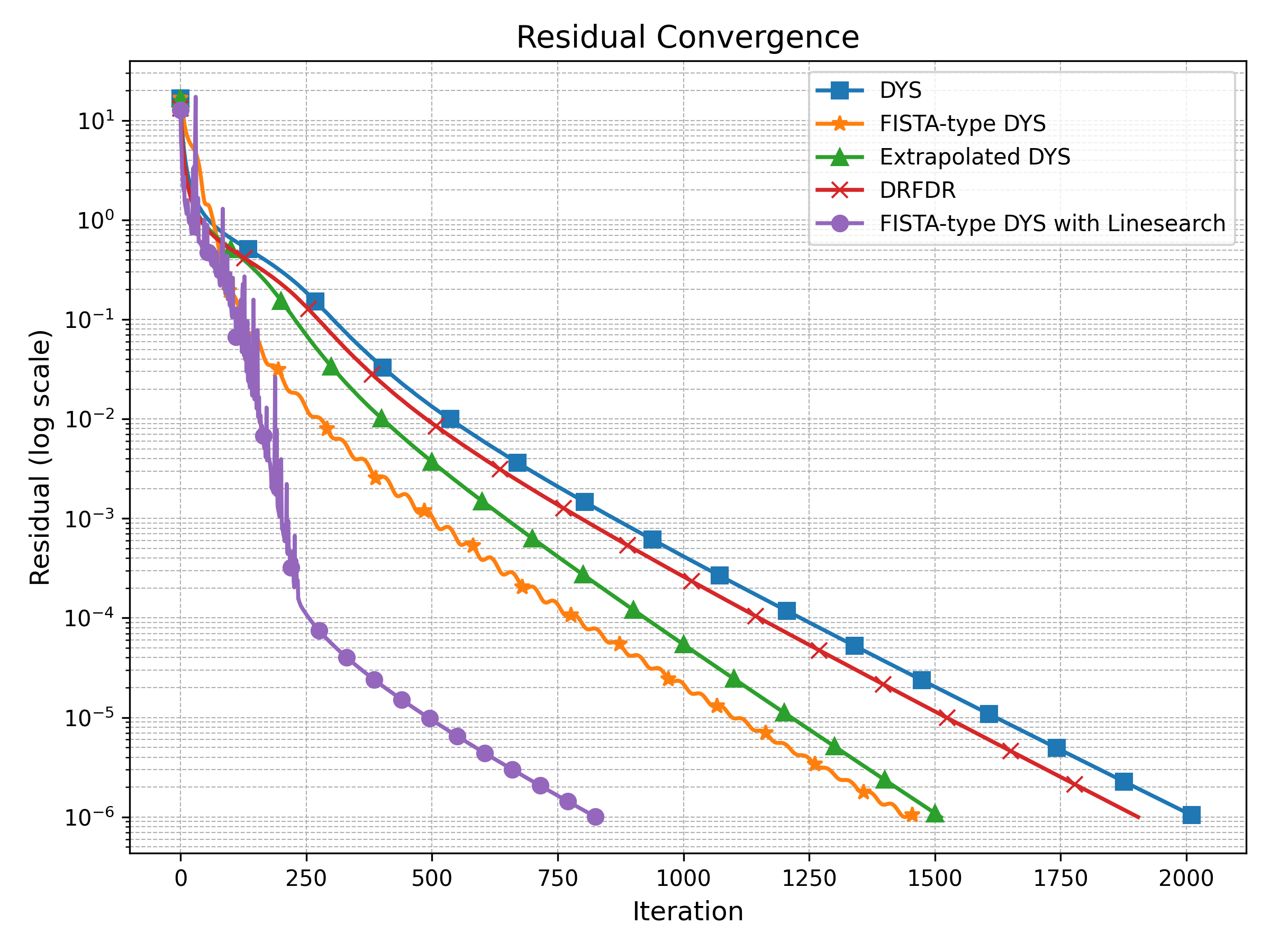}
        \caption{Residual Convergence with $\gamma=0.005$}
        \label{fig:Final_Results/Sparsity/gamma_0.005_residual_convergence.png}
    \end{subfigure}
    \caption{Residual convergence vs iterations for sparse regularized quadratic optimization problem}
    \label{Fig:Residual Convergence for sparse regularized quadratic optimization problem}
\end{figure}

\begin{table}[h]
\centering
\resizebox{1.0\linewidth}{!}{%
\begin{tabular}{c cc cc}
\toprule
\textbf{Values of $\gamma$} & \multicolumn{2}{c}{0.002} & \multicolumn{2}{c}{0.005}\\
\midrule
\textbf{Algorithms} & Iterations & times(in sec) & Iterations & times(in sec)\\
\midrule
DYS & 4651 & 2.8279 & 2019 & 1.2659\\
\midrule
Extrapolated-DYS & 3487 & 2.0762 & 1513 & 0.9816\\
\midrule
DRFDR & 4380 & 2.4711 & 1905 & 1.14504\\
\midrule
\textbf{FISTA-type DYS} & \textbf{3129} & \textbf{2.4985} & \textbf{1462} & \textbf{1.2458}\\
\midrule
\textbf{FISTA-type DYS with Line-search} & \textbf{2034} & \textbf{2490.0967} & \textbf{827} & \textbf{1219.8583}\\
\bottomrule
\end{tabular}
}
\caption{Number of iterations and total computational time taken by different algorithms to achieve the residual error $<10^{-6}$ for various values of $\gamma$}
\label{table: Sparse Quadratic problem}
\end{table}

\subsection{Nonnegative Low Rank Matrix Completion Problem}
Motivated by \cite{wang2021scalable,alcantara2025four}, here we consider a nonnegative low-rank matrix completion problem
\begin{equation}
    \min_{X\in\mathbb{R}^{n\times d}} \left\{\iota_1\|X\|_{\star}-\frac{\iota_2}{2}\|X\|_F^2 + \frac{1}{2}\|\mathcal{P}_{\Omega}(X-M)\|^2 + \frac{\beta}{2}\min_{Y\in\mathbb{R}_+^{n\times d}}\|X-Y\|_F^2\right\},\label{Nonnegative LRM problem}
\end{equation}
where $M\in\mathbb{R}^{n\times d}$ denotes an observed matrix, $\iota_1,~\iota_2>0$ are regularization parameters, $\Omega$ denotes the set of observed entries, and $\mathcal{P}_{\Omega}$ denotes the orthogonal projection onto the set $\Omega$ such that
\begin{align*}
    \left(\mathcal{P}_{\Omega}(X)\right)_{i,j}&=\begin{cases}
        X_{i,j}\;, \text{ if $(i,j)\in \Omega$},\\
        0\quad\;, \text{ otherwise}.
    \end{cases}
\end{align*}

In this problem, the regularization term $\iota_1\|X\|_{\star}-\frac{\iota_2}{2}\|X\|_F^2$ promotes low-ranked matrix recovery problem while avoiding unnecessary penalization of large entries, and the term $\frac{\beta}{2}\min_{Y\in\mathbb{R}_+^{n\times d}}\|X-Y\|_F^2$ ensures the nonnegativity of the recovered low-ranked matrix. Here, we choose $f(X)=\iota_1\|X\|_{\star}-\frac{\iota_2}{2}\|X\|_F^2,\;g(X)=\frac{1}{2}\|\mathcal{P}_{\Omega}(X-M)\|^2,$ and $h(X)=\frac{\beta}{2}\min_{Y\in\mathbb{R}_+^{n\times d}}\|X-Y\|_F^2$. So, $L_g=1$, $L_h=1$ and $f$ is lower semicontinuous, and $\iota_2$-weakly convex and Assumption \ref{Assumption 1} is satisfied. For this problem, we use Algorithm \ref{DYS}, \ref{FISTA-DYS}, \ref{FISTA-DYS-LS}, extrapolated-DYS, and DRFDR to examine the residual convergence for all of the algorithms. Here, we use two types of matrices, which are based on randomized data and a dataset, which is taken from the ``2025 Distribution zone substation data" published by Ausgrid\footnote{\url{https://www.ausgrid.com.au/about-us/about-ausgrid/research-data-sets/distribution-zone-substation-data}}.
\subsubsection{Using Randomized Data}
In this case, the matrix $M$ is generated randomly, and then we use our proposed algorithms. Here, we choose the parameters $n=200,\;\alpha=0.2,\;\beta=1.0,\;\lambda=0.6,\;\eta=1.1,\;\theta=0.8,\;\iota_1=0.8,\;\iota_2=0.1,\;\text{err}=10^{-6}$, which are satisfied our decent properties. Now, in Figure \ref{fig:Final_Results/Nonnegative_LRM/gamma_0.035_residual_convergence.png} and Figure \ref{fig:Final_Results/Nonnegative_LRM/gamma_0.05_residual_convergence.png}, we see how the residual values behave for different algorithms for the choices of $\gamma=0.035$ and $\gamma=0.05$, respectively. Then, in Table \ref{table:LRM problem}, we see how many iterations and computational time (in sec) are required by different algorithms for the same values of parameters. From Figure \ref{fig:Final_Results/Nonnegative_LRM/gamma_0.035_residual_convergence.png}, Figure \ref{fig:Final_Results/Nonnegative_LRM/gamma_0.05_residual_convergence.png}, and Table \ref{table:LRM problem}, we observe that the residual values for Algorithms \ref{FISTA-DYS} and \ref{FISTA-DYS-LS} converge faster to $<\text{err}$ than DYS, extrapolated-DYS, and DRFDR algorithms.

\begin{figure}[h]
    \centering
    \begin{subfigure}{0.42\textwidth}
        \centering
        \includegraphics[width=\linewidth]{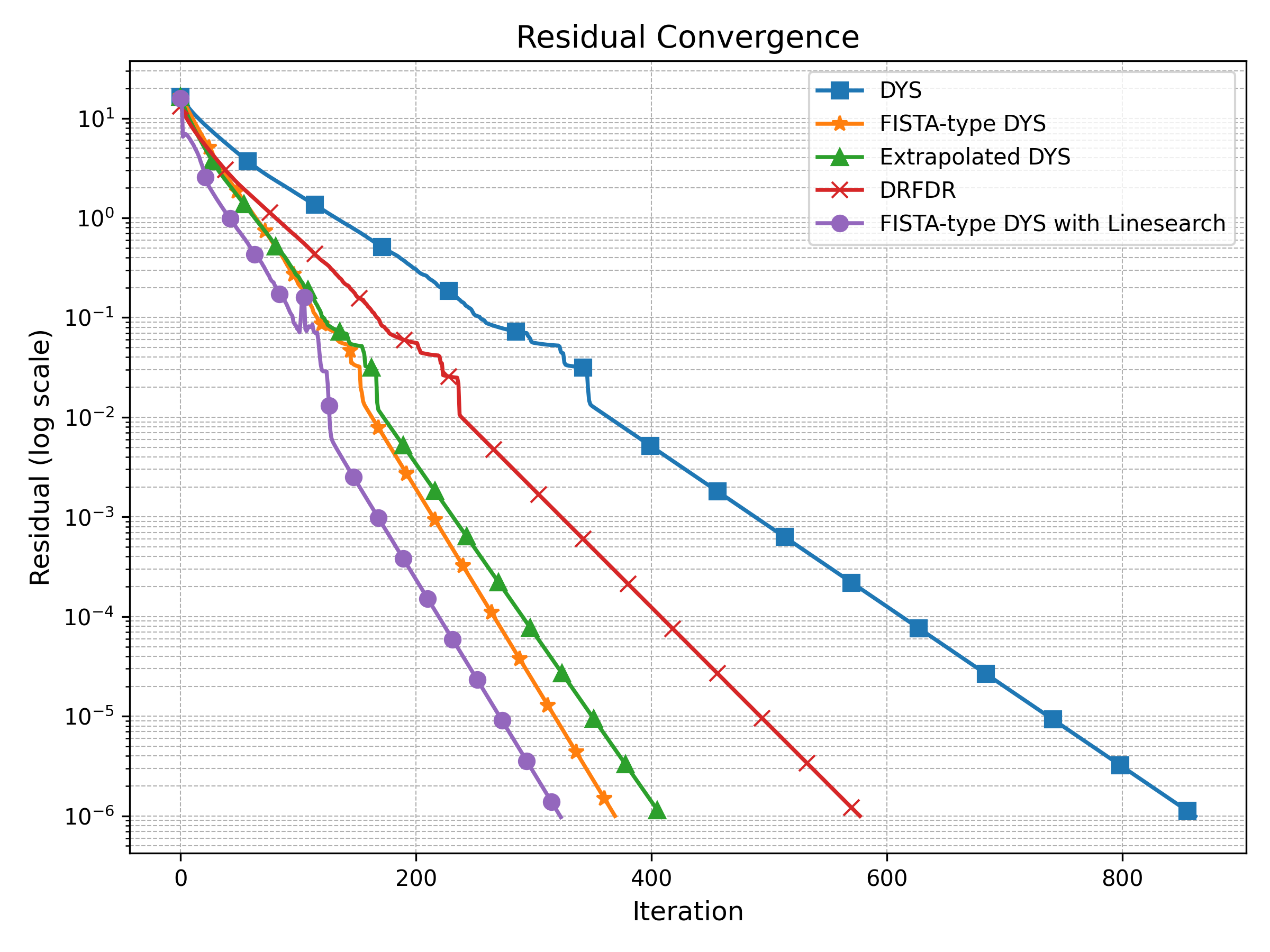}
        \caption{Residual Convergence with $\gamma=0.035$}
        \label{fig:Final_Results/Nonnegative_LRM/gamma_0.035_residual_convergence.png}
    \end{subfigure}
    \hfill
    \begin{subfigure}{0.42\textwidth}
        \centering
        \includegraphics[width=\linewidth]{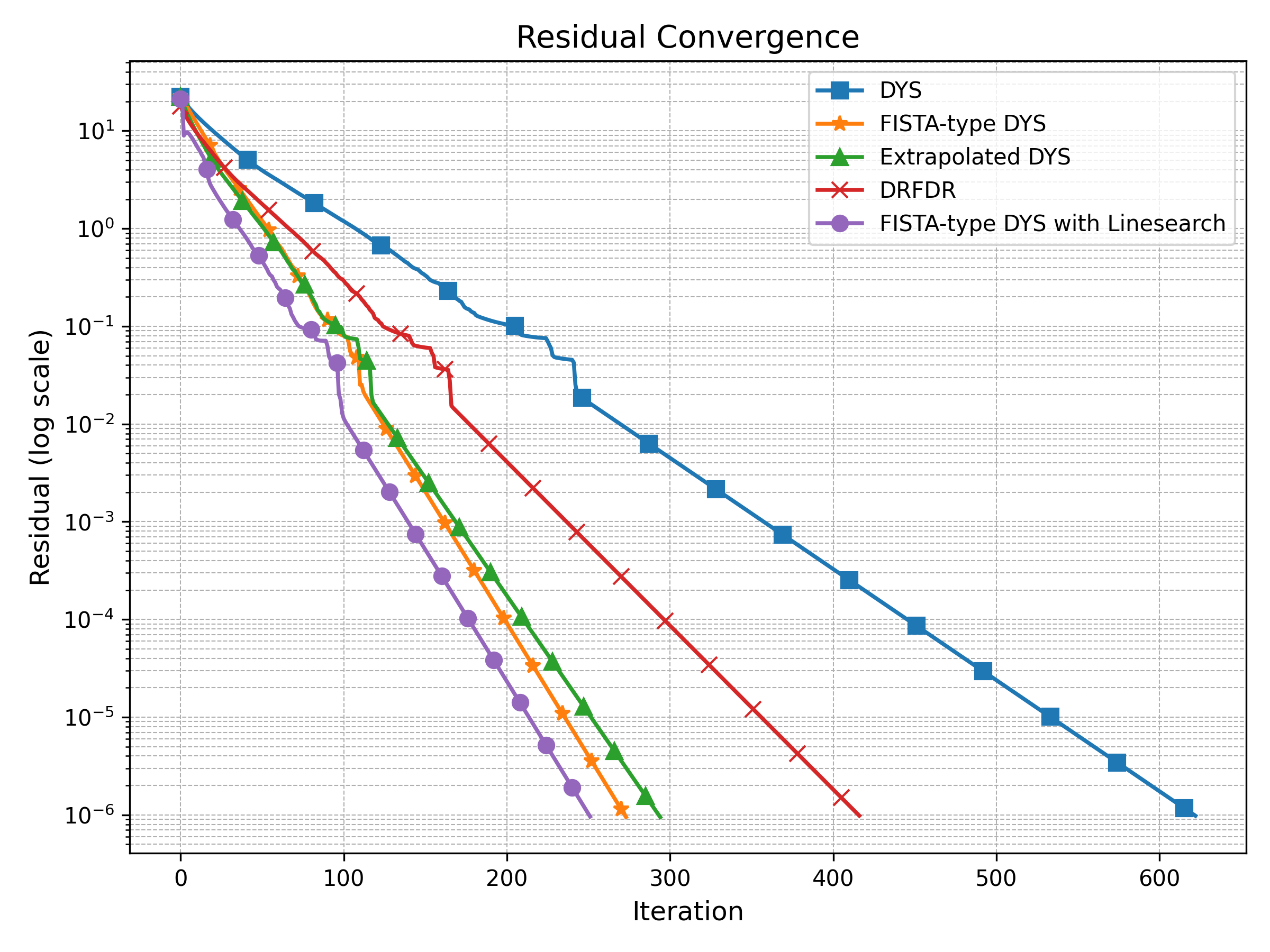}
        \caption{Residual Convergence with $\gamma=0.05$}
        \label{fig:Final_Results/Nonnegative_LRM/gamma_0.05_residual_convergence.png}
    \end{subfigure}
\caption{Residual convergence vs iterations for the nonnegative LRM problem using randomized data}
\label{fig:Residual Convergence graph for randomized data}
\end{figure}

\begin{table}[h]
\centering
\resizebox{1.0\linewidth}{!}{%
\begin{tabular}{c cc cc}
\toprule
\textbf{Values of $\gamma$} & \multicolumn{2}{c}{0.035} & \multicolumn{2}{c}{0.05}\\
\midrule
\textbf{Algorithms} & Iterations & times(in sec) & Iterations & times(in sec)\\
\midrule
DYS & 863 & 0.6712 & 623 & 0.5370\\
\midrule
Extrapolated-DYS & 410 & 0.2516 & 295 & 0.2336\\
\midrule
DRFDR & 578 & 0.3550 & 417 & 0.3512\\
\midrule
\textbf{FISTA-type DYS} & \textbf{370} & \textbf{0.2856} & \textbf{274} & \textbf{0.2633}\\
\midrule
\textbf{FISTA-type DYS with Line-search} & \textbf{324} & \textbf{1032.7709} & \textbf{252} & \textbf{831.4558}\\
\bottomrule
\end{tabular}
}
\caption{Number of iterations and total computational time taken by different algorithms to achieve the residual error $<10^{-6}$ for various values of $\gamma$}
\label{table:LRM problem}
\end{table}

\subsubsection{Using Given Data}
In this case, the matrix $M$ is generated from a dataset from the ``2025 Distribution zone substation data" published by Ausgrid. We use the load data from five 11kV meters, namely ``Beacon Hill", ``Botany", ``City East", ``Belrose", and ``Surry Hills", and form a $400\times364$ matrix. Then we use our proposed algorithms. But this is a large-scale problem, so for this experiment the Modified Broyden method does not work properly; it takes huge computational time and space, so for this type of problem we use L-BFGS \cite{liu1989limited,chen2014large,al2014broyden} in place of Modified Broyden to avoid the computation of the Hessian directly, and use an approximated Hessian. Here, we choose the parameters $\alpha=0.5,\;\beta=0.5,\;\lambda=0.9,\;\eta=1.6,\;\theta=0.5,\;\iota_1=0.8,\;\iota_2=0.1,\;\text{err}=10^{-6}$, which are satisfied our decent properties. Now, in Figure \ref{fig:Final_Results/LRM_with_Data/gamma_0.03_residual_convergence.png} and Figure \ref{fig:Final_Results/LRM_with_Data/gamma_0.05_residual_convergence.png}, we see how the residual values behave for different algorithms for the choices of $\gamma=0.03$ and $\gamma=0.05$, respectively. Then, in Table \ref{table:LRM problem with data}, we see how many iterations and computational time (in sec) are required by different algorithms for the same values of parameters. From Figure \ref{fig:Final_Results/LRM_with_Data/gamma_0.03_residual_convergence.png}, Figure \ref{fig:Final_Results/LRM_with_Data/gamma_0.05_residual_convergence.png}, and Table \ref{table:LRM problem with data}, we observe that, the residual values for Algorithm \ref{FISTA-DYS} and \ref{FISTA-DYS-LS} converge faster to $<\text{err}$ than DYS, extrapolated-DYS, and DRFDR algorithms.
\begin{figure}[h]
    \centering
    \begin{subfigure}{0.42\textwidth}
        \centering
        \includegraphics[width=\linewidth]{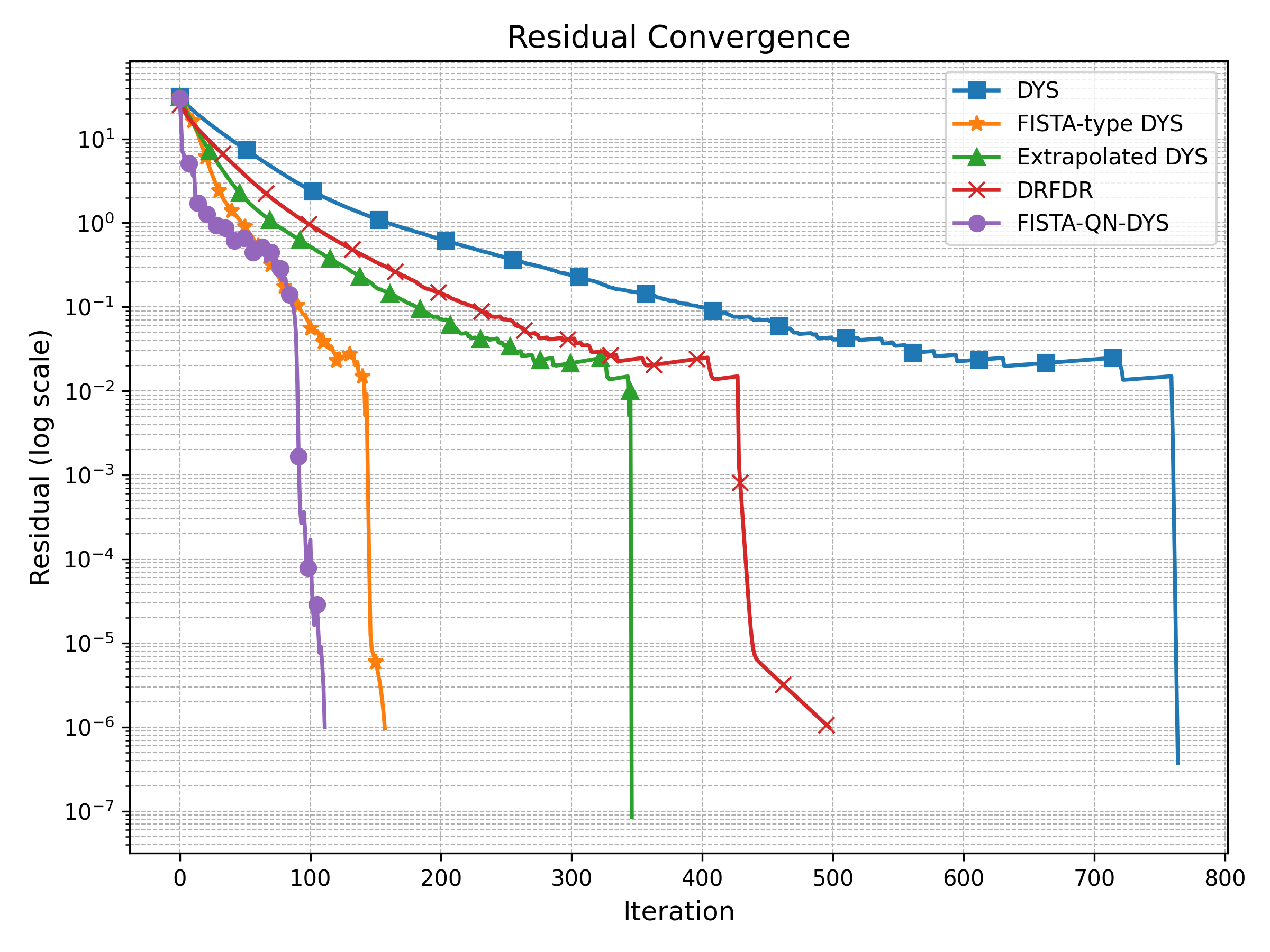}
        \caption{Residual Convergence with $\gamma=0.03$}
        \label{fig:Final_Results/LRM_with_Data/gamma_0.03_residual_convergence.png}
    \end{subfigure}
    \hfill
    \begin{subfigure}{0.42\textwidth}
        \centering
        \includegraphics[width=\linewidth]{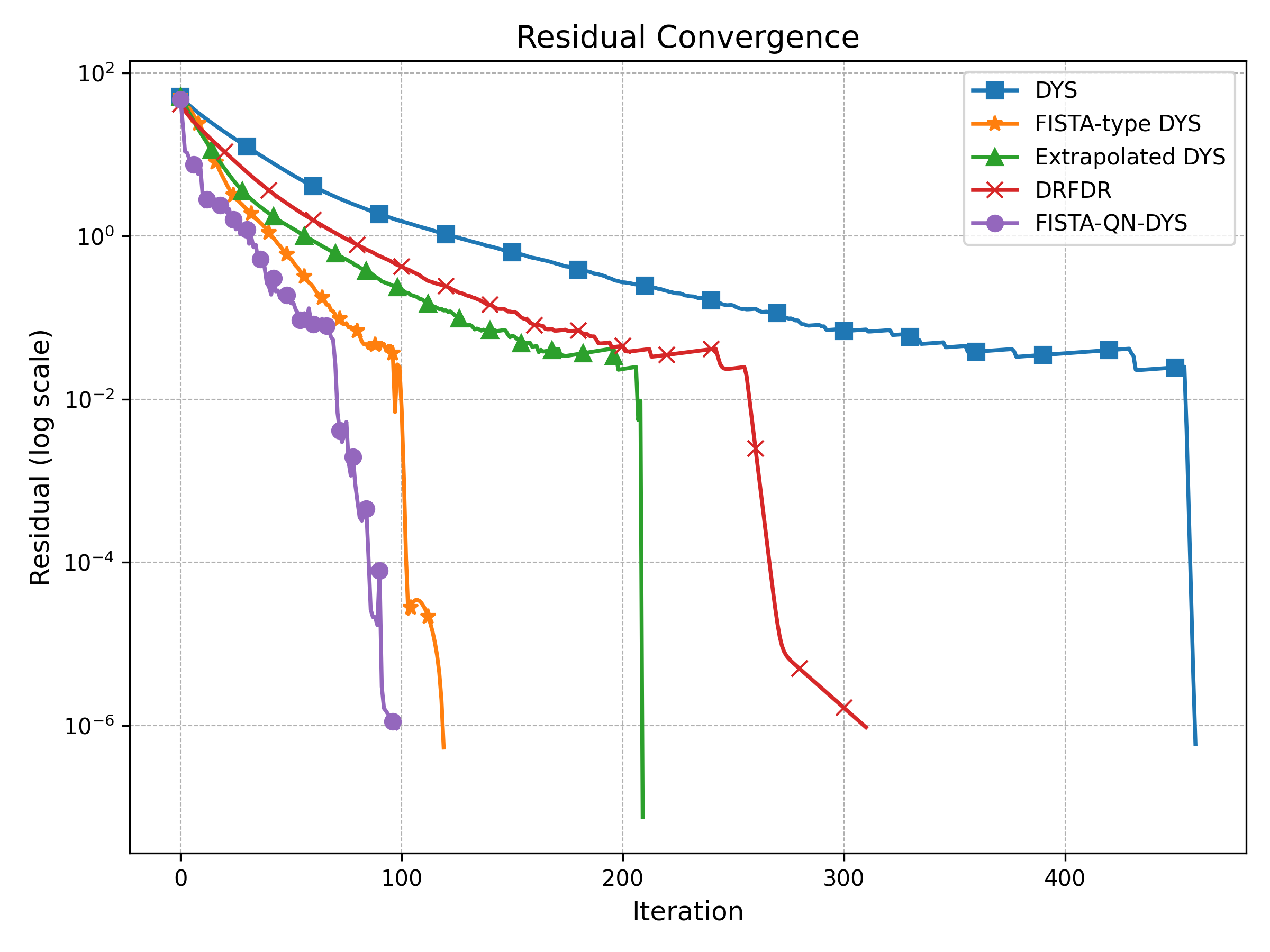}
        \caption{Residual Convergence with $\gamma=0.05$}
        \label{fig:Final_Results/LRM_with_Data/gamma_0.05_residual_convergence.png}
    \end{subfigure}
\caption{Residual convergence vs iterations for the nonnegative LRM problem using the given data}
\label{fig:Residual Convergence graph for given data}
\end{figure}
\begin{table}[h]
\centering
\resizebox{1.0\linewidth}{!}{%
\begin{tabular}{c cc cc}
\toprule
\textbf{Values of $\gamma$} & \multicolumn{2}{c}{0.03} & \multicolumn{2}{c}{0.05}\\
\midrule
\textbf{Algorithms} & Iterations & times(in sec) & Iterations & times(in sec)\\
\midrule
DYS & 765 & 0.6781 & 460 & 0.4563\\
\midrule
Extrapolated-DYS & 347 & 0.3435 & 210 & 0.1977\\
\midrule
DRFDR & 498 & 0.4504 & 311 & 0.2576\\
\midrule
\textbf{FISTA-type DYS} & \textbf{158} & \textbf{0.1829} & \textbf{120} & \textbf{0.1540}\\
\midrule
\textbf{FISTA-type DYS with Line-search} & \textbf{112} & \textbf{0.8145} & \textbf{99} & \textbf{0.6340}\\
\bottomrule
\end{tabular}
}
\caption{Number of iterations and total computational time taken by different algorithms to achieve the residual error $<10^{-6}$ for various values of $\gamma$}
\label{table:LRM problem with data}
\end{table}

\subsection{Image Restoration}

In this experiment, we consider an image restoration problem modeled as
\begin{equation}
    b=Ax+\xi,\label{image restoration problem}
\end{equation}
where $A$ is a blur operator, $x$ is the real image, $b$ is the observed image, and $\xi$ is an additive Gaussian noise with distribution $\mathcal{N}(0,\nu^2).$ Now, recovering the real image from this model \eqref{image restoration problem} is a very difficult task due to the presence of noise and the blur operator. So, by using the negative log-likelihood formulation for Gaussian noise in the inverse problem \eqref{image restoration problem} and then using the regularization term $\frac{1}{\eta}\phi_{\sigma}(x)$ and $\frac{\beta}{2}\|x\|^2$ as discussed in \cite{wu2024extrapolated}, we get

\begin{equation}
    \min_{x\in\mathbb{R}^n}\frac{1}{2\nu}\|Ax-b\|^2+\frac{1}{\eta}\phi_{\sigma}(x)+\frac{\beta}{2}\|x\|^2,\label{reformulated image restoration problem}
\end{equation}
where $\eta>0$ and $\beta>0$ are the regularizing parameters and $\nu$ and $\sigma$ are noise levels. Now, we choose $f(x)=\frac{1}{2\nu}\|Ax-b\|^2$ and $h(x)=\frac{\beta}{2}\|x\|^2$ and form the problem \eqref{reformulated image restoration problem} to \eqref{Prob 2}. This is also a large-scale problem, so we replace the modified Broyden update rule with L-BFGS. Here, $f$ and $h$ are $L_f,\;L_h$-smooth, respectively, where $L_f=\frac{\|A^{\top}A\|}{\nu^2}$ and $L_h=\beta,$ so the Assumption \ref{Assumption 2} is satisfied. Now, we take the noise levels as $2.55,~7.65,~12.75$ and the corresponding step sizes ($\gamma$) as $0.7,~0.45,~0.3$, respectively, and a step size of $0.5$ is chosen for other cases. We also take the relaxation parameter $\lambda=0.8$ and the extrapolated parameter $\alpha=0.5$. Here, we also consider that the iteration terminates if the relative difference between two consecutive objective values is less than $10^{-8}$. Now, we compare the peak signal-to-noise ratio (PSNR) values of some images after deblurring by our proposed Algorithms \ref{FISTA-PnP-DYS}, \ref{FISTA-PnP-DYS-LS}, and the extrapolated PnP-DYS algorithm \cite{wu2024extrapolated} and the DRFDR with PnP algorithm. In Figures \ref{fig:Butterfly images for 2.55}, \ref{fig:Leaves images for 2.55}, \ref{fig:Starfish images for 2.55}, \ref{fig:Butterfly images for 7.65}, \ref{fig:Leaves images for 7.65}, \ref{fig:Starfish images for 7.65}, \ref{fig:Butterfly images for 12.75}, \ref{fig:Leaves images for 12.75}, \ref{fig:Starfish images for 12.75}, we see how the original input images change after blurring for various algorithms and various noise levels. In Figures \ref{fig:PSNR graphs for noise level 2.55}, \ref{fig:PSNR graphs for noise level 7.65}, \ref{fig:PSNR graphs for noise level 12.75}, we see how the PSNR values change in each iteration for various algorithms and various noise levels. In Table \ref{table:Image restoration problem}, we see how our proposed algorithms work better than extrapolated PnP-DYS and DRFDR with PnP algorithms to restore images from blurred images for various noise levels. For this experiment, we take the help of a code, which is available on GitHub\footnote{\url{https://github.com/Huang-chao-yan/convergent_pnp/tree/main/PnP_restoration}}.
\begin{figure}[H]
    \centering
    \begin{subfigure}{0.18\textwidth}
        \centering
        \fbox{\includegraphics[width=\linewidth]{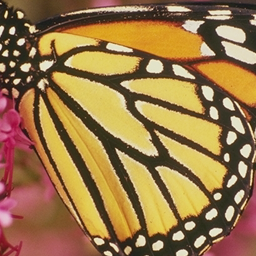}}
        \caption{Input image}
        \label{fig:Input butterfly image for noise level 2.55}
    \end{subfigure}
    \hfill
    \begin{subfigure}{0.18\textwidth}
        \centering
        \fbox{\includegraphics[width=\linewidth]{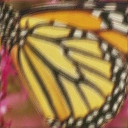}}
        \caption{Blurred image}
        \label{fig:Blurred butterfly image for noise level 2.55}
    \end{subfigure}
    \hfill
    \begin{subfigure}{0.18\textwidth}
        \centering
        \fbox{\includegraphics[width=\linewidth]{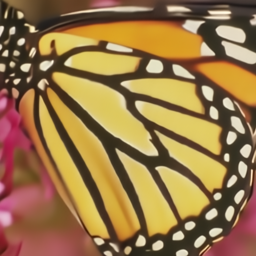}}
        \caption{Extrapolated PnP-DYS}
        \label{fig:Extrapolated_PnP butterfly image for noise level 2.55}
    \end{subfigure}

    \vspace{0.1cm}

    \begin{subfigure}{0.18\textwidth}
        \centering
        \fbox{\includegraphics[width=\linewidth]{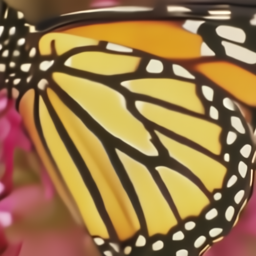}}
        \caption{PnP-DRFDR}
        \label{fig:PnP-DRFDR butterfly image for noise level 2.55}
    \end{subfigure}
    \hfill
    \begin{subfigure}{0.18\textwidth}
        \centering
        \fbox{\includegraphics[width=\linewidth]{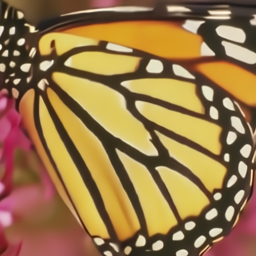}}
        \caption{FISTA-type PnP-DYS}
        \label{fig:FISTA_PnP_DYS butterfly image for noise level 2.55}
    \end{subfigure}
    \hfill
    \begin{subfigure}{0.18\textwidth}
        \centering
        \fbox{\includegraphics[width=\linewidth]{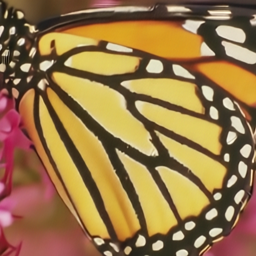}}
        \caption{FISTA-type PnP-DYS with Line-search}
        \label{fig:FISTA-type PnP-DYS with Line-search butterfly image for noise level 2.55}
    \end{subfigure}
    \caption{Images of the butterfly at noise level $2.55$}
    \label{fig:Butterfly images for 2.55}
\end{figure}

\begin{figure}[H]
    \centering
    \begin{subfigure}{0.18\textwidth}
        \centering
        \fbox{\includegraphics[width=\linewidth]{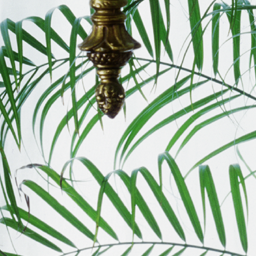}}
        \caption{Input image}
        \label{fig:Input leaves image for noise level 2.55}
    \end{subfigure}
    \hfill
    \begin{subfigure}{0.18\textwidth}
        \centering
        \fbox{\includegraphics[width=\linewidth]{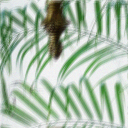}}
        \caption{Blurred image}
        \label{fig:Blurred leaves image for noise level 2.55}
    \end{subfigure}
    \hfill
    \begin{subfigure}{0.18\textwidth}
        \centering
        \fbox{\includegraphics[width=\linewidth]{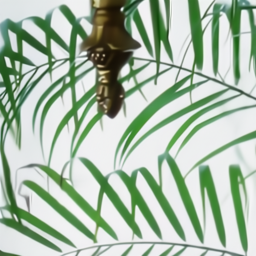}}
        \caption{Extrapolated PnP-DYS}
        \label{fig:Extrapolated_PnP leaves image for noise level 2.55}
    \end{subfigure}

    \vspace{0.1cm}

    \begin{subfigure}{0.18\textwidth}
        \centering
        \fbox{\includegraphics[width=\linewidth]{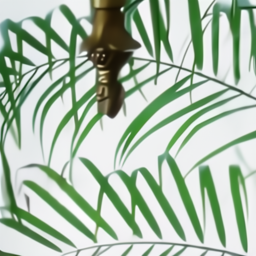}}
        \caption{PnP-DRFDR}
        \label{fig:PnP-DRFDR leaves image for noise level 2.55}
    \end{subfigure}
    \hfill
    \begin{subfigure}{0.18\textwidth}
        \centering
        \fbox{\includegraphics[width=\linewidth]{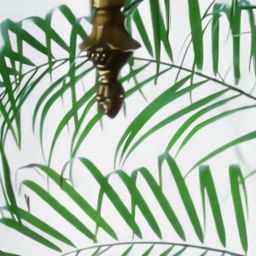}}
        \caption{FISTA-type PnP-DYS}
        \label{fig:FISTA_PnP_DYS leaves image for noise level 2.55}
    \end{subfigure}
    \hfill
    \begin{subfigure}{0.18\textwidth}
        \centering
        \fbox{\includegraphics[width=\linewidth]{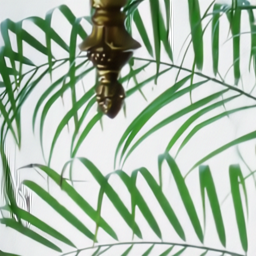}}
        \caption{FISTA-type PnP-DYS with Line-search}
        \label{fig:FISTA-type PnP-DYS with Line-search leaves image for noise level 2.55}
    \end{subfigure}
    \caption{Images of leaves at noise level $2.55$}
    \label{fig:Leaves images for 2.55}
\end{figure}

\begin{figure}[H]
    \centering
    \begin{subfigure}{0.18\textwidth}
        \centering
        \fbox{\includegraphics[width=\linewidth]{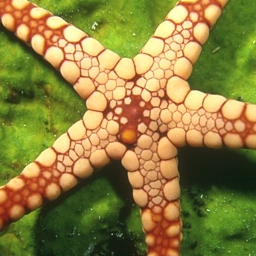}}
        \caption{Input image}
        \label{fig:Input starfish image for noise level 2.55}
    \end{subfigure}
    \hfill
    \begin{subfigure}{0.18\textwidth}
        \centering
        \fbox{\includegraphics[width=\linewidth]{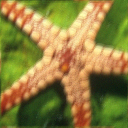}}
        \caption{Blurred image}
        \label{fig:Blurred starfish image for noise level 2.55}
    \end{subfigure}
    \hfill
    \begin{subfigure}{0.18\textwidth}
        \centering
        \fbox{\includegraphics[width=\linewidth]{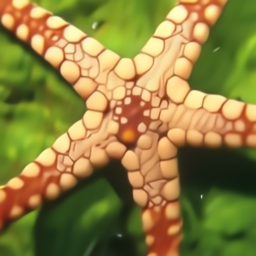}}
        \caption{Extrapolated PnP-DYS}
        \label{fig:Extrapolated_PnP starfish image for noise level 2.55}
    \end{subfigure}

    \vspace{0.1cm}

    \begin{subfigure}{0.18\textwidth}
        \centering
        \fbox{\includegraphics[width=\linewidth]{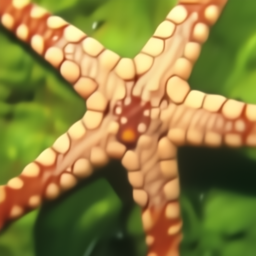}}
        \caption{PnP-DRFDR}
        \label{fig:PnP-DRFDR starfish image for noise level 2.55}
    \end{subfigure}
    \hfill
    \begin{subfigure}{0.18\textwidth}
        \centering
        \fbox{\includegraphics[width=\linewidth]{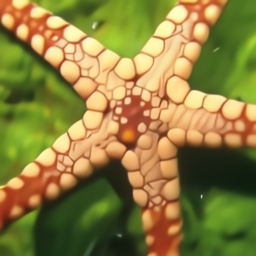}}
        \caption{FISTA-type PnP-DYS}
        \label{fig:FISTA_PnP_DYS starfish image for noise level 2.55}
    \end{subfigure}
    \hfill
    \begin{subfigure}{0.18\textwidth}
        \centering
        \fbox{\includegraphics[width=\linewidth]{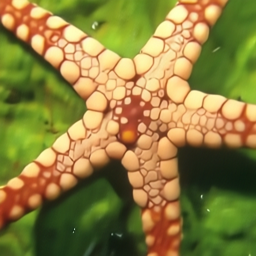}}
        \caption{FISTA-type PnP-DYS with Line-search}
        \label{fig:FISTA-type PnP-DYS with Line-search starfish image for noise level 2.55}
    \end{subfigure}
    \caption{Images of Starfish in noise level $2.55$}
    \label{fig:Starfish images for 2.55}
\end{figure}

\begin{figure}[H]
    \centering
    \begin{subfigure}{0.34\textwidth}
        \includegraphics[width=\linewidth]{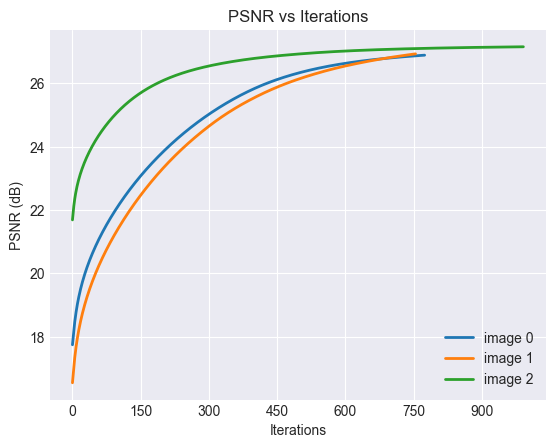}
        \caption{Extrapolated PnP-DYS}
        \label{fig:Extrapolated PnP PSNR for noise level 2.55}
    \end{subfigure}
    \hfill
    \begin{subfigure}{0.34\textwidth}
        \includegraphics[width=\linewidth]{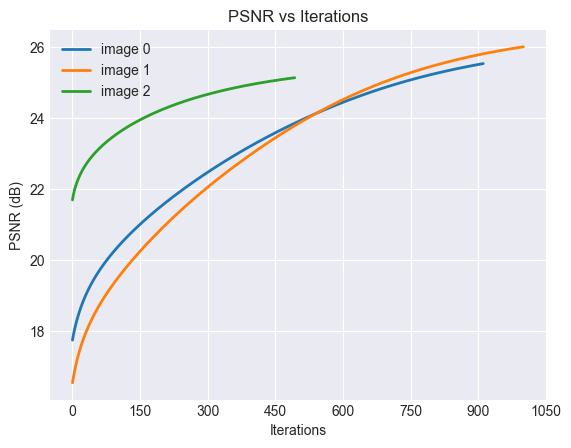}
        \caption{PnP-DRFDR}
        \label{fig:PnP-DRFDR PSNR for noise level 2.55}
    \end{subfigure}
    \hfill
    \begin{subfigure}{0.34\textwidth}
        \includegraphics[width=\linewidth]{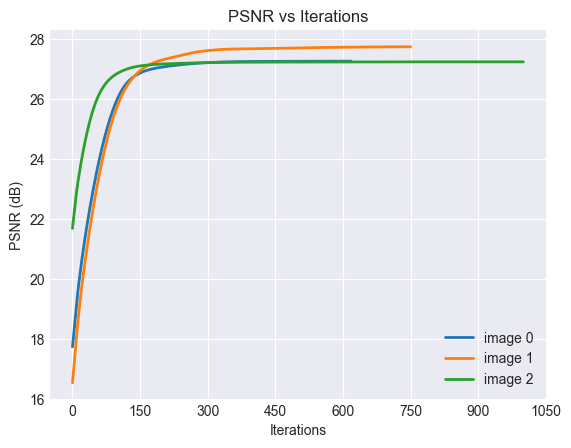}
        \caption{FISTA-type PnP-DYS}
        \label{fig:FISTA-type PnP-DYS PSNR for noise level 2.55}
    \end{subfigure}
    \hfill
    \begin{subfigure}{0.34\textwidth}
        \includegraphics[width=\linewidth]{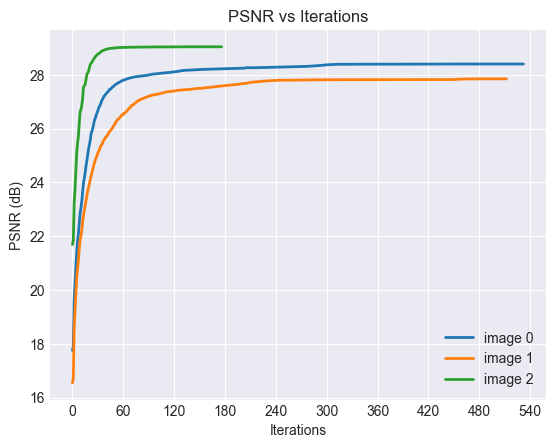}
        \caption{FISTA-type PnP-DYS with Line-search}
        \label{fig:FISTA-type PnP-DYS with Line-search PSNR for noise level 2.55}
    \end{subfigure}
    \caption{Iterations vs PSNR value graphs for noise level $2.55$}
    \label{fig:PSNR graphs for noise level 2.55}
\end{figure}

\begin{figure}[H]
    \centering
    \begin{subfigure}{0.18\textwidth}
        \centering
        \fbox{\includegraphics[width=\linewidth]{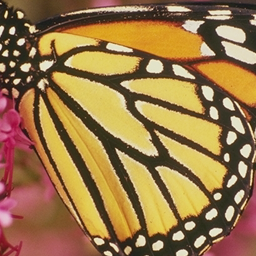}}
        \caption{Input image}
        \label{fig:Input butterfly image for noise level 7.65}
    \end{subfigure}
    \hfill
    \begin{subfigure}{0.18\textwidth}
        \centering
        \fbox{\includegraphics[width=\linewidth]{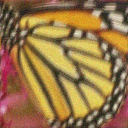}}
        \caption{Blurred image}
        \label{fig:Blurred butterfly image for noise level 7.65}
    \end{subfigure}
    \hfill
    \begin{subfigure}{0.18\textwidth}
        \centering
        \fbox{\includegraphics[width=\linewidth]{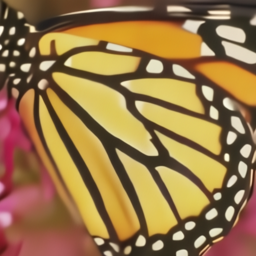}}
        \caption{Extrapolated PnP-DYS}
        \label{fig:Extrapolated_PnP butterfly image for noise level 7.65}
    \end{subfigure}

    \vspace{0.1cm}

    \begin{subfigure}{0.18\textwidth}
        \centering
        \fbox{\includegraphics[width=\linewidth]{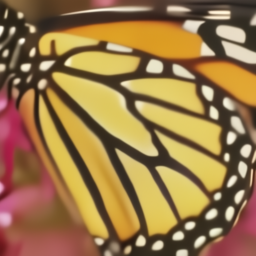}}
        \caption{PnP-DRFDR}
        \label{fig:PnP-DRFDR butterfly image for noise level 7.65}
    \end{subfigure}
    \hfill
    \begin{subfigure}{0.18\textwidth}
        \centering
        \fbox{\includegraphics[width=\linewidth]{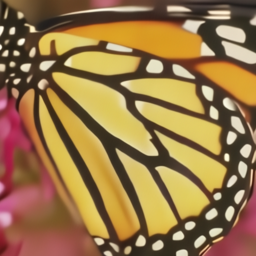}}
        \caption{FISTA-type PnP-DYS}
        \label{fig:FISTA_PnP_DYS butterfly image for noise level 7.65}
    \end{subfigure}
    \hfill
    \begin{subfigure}{0.18\textwidth}
        \centering
        \fbox{\includegraphics[width=\linewidth]{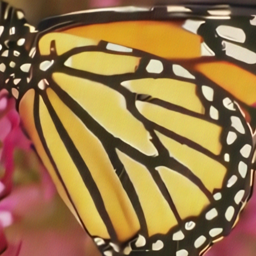}}
        \caption{FISTA-type PnP-DYS with Line-search}
        \label{fig:FISTA-type PnP-DYS with Line-search butterfly image for noise level 7.65}
    \end{subfigure}
    \caption{Images of the butterfly at noise level $7.65$}
    \label{fig:Butterfly images for 7.65}
\end{figure}

\begin{figure}[H]
    \centering
    \begin{subfigure}{0.18\textwidth}
        \centering
        \fbox{\includegraphics[width=\linewidth]{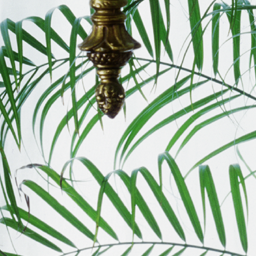}}
        \caption{Input image}
        \label{fig:Input leaves image for noise level 7.65}
    \end{subfigure}
    \hfill
    \begin{subfigure}{0.18\textwidth}
        \centering
        \fbox{\includegraphics[width=\linewidth]{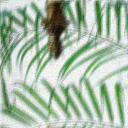}}
        \caption{Blurred image}
        \label{fig:Blurred leaves image for noise level 7.65}
    \end{subfigure}
    \hfill
    \begin{subfigure}{0.18\textwidth}
        \centering
        \fbox{\includegraphics[width=\linewidth]{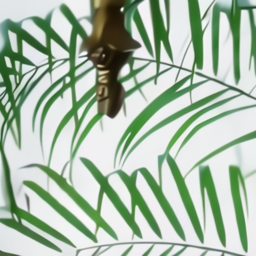}}
        \caption{Extrapolated PnP-DYS}
        \label{fig:Extrapolated_PnP leaves image for noise level 7.65}
    \end{subfigure}

    \vspace{0.1cm}

    \begin{subfigure}{0.18\textwidth}
        \centering
        \fbox{\includegraphics[width=\linewidth]{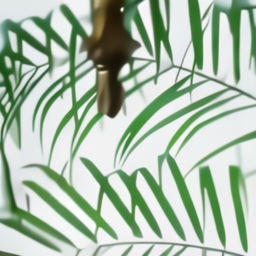}}
        \caption{PnP-DRFDR}
        \label{fig:PnP-DRFDR leaves image for noise level 7.65}
    \end{subfigure}
    \hfill
    \begin{subfigure}{0.18\textwidth}
        \centering
        \fbox{\includegraphics[width=\linewidth]{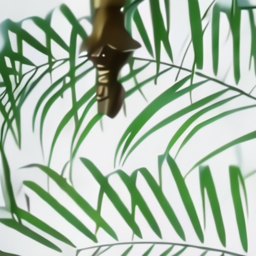}}
        \caption{FISTA-type PnP-DYS}
        \label{fig:FISTA_PnP_DYS leaves image for noise level 7.65}
    \end{subfigure}
    \hfill
    \begin{subfigure}{0.18\textwidth}
        \centering
        \fbox{\includegraphics[width=\linewidth]{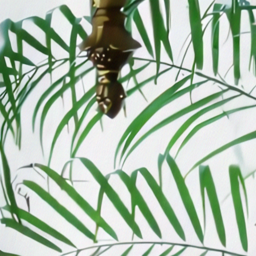}}
        \caption{FISTA-type PnP-DYS with Line-search}
        \label{fig:FISTA-type PnP-DYS with Line-search leaves image for noise level 7.65}
    \end{subfigure}
    \caption{Images of leaves at noise level $7.65$}
    \label{fig:Leaves images for 7.65}
\end{figure}

\begin{figure}[H]
    \centering
    \begin{subfigure}{0.18\textwidth}
        \centering
        \fbox{\includegraphics[width=\linewidth]{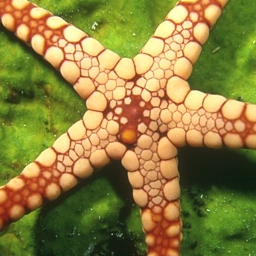}}
        \caption{Input image}
        \label{fig:Input starfish image for noise level 7.65}
    \end{subfigure}
    \hfill
    \begin{subfigure}{0.18\textwidth}
        \centering
        \fbox{\includegraphics[width=\linewidth]{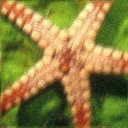}}
        \caption{Blurred image}
        \label{fig:Blurred starfish image for noise level 7.65}
    \end{subfigure}
    \hfill
    \begin{subfigure}{0.18\textwidth}
        \centering
        \fbox{\includegraphics[width=\linewidth]{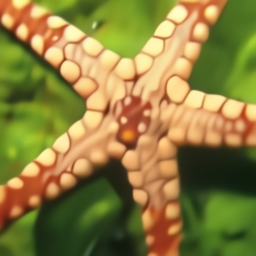}}
        \caption{Extrapolated PnP-DYS}
        \label{fig:Extrapolated_PnP starfish image for noise level 7.65}
    \end{subfigure}

    \vspace{0.1cm}

    \begin{subfigure}{0.18\textwidth}
        \centering
        \fbox{\includegraphics[width=\linewidth]{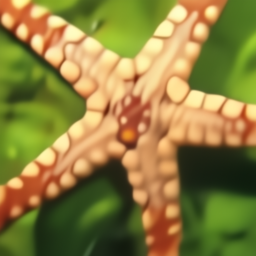}}
        \caption{PnP-DRFDR}
        \label{fig:PnP-DRFDR starfish image for noise level 7.65}
    \end{subfigure}
    \hfill
    \begin{subfigure}{0.18\textwidth}
        \centering
        \fbox{\includegraphics[width=\linewidth]{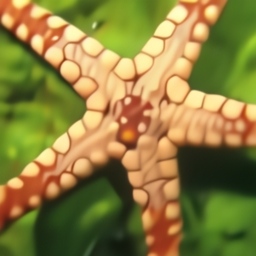}}
        \caption{FISTA-type PnP-DYS}
        \label{fig:FISTA_PnP_DYS starfish image for noise level 7.65}
    \end{subfigure}
    \hfill
    \begin{subfigure}{0.18\textwidth}
        \centering
        \fbox{\includegraphics[width=\linewidth]{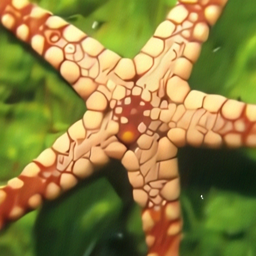}}
        \caption{FISTA-type PnP-DYS with Line-search}
        \label{fig:FISTA-type PnP-DYS with Line-search starfish image for noise level 7.65}
    \end{subfigure}
    \caption{Images of Starfish in noise level $7.65$}
    \label{fig:Starfish images for 7.65}
\end{figure}

\begin{figure}[H]
    \centering
    \begin{subfigure}{0.34\textwidth}
        \includegraphics[width=\linewidth]{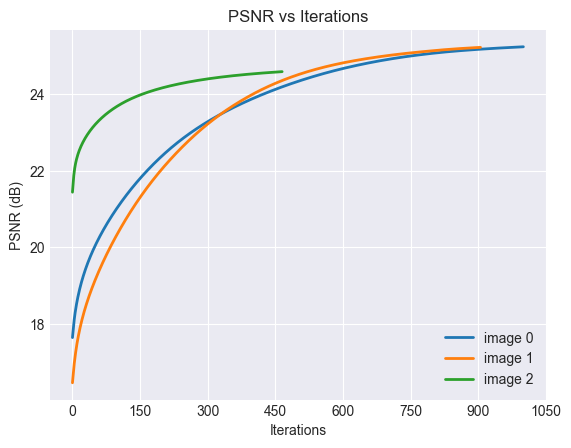}
        \caption{Extrapolated PnP-DYS}
        \label{fig:Extrapolated PnP PSNR for noise level 7.65}
    \end{subfigure}
    \hfill
    \begin{subfigure}{0.34\textwidth}
        \includegraphics[width=\linewidth]{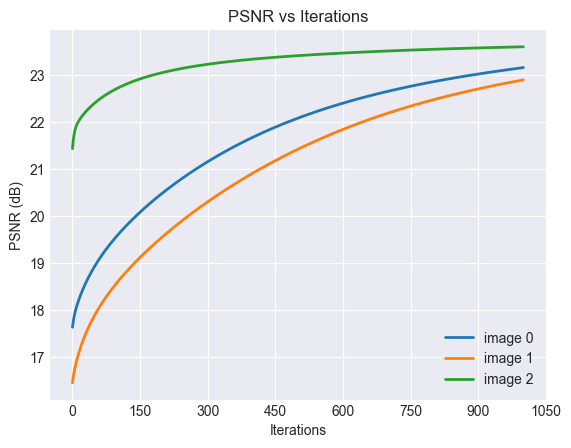}
        \caption{PnP-DRFDR}
        \label{fig:PnP-DRFDR PSNR for noise level 7.65}
    \end{subfigure}
    \hfill
    \begin{subfigure}{0.34\textwidth}
        \includegraphics[width=\linewidth]{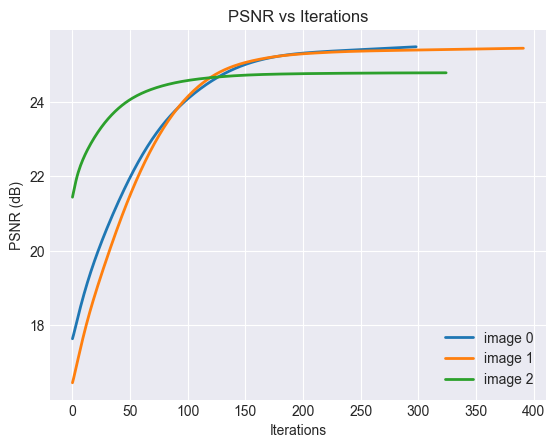}
        \caption{FISTA-type PnP-DYS}
        \label{fig:FISTA-type PnP-DYS PSNR for noise level 7.65}
    \end{subfigure}
    \hfill
    \begin{subfigure}{0.34\textwidth}
        \includegraphics[width=\linewidth]{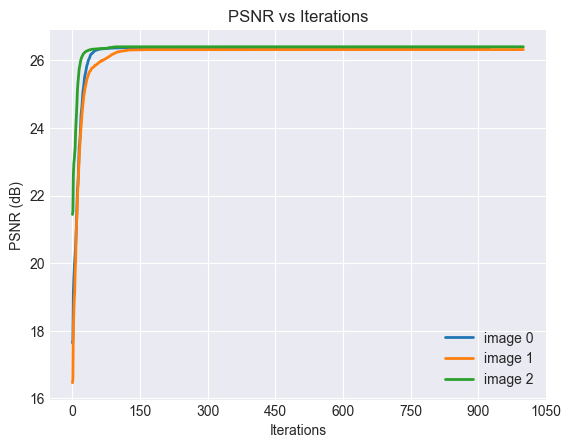}
        \caption{FISTA-type PnP-DYS with Line-search}
        \label{fig:FISTA-type PnP-DYS with Line-search PSNR for noise level 7.65}
    \end{subfigure}
    \caption{Iterations vs PSNR value graphs for noise level $7.65$}
    \label{fig:PSNR graphs for noise level 7.65}
\end{figure}

\begin{figure}[H]
    \centering
    \begin{subfigure}{0.18\textwidth}
        \centering
        \fbox{\includegraphics[width=\linewidth]{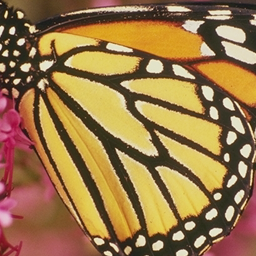}}
        \caption{Input image}
        \label{fig:Input butterfly image for noise level 12.75}
    \end{subfigure}
    \hfill
    \begin{subfigure}{0.18\textwidth}
        \centering
        \fbox{\includegraphics[width=\linewidth]{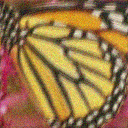}}
        \caption{Blurred image}
        \label{fig:Blurred butterfly image for noise level 12.75}
    \end{subfigure}
    \hfill
    \begin{subfigure}{0.18\textwidth}
        \centering
        \fbox{\includegraphics[width=\linewidth]{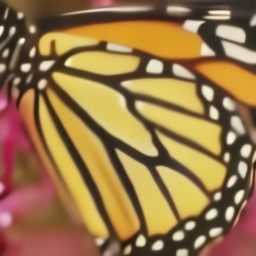}}
        \caption{Extrapolated PnP-DYS}
        \label{fig:Extrapolated_PnP butterfly image for noise level 12.75}
    \end{subfigure}

    \vspace{0.1cm}

    \begin{subfigure}{0.18\textwidth}
        \centering
        \fbox{\includegraphics[width=\linewidth]{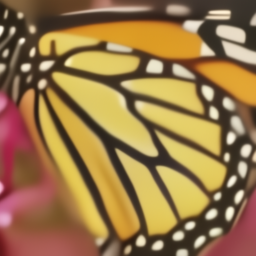}}
        \caption{PnP-DRFDR}
        \label{fig:PnP-DRFDR butterfly image for noise level 12.75}
    \end{subfigure}
    \hfill
    \begin{subfigure}{0.18\textwidth}
        \centering
        \fbox{\includegraphics[width=\linewidth]{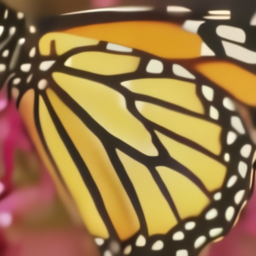}}
        \caption{FISTA-type PnP-DYS}
        \label{fig:FISTA_PnP_DYS butterfly image for noise level 12.75}
    \end{subfigure}
    \hfill
    \begin{subfigure}{0.18\textwidth}
        \centering
        \fbox{\includegraphics[width=\linewidth]{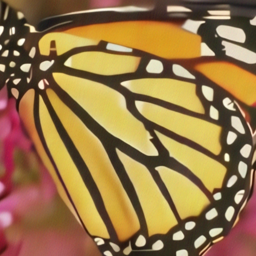}}
        \caption{FISTA-type PnP-DYS with Line-search}
        \label{fig:FISTA-type PnP-DYS with Line-search butterfly image for noise level 12.75}
    \end{subfigure}
    \caption{Images of Butterfly in noise level $12.75$}
    \label{fig:Butterfly images for 12.75}
\end{figure}

\begin{figure}[H]
    \centering
    \begin{subfigure}{0.18\textwidth}
        \centering
        \fbox{\includegraphics[width=\linewidth]{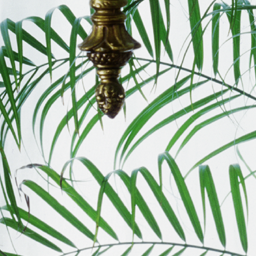}}
        \caption{Input image}
        \label{fig:Input leaves image for noise level 12.75}
    \end{subfigure}
    \hfill
    \begin{subfigure}{0.18\textwidth}
        \centering
        \fbox{\includegraphics[width=\linewidth]{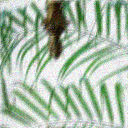}}
        \caption{Blurred image}
        \label{fig:Blurred leaves image for noise level 12.75}
    \end{subfigure}
    \hfill
    \begin{subfigure}{0.18\textwidth}
        \centering
        \fbox{\includegraphics[width=\linewidth]{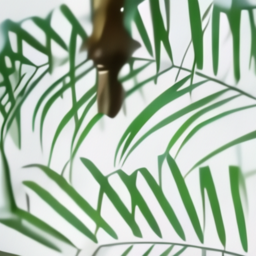}}
        \caption{Extrapolated PnP-DYS}
        \label{fig:Extrapolated_PnP leaves image for noise level 12.75}
    \end{subfigure}

    \vspace{0.1cm}

    \begin{subfigure}{0.18\textwidth}
        \centering
        \fbox{\includegraphics[width=\linewidth]{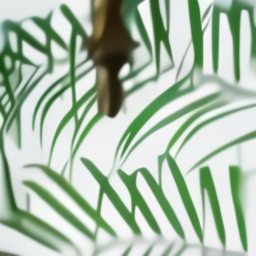}}
        \caption{PnP-DRFDR}
        \label{fig:PnP-DRFDR leaves image for noise level 12.75}
    \end{subfigure}
    \hfill
    \begin{subfigure}{0.18\textwidth}
        \centering
        \fbox{\includegraphics[width=\linewidth]{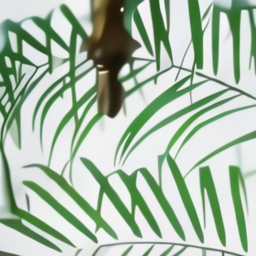}}
        \caption{FISTA-type PnP-DYS}
        \label{fig:FISTA_PnP_DYS leaves image for noise level 12.75}
    \end{subfigure}
    \hfill
    \begin{subfigure}{0.18\textwidth}
        \centering
        \fbox{\includegraphics[width=\linewidth]{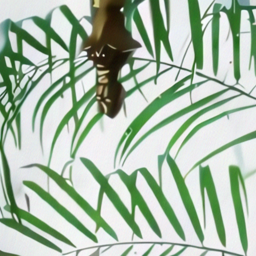}}
        \caption{FISTA-type PnP-DYS with Line-search}
        \label{fig:FISTA-type PnP-DYS with Line-search leaves image for noise level 12.75}
    \end{subfigure}
    \caption{Images of Leaves in noise level $12.75$}
    \label{fig:Leaves images for 12.75}
\end{figure}

\begin{figure}[H]
    \centering
    \begin{subfigure}{0.18\textwidth}
        \centering
        \fbox{\includegraphics[width=\linewidth]{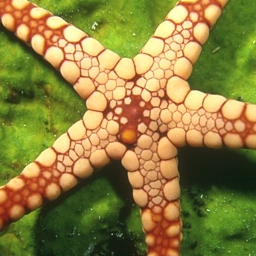}}
        \caption{Input image}
        \label{fig:Input starfish image for noise level 12.75}
    \end{subfigure}
    \hfill
    \begin{subfigure}{0.18\textwidth}
        \centering
        \fbox{\includegraphics[width=\linewidth]{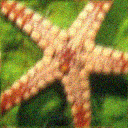}}
        \caption{Blurred image}
        \label{fig:Blurred starfish image for noise level 12.75}
    \end{subfigure}
    \hfill
    \begin{subfigure}{0.18\textwidth}
        \centering
        \fbox{\includegraphics[width=\linewidth]{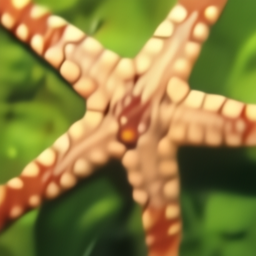}}
        \caption{Extrapolated PnP-DYS}
        \label{fig:Extrapolated_PnP starfish image for noise level 12.75}
    \end{subfigure}
    
    \vspace{0.1cm}

    \begin{subfigure}{0.18\textwidth}
        \centering
        \fbox{\includegraphics[width=\linewidth]{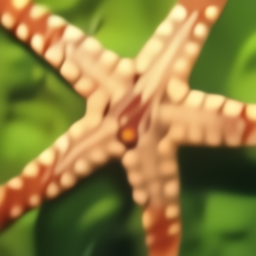}}
        \caption{PnP-DRFDR}
        \label{fig:PnP-DRFDR starfish image for noise level 12.75}
    \end{subfigure}
    \hfill
    \begin{subfigure}{0.18\textwidth}
        \centering
        \fbox{\includegraphics[width=\linewidth]{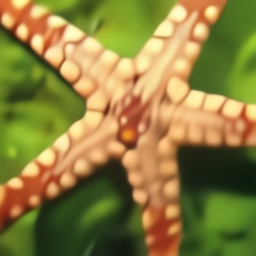}}
        \caption{FISTA-type PnP-DYS}
        \label{fig:FISTA_PnP_DYS starfish image for noise level 12.75}
    \end{subfigure}
    \hfill
    \begin{subfigure}{0.18\textwidth}
        \centering
        \fbox{\includegraphics[width=\linewidth]{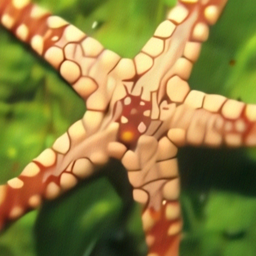}}
        \caption{FISTA-type PnP-DYS with Line-search}
        \label{fig:FISTA-type PnP-DYS with Line-search starfish image for noise level 12.75}
    \end{subfigure}
    \caption{Images of Starfish in noise level $12.75$}
    \label{fig:Starfish images for 12.75}
\end{figure}

\begin{figure}[H]
    \centering
    \begin{subfigure}{0.34\textwidth}
        \includegraphics[width=\linewidth]{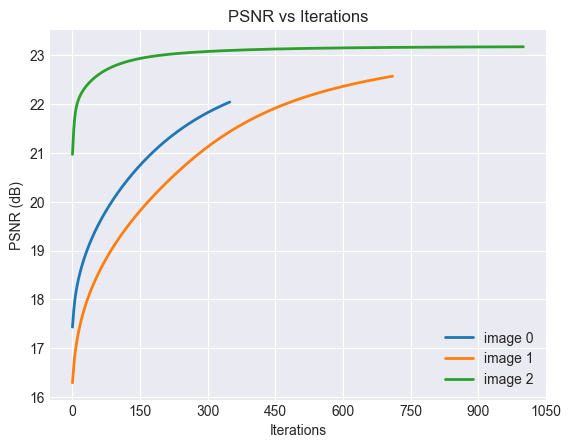}
        \caption{Extrapolated PnP-DYS}
        \label{fig:Extrapolated PnP PSNR for noise level 12.75}
    \end{subfigure}
    \hfill
    \begin{subfigure}{0.34\textwidth}
        \includegraphics[width=\linewidth]{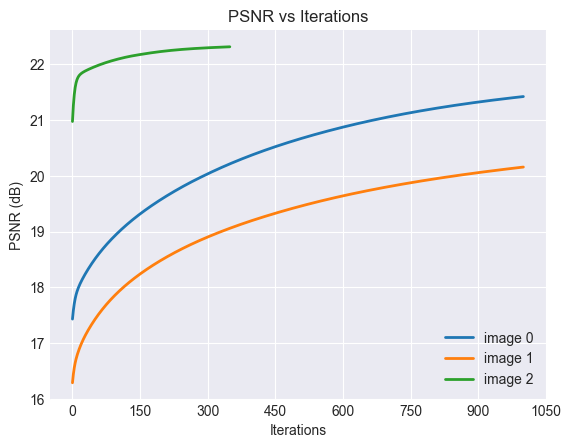}
        \caption{PnP-DRFDR}
        \label{fig:PnP-DRFDR PSNR for noise level 12.75}
    \end{subfigure}
    \hfill
    \begin{subfigure}{0.34\textwidth}
        \includegraphics[width=\linewidth]{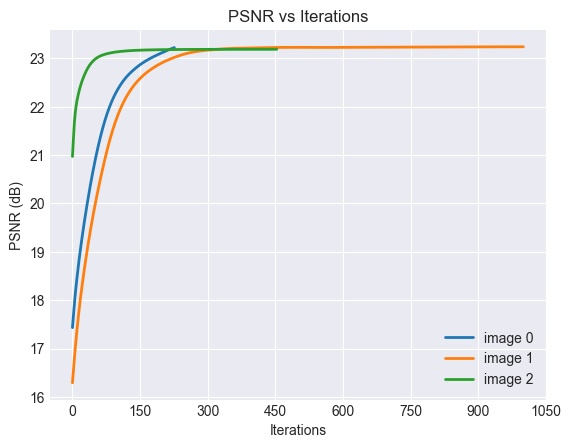}
        \caption{FISTA-type PnP-DYS}
        \label{fig:FISTA-type PnP-DYS PSNR for noise level 12.75}
    \end{subfigure}
    \hfill
    \begin{subfigure}{0.34\textwidth}
        \includegraphics[width=\linewidth]{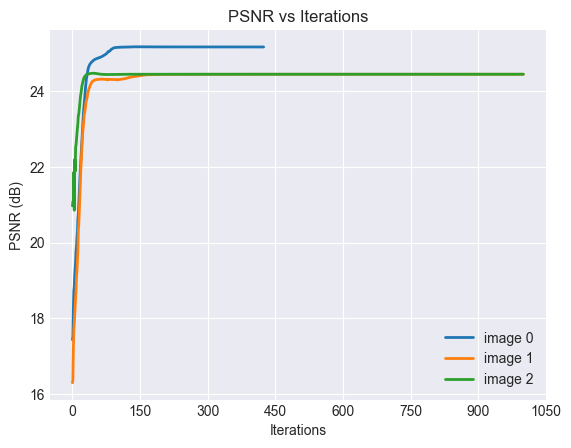}
        \caption{FISTA-type PnP-DYS with Line-search}
        \label{fig:FISTA-type PnP-DYS with Line-search PSNR for noise level 12.75}
    \end{subfigure}
    \caption{Iterations vs PSNR value graphs for noise level $12.75$}
    \label{fig:PSNR graphs for noise level 12.75}
\end{figure}

\begin{table}[H]
\centering

\resizebox{1.0\linewidth}{!}{%
\begin{tabular}{c ccc ccc ccc}
\toprule
\textbf{Noise level} & \multicolumn{3}{c}{2.55} & \multicolumn{3}{c}{7.65} & \multicolumn{3}{c}{12.75}\\
\midrule
\textbf{Algorithms} & Butterfly & Leaves & Starfish & Butterfly & Leaves & Starfish & Butterfly & Leaves & Starfish\\
\midrule
Extrapolated PnP-DYS & 26.93 & 26.97 & 27.21 & 25.27 & 25.25 & 24.63 & 22.07 & 22.60 & 23.21\\
\midrule
PnP-DRFDR & 25.58 & 26.05 & 25.19 & 23.21 & 22.94 & 23.66 & 21.45 & 20.19 & 22.36\\
\midrule
\textbf{FISTA-type PnP-DYS} & \textbf{27.31} & \textbf{27.78} & \textbf{27.29} & \textbf{25.50} & \textbf{25.46} & \textbf{24.81} & \textbf{23.24} & \textbf{23.26} & \textbf{23.22}\\
\midrule
\makecell{\textbf{FISTA type PnP-DYS}\\\textbf{with Line-search}} & \textbf{28.42} & \textbf{27.86} & \textbf{29.06} & \textbf{26.37} & \textbf{26.30} & \textbf{26.39} & \textbf{25.16} & \textbf{24.44} & \textbf{24.45}\\
\bottomrule
\end{tabular}
}

\caption{PSNR values (in dB) of the deblurred images obtained by different algorithms under different noise levels}
\label{table:Image restoration problem}
\end{table}


\section*{Acknowledgments}
The first author gratefully acknowledges the financial support provided by the University Grants Commission (UGC), Government of India, through the Junior Research Fellowship (UGC-JRF) (Student ID: 231620047584). The authors also acknowledge the Mahatma Gandhi Central Library (MGCL), Indian Institute of Technology Roorkee, for providing access to valuable library resources that supported this work.
\section{Conclusion}\label{Conclusion}


%
%

\bibliographystyle{spmpsci}      
\bibliography{reference}   

%
%



\end{document}